\documentclass[10pt]{amsart}
      \usepackage[mathscr]{eucal}
      \usepackage{amsmath,amsfonts}
\def\mapright#1{\smash{
\mathop{\rg}\limits^{#1}}}
\def\mapdown#1{\bigg\downarrow
\rlap{$\vcenter{\hbox{$\scriptstyle#1$}}$}}

\def\rg{\hbox to 30pt{\rightarrowfill}}
\def\lg{\hbox to 30pt{\leftarrowfill}}

      \parskip\smallskipamount
          \newtheorem{theorem}{Theorem}[section]
      
      \newtheorem{proposition}[theorem]{Proposition}
      \newtheorem{corollary}[theorem]{Corollary}
      \newtheorem{lemma}[theorem]{Lemma}

      \newtheorem{remark}[theorem]{Remark}
      \makeatletter
      \@addtoreset{equation}{section}
      \makeatother

      \newcommand{\BB}{{\mathbb B}}
      \newcommand{\CC}{{\mathbb C}}
      \newcommand{\NN}{{\mathbb N}}

      \newcommand{\DD}{{\mathbb D}}
      \newcommand{\RR}{{\mathbb R}}
      \newcommand{\FF}{{\mathbb F}}

      \newcommand{\cA}{{\mathcal A}}

      \newcommand{\cD}{{\mathcal D}}
      \newcommand{\cE}{{\mathcal E}}
      
      \newcommand{\cG}{{\mathcal G}}
      \newcommand{\cH}{{\mathcal H}}
      \newcommand{\cK}{{\mathcal K}}
      
      \newcommand{\cM}{{\mathcal M}}
      \newcommand{\cN}{{\mathcal N}}
      
      \newcommand{\cR}{{\mathcal R}}

      \newcommand{\cX}{{\mathcal X}}

      \newcommand{\rank}{\hbox{\rm{rank}}\,}

      \newdimen\expt
      \def\boxit#1{\setbox0\hbox{$\displaystyle{#1}$}
            \hbox{\lower.4\expt
       \hbox{\lower3\expt\hbox{\lower\dp0
            \hbox{\vbox{\hrule height.4\expt
       \hbox{\vrule width.4\expt\hskip3\expt
            \vbox{\vskip3\expt\box0\vskip2\expt}%
       \hskip3\expt\vrule width.4\expt}\hrule height.4\expt}}}}}}
      
\begin{document}
       \pagestyle{myheadings}
      \markboth{ Gelu Popescu}{  Free holomorphic automorphisms of the unit ball of $B(\cH)^n$   }

      \title [      Free holomorphic automorphisms of the unit ball of $B(\cH)^n$ ]
      {            Free holomorphic automorphisms of the unit ball of $B(\cH)^n$
      }
        \author{Gelu Popescu}
\date{March 8, 2008}
      \thanks{Research supported in part by an NSF grant}

      \address{Department of Mathematics, The University of Texas
      at San Antonio \\ San Antonio, TX 78249, USA}
      \email{\tt gelu.popescu@utsa.edu}

\begin{abstract}
In this paper we continue the study of free holomorphic functions on
the  noncommutative ball
$$
[B(\cH)^n]_1:=\left\{ (X_1,\ldots, X_n)\in B(\cH)^n: \
\|X_1X_1^*+\cdots + X_nX_n^*\|^{1/2}<1\right\},
$$
where $B(\cH)$ is the algebra of all bounded linear operators on a
Hilbert space $\cH$, and $n=1,2,\ldots$ or $n=\infty$.  These are
noncommutative multivariable analogues of the analytic functions on
the open unit disc $\DD:=\{z \in \CC:\ |z|<1\}$.

The theory of  characteristic functions for row contractions
(elements in $[B(\cH)^n]_1^-$) is used to determine the group
$Aut(B(\cH)^n_1)$ of all free holomorphic automorphisms of
$[B(\cH)^n]_1$. It is shown that $Aut(B(\cH)^n_1)\simeq Aut(\BB_n)$,
the Moebius group of the open unit ball $\BB_n:=\{\lambda\in \CC^n:\
\|\lambda\|_2<1\}$.

We show that the noncommutative Poisson transform commutes with the
action of the  automorphism group $Aut(B(\cH)^n_1)$. This leads to a
characterization
 of the unitarily implemented automorphisms  of the
 Cuntz-Toeplitz algebra $C^*(S_1,\ldots, S_n)$,  which leave invariant
 the noncommutative disc algebra $\cA_n$. This result provides new insight
  into   Voiculescu's group of automorphisms   of the Cuntz-Toeplitz
 algebra and  reveals new  connections  with    noncommutative  multivariable
 operator theory, especially, the theory of characteristic functions  for row contractions
 and  the noncommutative Poisson
 transforms.

We  show that   the unitarily implemented automorphisms of the
noncommutative  disc algebra $\cA_n$   and the noncommutative
analytic Toeplitz algebra $F_n^\infty$, respectively, are determined
by the free holomorphic automorphisms of $[B(\cH)^n]_1$, via the
noncommutative Poisson transform. Moreover, we prove that
$$
 Aut(B(\cH)^n_1)\simeq  Aut_u(\cA_n)\simeq Aut_u(F_n^\infty).
 $$
We also prove that any completely isometric automorphism of the
noncommutative disc algebra
  $A(B(\cH)^n_1)$  has the form
 $$
 \Phi(f)=f\circ \Psi,\qquad f\in A(B(\cH)^n_1),
 $$
 where   $\Psi\in
 Aut(B(\cH)^n_1)$.
We deduce a similar result for the noncommutative Hardy algebra
$H^\infty(B(\cH)^n_1)$, which, due to Davidson-Pitts results on the
automorphisms of $F_n^\infty$, implies that $Aut(B(\cH)^n_1)$ is
isomorphic to the group of contractive automorphisms of
$H^\infty(B(\cH)^n_1)$.

We study   the isometric dilations and the characteristic functions
of row contractions under the action of the automorphism group
$Aut(B(\cH)^n_1)$.
  This enables us to  obtain some results concerning the behavior of  the curvature  and the
Euler characteristic of  a  row contraction under $Aut(B(\cH)^n_1)$.

Finally, in the last section,  we  prove   a   maximum principle for
free holomorphic functions on the noncommutative ball $[B(\cH)^n]_1$
and provide some extensions of the classical Schwarz lemma to  our
noncommutative setting.

\end{abstract}

      \maketitle

\section*{Contents}
{\it

\quad Introduction

\begin{enumerate}
   \item[1.]   Free holomorphic   functions on $[B(\cH)^n]_1$ and
   Cartan type results
 \item[2.]  The group  of  free  holomorphic   automorphisms of  $[B(\cH)^n]_1$
\item[ 3.]  The noncommutative Poisson transform under the action of
$Aut(B(\cH)^n_1)$
\item[ 4.]  Model theory   for row contractions  and  the automorphism group
$Aut(B(\cH)^n_1)$
 \item[ 5.]   Maximum principle and Schwartz type results   for free  holomorphic   functions
   \end{enumerate}

\quad References

}

\bigskip

\bigskip

\section*{Introduction}

In \cite{Po-holomorphic}, \cite{Po-free-hol-interp}, and
\cite{Po-pluriharmonic} we developed a theory of  free holomorphic
(resp. pluriharmonic) functions and provide a framework for the
study of arbitrary
 $n$-tuples of operators on a Hilbert space $\cH$. Several classical
 results from complex analysis have
 free analogues in
 this  noncommutative multivariable setting.
We introduced    a notion of {\it radius of convergence} for formal
power series in $n$ noncommuting indeterminates $Z_1,\ldots, Z_n$
and proved noncommutative multivariable analogues of Abel theorem
and Hadamard formula from complex analysis (\cite{Co}, \cite{R}).
This enabled us to define the algebra $Hol(B(\cX)^n_\gamma)$~ of
free holomorphic functions on the open operatorial  $n$-ball of
radius $\gamma>0$, as the set of all power series $\sum_{\alpha\in
\FF_n^+}a_\alpha Z_\alpha$ with radius of convergence $\geq \gamma$,
i.e.,
 $\{a_\alpha\}_{\alpha\in \FF_n^+}$ are complex numbers  with
$$\limsup\limits_{k\to\infty} \left(\sum\limits_{|\alpha|=k}
|a_\alpha|^2\right)^{1/2k}\leq \frac{1}{\gamma}, $$
 where $\FF_n^+$ is the
free semigroup with $n$ generators.   The algebra of free
holomorphic functions ~$Hol(B(\cX)^n_\gamma)$~ has the following
universal property.

{\it Any   representation  $\pi: \CC[Z_1,\ldots, Z_n]\to B(\cH)$
with $\|[\pi(Z_1),\ldots, \pi(Z_n)]\|^{1/2}<\gamma$  extends
uniquely to a representation of $Hol(B(\cX)^n_\gamma)$.}

For simplicity, throughout this paper, $[X_1,\ldots, X_n]$ denotes
either the $n$-tuple $(X_1,\ldots, X_n)\in B(\cH)^n$ or the operator
row matrix $[ X_1\, \cdots \, X_n]$ acting from $\cH^{(n)}$, the
direct sum  of $n$ copies of  a Hilbert space $\cH$, to $\cH$.
  A
free holomorphic function on the open operatorial ball of radius
$\gamma$,
$$
[B(\cH)^n]_\gamma:=\left\{ [X_1,\ldots, X_n]\in B(\cH)^n: \
\|X_1X_n^*+\cdots + X_nX_n^*\|^{1/2}<\gamma\right\},
$$
 is
the representation of an element $F\in Hol(B(\cX)^n_\gamma)$ on the
Hilbert space $\cH$, that is,  the mapping
$$
[B(\cH)^n]_\gamma\ni (X_1,\ldots, X_n)\mapsto F(X_1,\ldots,
X_n)=\sum_{k=0}^\infty \sum_{|\alpha|=k}
 a_\alpha X_\alpha\in
B(\cH),
$$
where  the convergence is in the operator norm topology.
  We mention that the results from \cite{Po-holomorphic}, \cite{Po-free-hol-interp},
and \cite{Po-pluriharmonic}  hold true in the context of free
holomorphic (resp. pluriharmonic)  functions with operator-valued
coefficients. Due to the fact that a free holomorphic function is
uniquely determined by its representation on an infinite dimensional
Hilbert space, we assume, throughout this paper, that  $\cH$ is a
separable infinite dimensional Hilbert space.

 We prove in \cite{Po-holomorphic} that   the Hausdorff derivations (\cite{MKS})
 $\frac{\partial}{\partial
Z_i}$, \, $i=1,\ldots, n$, on the algebra of noncommutative
polynomials $\CC[Z_1,\ldots, Z_n]$  can be  extended to the algebra
of free holomorphic functions.
  Let $F_1,\ldots, F_n$ be free
  holomorphic functions  on $[B(\cH)^n]_\gamma$  with
 scalar coefficients. Then the map $F:[B(\cH)^n]_\gamma\to B(\cH)^n$ defined by $F:=(F_1,\ldots,
 F_n)$ is a free holomorphic function.  We define $F'(0)$ as the
 linear operator on $\CC^n$ having the matrix $\left[\left( \frac{\partial F_i}
{\partial Z_j}\right)(0)\right]_{i,j=1,\ldots, n}$.

 In Section 1, we
show that, under natural conditions,  the composition of free
holomorphic functions is a free holomorphic function.  We obtain
  a noncommutative version of Cartan's uniqueness theorem (see \cite{Ca}), which
  states that  if $F:[B(\cH)^n]_\gamma\to [B(\cH)^n]_\gamma$ is  a free
 holomorphic function such that $F(0)=0$ and $F'(0)=I_n$,
 then
 $$F(X_1,\ldots, X_n)=(X_1,\ldots, X_n),\qquad  (X_1,\ldots,
 X_n)\in[B(\cH)^n]_\gamma.
 $$
This is used to characterize the free biholomorphic functions
$F:[B(\cH)^n]_{\gamma_1}\to [B(\cH)^n]_{\gamma_2}$ with $F(0)=0$. As
a consequence, we show that any free holomorphic automorphism $\Psi$
of the unit ball $[B(\cH)^n]_1$ which fixes the origin is
implemented by a unitary operator on $\CC^n$, i.e., there is a
unitary operator $U$ on $\CC^n$  such that
\begin{equation*}
 \Psi(X_1,\ldots X_n)= \Phi_U(X_1,\ldots
X_n):=[X_1,\ldots, X_n]U , \qquad (X_1,\ldots X_n)\in  [B(\cH)^n]_1.
\end{equation*}

In Section 2, we use the  theory of noncommutative characteristic
functions for row contractions (see \cite{Po-charact},
\cite{Po-varieties}) to find all the involutive free holomorphic
automorphisms of $[B(\cH)^n]_1$, which turn out to be of the form
\begin{equation*}
 \Psi_\lambda=- \Theta_\lambda(X_1,\ldots, X_n):={
\lambda}-\Delta_{ \lambda}\left(I_\cK-\sum_{i=1}^n \bar{{
\lambda}}_i X_i\right)^{-1} [X_1,\ldots, X_n] \Delta_{{\lambda}^*},
\end{equation*}
for some $\lambda=[\lambda_1,\ldots, \lambda_n]\in \BB_n$, where
$\Theta_\lambda$ is the characteristic function  of the row
contraction $\lambda$, and $\Delta_{ \lambda}$,
$\Delta_{{\lambda}^*}$ are certain defect operators. Combining this
result  with the results of Section 1, we determine all the free
holomorphic automorphisms of the noncommutative ball $[B(\cH)^n]_1$.
We show that if $\Psi\in Aut(B(\cH)^n_1)$ and
$\lambda:=\Psi^{-1}(0)$, then there is a unitary operator $U$ on
$\CC^n$ such that
$$
\Psi=\Phi_U\circ \Psi_\lambda.
$$
Moreover, we prove  that the automorphism group $Aut(B(\cH)^n_1)$ is
isomorphic to $Aut(\BB_n)$, the Moebius group of  the open unit ball
$\BB_n$ (see \cite{Ru}), via the noncommutative Poisson transform.
More precisely, we show that the map
 $\Gamma :Aut(B(\cH)^n_1)\to  Aut(\BB_n)$, defined by
 $$[\Gamma(\Psi)](z):=( P_z\otimes \text{\rm id})[\hat \Psi], \qquad
 z\in \BB_n,
 $$
 is a group isomorphism, where $\hat \Psi$ is the boundary
 function  of $\Psi$ with respect to the left creation operators on the full Fock space
  and $P_z$ is the noncommutative Poisson transform at $z$.

We recall that a free holomorphic function $F$ on the open
operatorial $n$-ball of radius $1$ is bounded if
$$
\|F\|_\infty:=\sup  \|F(X_1,\ldots, X_n)\|<\infty,
$$
where the supremum is taken over all $n$-tuples  of operators
$[X_1,\ldots, X_n]\in [B(\cH)^n]_1$. Let $H^\infty(B(\cH)^n_1)$ be
the set of all bounded free holomorphic functions and  let
$A(B(\cH)^n_1)$ be the set of all  elements $F$   such that the
mapping
$$[B(\cH)^n]_1\ni (X_1,\ldots, X_n)\mapsto F(X_1,\ldots, X_n)\in B(\cH)$$
 has a continuous extension to the closed unit ball $[B(\cH)^n]^-_1$.
We  showed in \cite{Po-holomorphic} that $H^\infty(B(\cH)^n_1)$  and
$A(B(\cH)^n_1)$ are Banach algebras under pointwise multiplication
and the norm $\|\cdot \|_\infty$, which can be identified with the
noncommutative analytic Toeplitz algebra $F_n^\infty$ and the
noncommutative disc algebra $\cA_n$, respectively. We recall that
the algebra $F_n^\infty$ (resp. $\cA_n$) is the weakly (resp. norm)
closed algebra generated by the left creation  operators
$S_1,\ldots, S_n$ on the full Fock space with $n$ generators, and
the identity. These algebras have been intensively studied in recent
years (\cite{Po-charact}, \cite{Po-von}, \cite{Po-funct},
\cite{Po-analytic}, \cite{Po-poisson}, \cite{Po-curvature},
\cite{ArPo},  \cite{DP1}, \cite{DP2}, \cite{DKP},
\cite{Po-unitary}).

In Section 3, we show that the noncommutative Poisson transform
commutes with the action of $Aut(B(\cH)^n_1)$. This leads to a
characterization
 of the unitarily implemented automorphisms  of the
 Cuntz-Toeplitz algebra $C^*(S_1,\ldots, S_n)$ (see \cite{Cu}), which leave invariant
 the noncommutative disc algebra $\cA_n$. We remark that this result provides new insight
  into   Voiculescu's group of automorphisms (see \cite{Vo})  of the Cuntz-Toeplitz
 algebra.
More precisely, we show that if $\Psi\in Aut([B(\cH)^n]_1)$ and
$\hat\Psi$ is its boundary function,
 then the noncommutative Poisson
 transform $P_{\hat\Psi}$   is a  unitarily implemented automorphism
 of the Cuntz-Toeplitz algebra, which leaves invariant
 the noncommutative disc algebra $\cA_n$. These  are precisely
 Voiculescu's automorphisms  considered in \cite{Vo}.
 Conversely, we prove that  if $\Phi\in Aut_u(C^*(S_1,\ldots, S_n))$ and
 $\Phi(\cA_n)\subset \cA_n$, then there is  $\Psi\in
 Aut([B(\cH)^n]_1)$ such that $\Phi=P_{\hat\Psi}$.
Moreover, in this case
 $$\Phi(g)=K_{\hat \Psi}^* ( I_{\cD_{ {\hat\Psi}}}\otimes g) K_{\hat \Psi},
  \qquad g\in C^*(S_1,\ldots,
 S_n),
 $$
  where the noncommutative Poisson kernel $K_{\hat \Psi}$ is a
  unitary operator.

In \cite{DP2}, Davidson and Pitts showed that the automorphisms of
the analytic Toeplitz algebra $F_n^\infty$ are norm and WOT
continuous. They proved that there is a natural homomorphism from
$Aut(F_n^\infty)$ onto $Aut(\BB_n)$, the group of conformal
automorphisms of the unit ball $\BB_n$. Moreover, using Voiculescu's
group of automorphisms of the Cuntz-Toeplitz algebra,
 they showed that the subgroup $Aut_u(F_n^\infty)$ of  unitarily implemented
automorphisms of $F_n^\infty$  is isomorphic with $Aut(\BB_n)$. In
Section 3, we obtain a new proof of their result, using
 noncommutative Poisson transforms.
We show that   the unitarily implemented automorphisms of the
noncommutative  disc algebra $\cA_n$   and the noncommutative
analytic Toeplitz algebra $F_n^\infty$, respectively, are determined
by the free holomorphic automorphisms of $[B(\cH)^n]_1$, via the
noncommutative Poisson transform. Moreover, we deduce that
$$
 Aut(B(\cH)^n_1)\simeq  Aut_u(\cA_n)\simeq Aut_u(F_n^\infty).
 $$

 In Section 3, we also prove that   $\Phi:A(B(\cH)^n_1)\to
A(B(\cH)^n_1)$ is a completely isometric automorphism of the
noncommutative disc algebra
  $A(B(\cH)^n_1)$ if and only if
 there is a   $\Psi\in
 Aut(B(\cH)^n_1)$ such that
 $$
 \Phi(f)=f\circ \Psi,\qquad f\in A(B(\cH)^n_1),
 $$
 which extends the  characterization of  the conformal automorphisms
  of the disc
 algebra (see \cite{H}).
We deduce a similar result for the noncommutative Hardy algebra
$H^\infty(B(\cH)^n_1)$ which, due to Davidson-Pitts results on the
automorphisms of $F_n^\infty$ (see \cite{DP2}),  implies that
$Aut(B(\cH)^n_1)$ is isomorphic to the group of contractive
automorphisms of $H^\infty(B(\cH)^n_1)$. We remark  here that the
conformal automorphisms of $\BB_n$ also occur in the   work  of
Muhly and Solel  (\cite{MuSo3})
 concerning the automorphisms of Hardy algebras associated with
 $W^*$-correspondence over von Neumann algebras (\cite{MuSo1},
 \cite{MuSo2}), and  the work  of Power and Solel (\cite{PS}) in a
 related context.

In Section 4, we deal with the dilation and model theory of row
contractions (\cite{F}, \cite{Bu}, \cite{Po-isometric},
\cite{Po-charact}) under the action of the the free holomorphic
automorphisms of $[B(\cH)^n]_1$.  We show that if $T$ is a row
contraction  and $\Psi \in Aut(B(\cH)^n_1)$,
  then  the    characteristic function   $\Theta_{\Psi(T)}$
  {\it coincides} with  $\Theta_T\circ\Psi^{-1}$.
  This enables us to obtain   some results concerning the behavior of
  the curvature  and the
Euler characteristic of  a  row contraction (see
\cite{Po-curvature}, \cite{Kr}) under the automorphism group
$Aut(B(\cH)^n_1)$. In particular, when $T:=[T_1,\ldots, T_n]$ is a
commutative  row contraction  with $\rank \Delta_T<\infty$, we show
that  Arveson's curvature \cite{Arv2} satisfies the equation
$$K(\Psi(T))= \int_{\partial \BB_n}\lim_{r\to 1}\text{\rm trace}\,
 [I_{\cD_T}-\Theta_{T}(\Psi^{-1}(r\xi))\Theta_{T}(\Psi^{-1}(r\xi))^*] d\sigma(\xi)
$$
for any  $\Psi\in Aut(B(\cH)^n_1)$, where  the constrained
characteristic function  (also denoted by $\Theta_T$)  is given by
$$
\Theta_{T}(z):= -T+\Delta_T(I-z_1T_1^*-\cdots -z_nT_n^*)^{-1}
[z_1I_\cH,\ldots, z_nI_\cH]\Delta_{T^*},\quad z\in \BB_n.
$$
It will be interesting to know if  $ \text{\rm curv}\,(T)=\text{\rm
curv}\,(\Psi(T))$  for any pure row contraction $T$ and any free
holomorphic  automorphism of $[B(\cH)^n]_1$, where  $ \text{\rm
curv}\,(T)$ donotes the curvature of an arbitrary  row contraction
$T$. The answer is positive if $n=1$  and also when $n\geq 2$ and
$T$ is a pure row isometry. We mention that Benhida and Timotin
\cite{BeTi2} used Redheffer products to study the behavior of the
characteristic function of a row contraction under Voiculescu's
group of automorphisms \cite{Vo}.

In Section 5, we  prove  the  following  maximum principle for free
holomorphic functions on the noncommutative ball $[B(\cH)^n]_1$.  If
$F:[B(\cH)^n]_1\to B(\cH)$ is a free holomorphic function  and there
exists $X_0\in [B(\cH)^n]_1$ such that
$$\|F(X_0)\|\geq \|F(X)\| \quad \text{ for all }\ X\in [B(\cH)^n]_1,
$$
 then $F$ must be a constant.

The classical Schwarz's lemma (see \cite{Co}, \cite{R})  states that
if $f:\DD\to \CC$ is a bounded analytic function with $f(0)=0$ and
$|f(z)|\leq 1$ for $z\in \DD$, then $|f'(0)|\leq 1$ and $|f(z)|\leq
|z|$ for $z\in \DD$. Moreover, if $|f'(0)|= 1$ or if $|f(z)|=|z|$
for some $z\neq 0$, then there is a constant $c$ with $|c|=1$ such
that $f(w)=cw$ for any $w\in \DD$.
  We proved in \cite{Po-holomorphic} that if  $F$  is  a
free holomorphic function on $[B(\cH)^n]_1$ with $F(0)=0$, then
$\|F(X)\|\leq \|X\|$ for any $X\in [B(\cH)^n]_1$.

In Section 5, we complete this noncommutative version of Schwarz
lemma, by adding  new results concerning the uniqueness.  If
$F:[B(\cH)^n]_{1}\to [B(\cH)^m]_{1}$ is a free holomorphic function,
we show that
$$\varphi(X_1,\ldots, X_n)= [X_1,\ldots, X_n] F'(0)^t, \quad (X_1,\ldots,
X_n)\in [B(\cH)^n]_{1},
$$
 maps $[B(\cH)^n]_{1}$ into
$[B(\cH)^m]_{1}$, where ${}^t$ denotes the transpose. In particular,
$\|F'(0)\|\leq 1$. Moreover,  we show that if
 $F:[B(\cH)^n]_{1}\to [B(\cH)^n]_{1}$ is  a free holomorphic
function such that $F'(0)$ is a unitary operator on $\CC^n$, then
$F$ is a  free holomorphic automorphism  of $[B(\cH)^n]_{1}$ and
$F(X)=X[F'(0)]^t$ for  $X\in [B(\cH)^n]_{1}$.

 Using the free holomorphic automorphisms of $[B(\cH)^n]_1$,
we obtain another  extension of Schwarz  lemma for bounded free
holomorphic functions, which states that if
  $F:[B(\cH)^n]_1\to [B(\cH)^m]_1$ is a free
holomorphic function, $a\in \BB_n$, and $b:=F(a)\in \BB_m$, then
$$
\left\|\Psi_b(F(X))\right\|\leq \|\Psi_a(X)\|,\quad X\in
[B(\cH)^n]_1,
$$
where $\Psi_a$ and $\Psi_b$ are the corresponding free holomorphic
automorphisms.

We remark that all the results of this paper are valid  even when
$n=\infty$. We also mention that many  of the above-mentioned
results have commutative counterparts.
 We defer this discussion for a future paper, where we determine
and study the group of all commutative free holomorphic
automorphisms of the commutative open ball
$$
[B(\cH)^n]_{1,c}:=\left\{ (X_1,\ldots, X_n)\in [B(\cH)^n]_1:\ X_i
X_j=X_j X_i, \quad i,j=1,\ldots,n \right\}.
$$
Finally, we should emphasize that the present paper makes new
connections between noncommutative  multivariable operator theory
(\cite{Po-isometric}, \cite{Po-charact}, \cite{Po-poisson},
\cite{Po-holomorphic}, \cite{Po-unitary}), the classical theory of
analytic functions (\cite{R}, \cite{Co}, \cite{H}, \cite{Ru}), and
Voiculescu's group of automorphisms \cite{Vo} of the Cuntz-Toeplitz
algebra \cite{Cu}.

 \bigskip

\section{  Free holomorphic   functions on $[B(\cH)^n]_1$  and Cartan type results}

In this section we show that, under  natural conditions, the
composition of free holomorphic functions is a free holomorphic
function. We obtain a noncommutative version of Cartan's uniqueness
theorem, and use it to characterize the free biholomorphic functions
$F$ with $F(0)=0$.  As a consequence, we show that any free
holomorphic automorphism   of the noncommutative ball $[B(\cH)^n]_1$
with fixes the origin is implemented by a unitary operator on
$\CC^n$.

Let $H_n$ be an $n$-dimensional complex  Hilbert space with
orthonormal
      basis
      $e_1$, $e_2$, $\dots,e_n$, where $n=1,2,\dots$, or $n=\infty$.
       We consider the full Fock space  of $H_n$ defined by
      $$F^2(H_n):=\CC1\oplus \bigoplus_{k\geq 1} H_n^{\otimes k},$$
      where  $H_n^{\otimes k}$ is the (Hilbert)
      tensor product of $k$ copies of $H_n$.
      Define the left  (resp.~right) creation
      operators  $S_i$ (resp.~$R_i$), $i=1,\ldots,n$, acting on $F^2(H_n)$  by
      setting
      $$
       S_i\varphi:=e_i\otimes\varphi, \quad  \varphi\in F^2(H_n),
      $$
       (resp.~$
       R_i\varphi:=\varphi\otimes e_i, \quad  \varphi\in F^2(H_n).
      $)
The noncommutative disc algebra $\cA_n$ (resp.~$\cR_n$) is the norm
closed algebra generated by the left (resp.~right) creation
operators and the identity. The   noncommutative analytic Toeplitz
algebra $F_n^\infty$ (resp.~$\cR_n^\infty$)
 is the  weakly
closed version of $\cA_n$ (resp.~$\cR_n$). These algebras were
introduced in \cite{Po-von} in connection with a noncommutative von
Neumann  type inequality \cite{von}.

 Let $\FF_n^+$ be the unital free semigroup on $n$ generators
$g_1,\ldots, g_n$ and the identity $g_0$.  The length of $\alpha\in
\FF_n^+$ is defined by $|\alpha|:=0$ if $\alpha=g_0$ and
$|\alpha|:=k$ if
 $\alpha=g_{i_1}\cdots g_{i_k}$, where $i_1,\ldots, i_k\in \{1,\ldots, n\}$.
If $(X_1,\ldots, X_n)\in B(\cH)^n$, where $B(\cH)$ is the algebra of
all bounded linear operators on the Hilbert space $\cH$,    we set
$X_\alpha:= X_{i_1}\cdots X_{i_k}$  and $X_{g_0}:=I_\cH$. We denote
$e_\alpha:= e_{i_1}\otimes\cdots \otimes  e_{i_k}$ and $e_{g_0}:=1$.
Note that $\{e_\alpha\}_{\alpha\in \FF_n^+}$ is an orthonormal basis
for $F^2(H_n)$.

  We recall   (\cite{Po-charact},
       \cite{Po-von},  \cite{Po-funct},
      \cite{Po-analytic})
       a few facts
       concerning multi-analytic   operators on Fock
      spaces.
         We say that
       a bounded linear
        operator
      $M$ acting from $F^2(H_n)\otimes \cK$ to $ F^2(H_n)\otimes \cK'$ is
       multi-analytic with respect to $S_1,\ldots, S_n$
      if
      \begin{equation*}
      M(S_i\otimes I_\cK)= (S_i\otimes I_{\cK'}) M\quad
      \text{\rm for any }\ i=1,\dots, n.
      \end{equation*}
       We can associate with $M$ a unique formal Fourier expansion
      \begin{equation*}       M(R_1,\ldots, R_n):= \sum_{\alpha \in \FF_n^+}
      R_\alpha \otimes \theta_{(\alpha)}, \end{equation*}
where $\theta_{(\alpha)}\in B(\cK, \cK')$.
       We  know  that
        $$M =\text{\rm SOT-}\lim_{r\to 1}\sum_{k=0}^\infty
      \sum_{|\alpha|=k}
         r^{|\alpha|} R_\alpha\otimes \theta_{(\alpha)},
         $$
         where, for each $r\in [0,1)$, the series converges in the uniform norm.
      Moreover, the set of  all multi-analytic operators in
      $B(F^2(H_n)\otimes \cK,
      F^2(H_n)\otimes \cK')$  coincides  with
      $R_n^\infty\bar \otimes B(\cK,\cK')$,
      the WOT-closed operator space generated by the spatial tensor
      product.
A multi-analytic operator is called inner if it is an isometry. We
remark that  similar results are valid  for  multi-analytic
operators with respect to the right creation operators $R_1,\ldots,
R_n$.

 According to \cite{Po-holomorphic}, a map $F:[B(\cH)^n]_{\gamma}\to
  B( \cH)\bar\otimes_{min} B(\cE, \cG)$ is
  {\it free
holomorphic function} on  $[B(\cH)^n]_{\gamma}$  with coefficients
in $B(\cE, \cG)$ if there exist $A_{(\alpha)}\in B(\cE, \cG)$,
$\alpha\in \FF_n^+$, such that
$$
F(X_1,\ldots, X_n)=\sum\limits_{k=0}^\infty \sum\limits_{|\alpha|=k}
X_\alpha\otimes  A_{(\alpha)},
$$
where the series converges in the operator  norm topology  for any
$(X_1,\ldots, X_n)\in [B(\cH)^n]_{\gamma}$. A power series
$F:=\sum\limits_{\alpha\in \FF_n^+}
 Z_\alpha\otimes A_{(\alpha)}$ represents a free holomorphic function
on the open  operatorial $n$-ball of radius $\gamma$, with
coefficients in $B(\cE, \cG)$,      if and only if the series
$$
\sum\limits_{k=0}^\infty  \sum\limits_{|\alpha|=k} r^{|\alpha|}
 S_\alpha\otimes A_{(\alpha)}
$$
is convergent in the operator norm topology for any $r\in
[0,\gamma)$, where $S_1,\ldots, S_n$ are the left creation operators
on the Fock space $F^2(H_n)$.
 Moreover, in this case, we have
\begin{equation*}
\sum\limits_{k=0}^\infty\left\| \sum\limits_{|\alpha|=k}
r^{|\alpha|}   S_\alpha\otimes A_{(\alpha)}
\right\|=\sum_{k=0}^\infty r^k\left\| \sum_{|\alpha|=k}
A_{(\alpha)}^*A_{(\alpha)}\right\|^{1/2},
\end{equation*}
where the series are convergent for any $r\in [0,\gamma)$.

We remark  that  the coefficients  of a  free holomorphic function
are uniquely determined by its representation on
   an infinite dimensional   Hilbert space. Indeed,  let $0<r<\gamma$ and assume
   $F(rS_1,\ldots, rS_n)=0$.  Taking into account that
     $S_i^* S_j=\delta_{ij} I$ for $i,j=1,\ldots, n$, we have
\begin{equation*}
\left< F(rS_1,\ldots, rS_n)(1\otimes x), ( S_\alpha\otimes
I_\cG)(1\otimes y)\right>=r^{|\alpha|}\left<A_{(\alpha)}x,y\right>=0
\end{equation*}
for any $x\in \cE$, $y\in \cG$,  and $\alpha\in \FF_n^+$. Therefore
$A_{(\alpha)}=0$ for any $\alpha\in \FF_n^+$.

Due to this reason, throughout  this paper, we assume  that $\cH$ is
a separable infinite dimensional Hilbert space.

\begin{lemma}
\label{range} If $\varphi:[B(\cH)^n]_{\gamma_1}\to B(\cH)^n$ is a
free holomorphic function, then $\text{\rm range}\, \varphi\subseteq
[B(\cH)^n]_{\gamma_2}$ if and  only if
$$
\|\varphi(rS_1,\ldots, rS_n)\|<\gamma_2 \quad \text{ for any } \
r\in[0,\gamma_1),
$$
where $\gamma_1>0$ and $\gamma_2>0$.
\end{lemma}

\begin{proof}
Since $\cH$ is infinite dimensional, one implication is obvious.
Conversely, assume that, for $r\in(0,\gamma_1)$,
$\|\varphi(rS_1,\ldots, rS_n)\|<\gamma_2$. If $X:=(X_1,\ldots, X_n)$
is in $[B(\cH)^n]_{\gamma_1}$, then there exists $\gamma \in (0,1)$
such that $\|\frac{1}{\gamma} X\|<\gamma_1$.  Since $\varphi$ is a
free holomorphic function on $[B(\cH)^n]_{\gamma_1}$,  the map
$\varphi_\gamma$ defined by
$$\varphi_\gamma(X_1,\ldots, X_n):=\varphi(\gamma X_1,\ldots, \gamma
X_n), \qquad (X_1,\ldots, X_n)\in [B(\cH)^n]_{\gamma_1}^-,$$ is free
holomorphic on $[B(\cH)^n]_{\gamma_1}^-$. Consequently,
$\varphi_\gamma(\gamma_1 S_1,\ldots, \gamma_1 S_n)$ is in
$M_{1\times n}\otimes \cA_n$, i.e., a row matrix with entries in
$\cA_n$. Using the noncommutative von Neumann inequality
\cite{Po-von} (see \cite{von} for the classical case), we deduce
that
$$
\|\varphi (X_1,\ldots,
X_n)\|=\left\|\varphi_\gamma\left(\frac{1}{\gamma}X_1,\ldots,\frac{1}{\gamma}X_n\right)
\right\|\leq \|\varphi_\gamma(\gamma_1 S_1,\ldots, \gamma_1
S_n)\|=\|\varphi(\gamma \gamma_1S_1,\ldots, \gamma
\gamma_1S_n)\|<\gamma_2.
$$ This completes the proof.
\end{proof}

The next result shows that the composition of free holomorphic
functions is a free holomorphic function.

\begin{theorem}
\label{compo} Let  $F:[B(\cH)^n]_{\gamma_2}\to B(\cH)\bar \otimes
B(\cE,\cG)$ and $\varphi: [B(\cH)^n]_{\gamma_1}\to B(\cH)^n$ be free
holomorphic functions such that $\text{\rm range}\, \varphi
\subseteq [B(\cH)^n]_{\gamma_2}$.  Then $F\circ \varphi$ is a free
holomorphic function on $[B(\cH)^n]_{\gamma_1}$.
\end{theorem}
\begin{proof}
Let $F$ have the representation
\begin{equation}
\label{rep} F(Y_1,\ldots, Y_n)=\sum_{k=0}^\infty \sum_{|\alpha|=k}
 Y_\alpha\otimes A_{(\alpha)}, \qquad (Y_1,\ldots, Y_n)\in
[B(\cH)^n]_{\gamma_2},
\end{equation}
where the series is convergent in the operator norm topology.
Suppose  that $\varphi=(\varphi_1,\ldots, \varphi_n)$, where
$\varphi_1,\ldots, \varphi_n$ are free holomorphic functions on
$[B(\cH)^n]_{\gamma_1}$ with scalar coefficients, and  that
$\text{\rm range}\, \varphi \subseteq [B(\cH)^n]_{\gamma_2}$.
Since $\|[\varphi_1(X),\ldots, \varphi_n(X)]\|<\gamma_2$ for
$X:=(X_1,\ldots, X_n)\in [B(\cH)^n]_{\gamma_1}$, relation
\eqref{rep} implies
$$
(F\circ \varphi)(X_1,\ldots, X_n)=\sum_{k=0}^\infty
\sum_{|\alpha|=k}  \varphi_\alpha(X_1,\ldots, X_n)\otimes
A_{(\alpha)},
$$
where   $\varphi_\alpha:=\varphi_{i_1}\cdots \varphi_{i_k}$ if
$\alpha=g_{i_1}\cdots g_{i_k}\in \FF_n^+$ and
 the series converges in the operator norm topology for any
$(X_1,\ldots, X_n)\in [B(\cH)^n]_{\gamma_1}$. According to Lemma
\ref{range}, we have $\|\varphi(rS_1,\ldots, rS_n)\|<\gamma_2$,
$r\in [0,\gamma_1)$. Since $F$ is free holomorphic on
$[B(\cH)^n]_{\gamma_2}$, for each $r\in [0,\gamma_1)$,
\begin{equation}
\label{Mr}
 M_r:=\sum_{k=0}^\infty \sum_{|\alpha|=k}
 \varphi_\alpha(rS_1,\ldots, rS_n)\otimes A_{(\alpha)}
\end{equation}
is convergent in the operator norm topology. Due to the fact that
$\varphi_i(rS_1,\ldots, rS_n)$ is in the noncommutative disc algebra
$\cA_n$ for each $i=1,\ldots, n$, it is clear that $M_r$ is in the
operator space $\cA_n\otimes B(\cE, \cG)\subset F_n^\infty\bar
\otimes B(\cE, \cG)$. Therefore, for each $r\in [0,\gamma_1)$, the
operator $M_r$ has a unique ``Fourier representation''
$\sum_{k=0}^\infty \sum_{|\alpha|=k} r^{|\alpha|} S_\alpha\otimes
B_{(\alpha)}(r)$ and
\begin{equation}
\label{rep2} M_r=\text{\rm SOT-}\lim_{\gamma\to 1}\sum_{k=0}^\infty
\sum_{|\alpha|=k} r^{|\alpha|}\gamma^{|\alpha|} S_\alpha\otimes
B_{(\alpha)}(r),
\end{equation}
where the series $\sum_{k=0}^\infty \sum_{|\alpha|=k}
 r^{|\alpha|}\gamma^{|\alpha|} S_\alpha\otimes B_{(\alpha)}(r)$
converges in the operator norm topology.

The next step is to show that the coefficients $B_{(\alpha)}(r)\in
B(\cE, \cG)$, $\alpha\in \FF_n^+$,  do not depend on $r\in
[0,\gamma_1)$. Using relations \eqref{Mr} and \eqref{rep2}, we
deduce that
\begin{equation*}
\begin{split}
\left< B_{\alpha}(r)x,y\right>&= \left<( S_\alpha^*\otimes
I)\frac{1}{r^{|\alpha|}} M_r(1\otimes x), 1\otimes y\right>\\
&=\lim_{p\to\infty}\sum_{k=0}^p\sum_{|\beta|=k}\left< A_{(\beta)}
x,y\right>\left<\frac{1}{r^{|\alpha|}}
S_\alpha^*\varphi_\beta(rS_1,\ldots, rS_n)1,1\right>
\end{split}
\end{equation*}
for any $x\in \cE$, $y\in \cG$, and $\alpha\in \FF_n^+$. On the
other hand, for each $\beta\in \FF_n^+$, $\varphi_\beta$  is a free
holomorphic function on $[B(\cH)^n]_{\gamma_1}$  with scalar
coefficients and has a representation $\varphi_\beta(X_1,\ldots,
X_n)=\sum_{k=0}^\infty \sum_{|\alpha|=k} d_\alpha X_\alpha$  for any
$(X_1,\ldots, X_n)\in [B(\cH)^n]_{\gamma_1}$,  where $d_\alpha\in
\CC$. Consequently, we have
$$
\left<\frac{1}{r^{|\alpha|}} S_\alpha^*\varphi_\beta(rS_1,\ldots,
rS_n)1,1\right>=d_\alpha
$$
for any $r\in [0,\gamma_1)$, and $\alpha,\beta\in \FF_n^+$. Now, it
is clear that $B_{(\alpha)}:=B_{(\alpha)}(r)$ does not depend on
$r\in [0,\gamma_1)$. Going back to   relation \eqref{rep2},  we
deduce that
\begin{equation}\label{Mr2}
 M_r=\text{\rm SOT-}\lim_{\gamma\to 1}\sum_{k=0}^\infty
\sum_{|\alpha|=k}  r^{|\alpha|}\gamma^{|\alpha|} S_\alpha \otimes
B_{(\alpha)},
\end{equation}
where the series $\sum_{k=0}^\infty \sum_{|\alpha|=k}
r^{|\alpha|}\gamma^{|\alpha|} S_\alpha\otimes B_{(\alpha)}$
converges in the operator norm topology for any $r\in[0,\gamma_1)$
and $ \gamma\in [0,1)$. This shows that
$$
G(X_1,\ldots, X_n):=\sum_{k=0}^\infty \sum_{|\alpha|=k}
X_\alpha\otimes B_{(\alpha)},\qquad (X_1,\ldots, X_n)\in
[B(\cH)^n]_{\gamma_1},
$$
is a free holomorphic function on $[B(\cH)^n]_{\gamma_1}$.
Consequently, using the continuity of $G$ in the norm operator
topology and relations \eqref{Mr} and \eqref{Mr2}, we deduce that
\begin{equation}
\label{AB} M_r:=\sum_{k=0}^\infty \sum_{|\alpha|=k}
 \varphi_\alpha(rS_1,\ldots, rS_n)\otimes A_{(\alpha)}=
\sum_{k=0}^\infty \sum_{|\alpha|=k} r^{|\alpha|} S_\alpha\otimes
B_{(\alpha)}
\end{equation}
for any $r\in [0,\gamma_1)$.

Now, let  $X:=(X_1,\ldots, X_n)\in [B(\cH)^n]_{\gamma_1}$ and set
$\gamma:=\|X\|<\gamma_1$.  Applying  the noncommutative Poisson
transform (see \cite{Po-poisson}) at $(\frac{1}{\gamma} X_1,\ldots,
\frac{1}{\gamma} X_n)$ to relation \eqref{AB}, when $r=\gamma$,  we
deduce that
$$
(F\circ \varphi)(X_1,\ldots, X_n)=\sum_{k=0}^\infty
\sum_{|\alpha|=k} \varphi_\alpha(X_1,\ldots, X_n)\otimes
A_{(\alpha)}=\sum_{k=0}^\infty \sum_{|\alpha|=k}  X_\alpha\otimes
B_{(\alpha)}
$$
for any $(X_1,\ldots, X_n)\in [B(\cH)^n]_{\gamma_1}$. This completes
the proof.
\end{proof}

 For each $i=1,\ldots, n$,  we define the free  partial derivation  $\frac{\partial } {\partial Z_i}$  on $\CC[Z_1,\ldots, Z_n]$ as the unique linear operator  on this algebra, satisfying the conditions
$$
\frac{\partial I} {\partial Z_i}=0, \quad  \frac{\partial Z_i}
{\partial Z_i}=I, \quad  \frac{\partial Z_j} {\partial Z_i}=0\
\text{ if }  \ i\neq j,
$$
and
$$
\frac{\partial (fg)} {\partial Z_i}=\frac{\partial f} {\partial Z_i}
g +f\frac{\partial g} {\partial Z_i}
$$
for any  $f,g\in \CC[Z_1,\ldots, Z_n]$ and $i,j=1,\ldots n$.  Notice
that if $\alpha=g_{i_1}\cdots g_{i_p}$, $|\alpha|=p$, and $q$ of the
$g_{i_1},\ldots, g_{i_p}$ are equal to $g_j$, then $\frac{\partial
Z_\alpha} {\partial Z_j}$ is the sum of the $q$ words obtained by
deleting each occurence of $Z_j$ in $Z_\alpha:=Z_{i_1}\cdots
Z_{i_p}$. The same definition extends to formal power series in the
noncommuting indeterminates $Z_1,\ldots, Z_n$.
  If $F:=\sum\limits_{\alpha\in \FF_n^+} Z_\alpha\otimes A_{(\alpha)}$ is a
  power series with  operator-valued coefficients, then   the    free partial derivative
    of $F$ with respect to $Z_{i} $ is the power series
$$
\frac {\partial F}{\partial Z_{i}} := \sum_{\alpha\in \FF_n^+}
 \frac {\partial Z_\alpha}{\partial Z_{i} }\otimes A_{(\alpha)}.
$$
In \cite{Po-holomorphic}, we showed that if $F$ is a free
holomorphic function  on $[B(\cH)^n]_{\gamma}$ then so is
 $\frac {\partial F}{\partial Z_{i}}$, $i=1,\ldots, n$.

\bigskip

 Let $F_1,\ldots, F_n$ be free  holomorphic functions  on $[B(\cH)^n]_\gamma$  with
 scalar coefficients. Then the map $F:[B(\cH)^n]_\gamma\to B(\cH)^n$ defined by $F:=(F_1,\ldots,
 F_n)$ is a free holomorphic function.  We define $F'(0)$ as the
 linear operator on $\CC^n$ having the matrix $\left[\left( \frac{\partial F_i}
{\partial Z_j}\right)(0)\right]_{i,j=1,\ldots, n}$.

 Now, we can prove the following    noncommutative version of
 Cartan's  uniqueness theorem \cite{Ca}, for free holomorphic functions.

 \begin{theorem} \label{cartan1}
  Let $F:[B(\cH)^n]_\gamma\to [B(\cH)^n]_\gamma$ be a free
 holomorphic function such that $F(0)=0$ and $F'(0)=I_n$. Then
 $$F(X_1,\ldots, X_n)=(X_1,\ldots, X_n),\qquad   (X_1,\ldots,
 X_n)\in[B(\cH)^n]_\gamma.
 $$
\end{theorem}

\begin{proof} First notice that, due to the hypothesis,  $F$ has the form $F=(F_1,\ldots,
F_n)$, where each  $F_i$, $i=1,\ldots, n$,  is a free holomorphic
function on $[B(\cH)^n]_\gamma$ with scalar coefficients, of the
form
$$
F_i(X_1,\ldots, X_n)=X_i+\sum_{k=2}^\infty \sum_{|\alpha|=k}
a_\alpha^{(i)} X_\alpha,\qquad (X_1,\ldots, X_n)\in
[B(\cH)^n]_\gamma.
$$
Assume that there exists $\alpha\in \FF_n^+$, $|\alpha|\geq2$, and
$i\in \{1,\ldots, n\}$ such that $a_\alpha^{(i)}\neq 0$ . Let $m\geq
2$ be the smallest natural number such that there exists
$\alpha_0\in \FF_n^+$, $|\alpha_0|=m$, and $i_0\in \{1,\ldots, n\}$
such that $a_{\alpha_0}^{(i_0)}\neq 0$. Then we have
$$
F_i(X_1,\ldots, X_n)=X_i+ H_i(X_1,\ldots, X_n) \quad \text{ and }
\quad H_i(X_1,\ldots, X_n)= \sum_{k=m}^\infty G_k^{(i)}(X_1,\ldots,
X_n)
$$
for any $ (X_1,\ldots, X_n)\in [B(\cH)^n]_\gamma,$ where
$G_k^{(i)}(X_1,\ldots, X_n):= \sum_{|\alpha|=k} a_\alpha^{(i)}
X_\alpha$ for each $k\geq m$ and $i\in \{1,\ldots, n\}$.
Due to Theorem \ref{compo} and the fact that $Hol(B(\cH)^n_\gamma)$
is an algebra, we deduce that, for each $i\in \{1,\ldots, n\}$,
$G_m^{(i)}\circ F$ is a free holomorphic function and
$$
(G_m^{(i)}\circ F)(X_1,\ldots, X_n)=\sum_{|\alpha|=m} a_\alpha^{(i)}
F_\alpha (X_1,\ldots, X_n)=G_m^{(i)}(X_1,\ldots, X_n)+
K_{m+1}^{(i)}(X_1,\ldots, X_n),
$$
where $K_{m+1}^{(i)}$ is a free holomorphic function containing only
monomials of degree $\geq m+1$  in its representation. Theorem
\ref{compo} implies that $F\circ F$ is a free holomorphic function
and, due to the considerations  above, we have
\begin{equation*}
\begin{split}
(F\circ F)(X_1,\ldots, X_n) &=[X_1,\ldots,
X_n]+\left[2G_m^{(1)}(X_1,\ldots, X_n),\ldots,2G_m^{(n)}(X_1,\ldots, X_n)\right]\\
&\qquad+ \left[K_{m+1}^{(1)}(X_1,\ldots, X_n),\ldots,
K_{m+1}^{(n)}(X_1,\ldots, X_n)\right].
\end{split}
\end{equation*}
Continuing this process  and setting $F^N:=\underbrace{(F\circ\cdots
\circ F)}_{\text{$N$}\ times}$,  $N\in \NN$, one can see that
\begin{equation*}
\begin{split}
 F^N(X_1,\ldots, X_n) &=[X_1,\ldots,
X_n]+\left[NG_m^{(1)}(X_1,\ldots, X_n),\ldots,NG_m^{(n)}(X_1,\ldots, X_n)\right]\\
&\qquad+ \left[E_{m+1}^{(1)}(X_1,\ldots, X_n),\ldots,
E_{m+1}^{(n)}(X_1,\ldots, X_n)\right],
\end{split}
\end{equation*}
where, for each $i=1,\ldots, n$, $E_{m+1}^{(i)}$  is a free
holomorphic function containing only monomials of degree $\geq m+1$
in its representation. Hence, and using the fact that $S_1,\ldots,
S_n$ are isometries with orthogonal ranges, we deduce that
\begin{equation}
\label{Fk} F^N(rS_1,\ldots, rS_n)^* e_\alpha = r\left[\begin{matrix}
S_1^*\\\vdots \\S_n^*\end{matrix}\right]e_\alpha+ r^m
N\left[\begin{matrix} G_m^{(1)}(S_1,\ldots, S_n)^*\\\vdots\\
G_m^{(n)}(S_1,\ldots, S_n)^*\end{matrix}\right]e_\alpha
\end{equation}
for any $\alpha\in \FF_n^+$ with $|\alpha|=2$, $r\in (0,\gamma)$,
and $N=1,2, \ldots$.  We recall that $\{e_\alpha\}_{\alpha\in
\FF_n^+}$ is the standard orthonormal basis for $F^2(H_n)$. Since
$G_m^{(i)}$, $i=1,\ldots, n$, are homogeneous  noncommutative
polynomials of degree $m$ and $|\alpha_0|=m$, we have
$$
C:=\left\|\left[\begin{matrix} G_m^{(1)}(S_1,\ldots, S_n)^*\\\vdots\\
G_m^{(n)}(S_1,\ldots,
S_n)^*\end{matrix}\right]e_{\alpha_0}\right\|\geq
|a_{\alpha_0}^{i_0}|>0.
$$
Now, relation  \eqref{Fk} implies
\begin{equation}
\label{NC}
 r^m N C< r+\|F^N(rS_1,\ldots, rS_n)^*e_{\alpha_0}\|\quad
\text{ for any } \ N\in \NN.
\end{equation}
On the other hand, since $\text{\rm range}\, F^N \subseteq
[B(\cH)^n]_{\gamma}$, Lemma \ref{range} implies  $\|F^N(rS_1,\ldots,
rS_n)\|<\gamma$ for any $N\in \NN$. Hence and using \eqref{NC}, we
deduce that $r^m N C<r+\gamma$ for any $N\in \NN$, which is a
contradiction. This completes the proof.
\end{proof}

If $L:=[a_{ij}]_{n\times n}$ is a  bounded linear operator on
$\CC^n$, it generates a free holomorphic function on
$[B(\cH)^n]_\gamma$ by setting
$$
\Phi_{L}(X_1,\ldots, X_n):=[X_1,\ldots, X_n]{\bf
L}=\left[\sum_{i=1}^n a_{i1} X_i,\cdots,\sum_{i=1}^n a_{in}
X_i\right]
$$
where ${\bf L}:=[a_{ij}I_\cH]_{n\times n}$. By abuse of notation, we
also write  $\Phi_L(X)=XL$.

A map $F:[B(\cH)^n]_{\gamma_1}\to [B(\cH)^n]_{\gamma_2}$ is called
free biholomorphic  if $F$  is  free homolorphic, one-to-one and
onto, and  has  free holomorphic inverse. The automorphism group of
$[B(\cH)^n]_{\gamma}$, denoted by $Aut([B(\cH)^n_\gamma)$, consists
of all free biholomorphic functions  of $[B(\cH)^n]_{\gamma}$. It is
clear that $Aut([B(\cH)^n_\gamma)$ is  a group with respect to the
composition of free holomorphic functions.

In what follows, we characterize the free biholomorphic functions
with $F(0)=0$.

\begin{theorem}
\label{Cartan2} Let $F:[B(\cH)^n]_{\gamma_1}\to
[B(\cH)^n]_{\gamma_2}$  be a free biholomorphic function with
$F(0)=0$. Then there is an invertible bounded linear operator $L$ on
$\CC^n$ such that
$$ F(X_1,\ldots, X_n)=\Phi_L(X_1,\ldots, X_n), \qquad (X_1,\ldots, X_n)\in [B(\cH)^n]_{\gamma_1}.
$$
\end{theorem}

\begin{proof}
Since $F(0)=0$, $F$ has the representation $F=(F_1,\ldots, F_n)$,
where $F_j$ is a free holomorphic function on
$[B(\cH)^n]_{\gamma_1}$ with scalar coefficients, having the form
\begin{equation}
\label{Fj} F_j(X_1,\ldots, X_n)=\sum_{k=1}^n a_{kj}X_k+
\Psi_2^{(j)}(X_1,\ldots, X_n),
\end{equation}
and $\Psi_2^{(j)}$ is a free holomorphic function of the form
$\Psi_2^{(j)}(X_1,\ldots, X_n)=\sum_{m=2}^\infty \sum_{|\alpha|=m}
a_\alpha^{(j)} X_\alpha$.
Since $F$ is a  free biholomorphic function with $F(0)=0$, its
inverse $G:=F^{-1}:[B(\cH)^n]_{\gamma_2}\to [B(\cH)^n]_{\gamma_1}$
is a free holomorphic function with $G(0)=0$. Therefore, we have
$G=(G_1,\ldots, G_n)$, where $G_j$ is a free holomorphic function on
$[B(\cH)^n]_{\gamma_2}$ with scalar coefficients, having the form
\begin{equation}
\label{Gj} G_j(X_1,\ldots, X_n)=\sum_{k=1}^n b_{kj}X_k+
\Gamma_2^{(j)}(X_1,\ldots, X_n),
\end{equation}
and $\Gamma_2^{(j)}$ is a free holomorphic function of the form
$\Gamma_2^{(j)}(X_1,\ldots, X_n)=\sum_{m=2}^\infty \sum_{|\alpha|=m}
b_\alpha^{(j)} X_\alpha$. Consider the matrices
$L:=[a_{ij}]_{n\times n}$ and $B:=[b_{ij}]_{n\times n}$.
 Using the representations \eqref{Fj} and
\eqref{Gj}, we have
\begin{equation*}
\begin{split}
&(G\circ F)(X_1,\ldots, X_n)\\
 & =\left[\sum_{j=1}^n b_{j1}
F_j+\Gamma_2^{(1)}(F_1,\ldots, F_n),\ldots, \sum_{j=1}^n b_{jn}
F_j+\Gamma_2^{(n)}(F_1,\ldots, F_n)\right]\\
&=\left[\sum_{j=1}^n b_{j1} \left(\sum_{k=1}^n a_{kj}
X_k\right),\ldots, \sum_{j=1}^n b_{jn} \left(\sum_{k=1}^n a_{kj}
X_k\right)\right] + \left[\Gamma_2^{(1)}(F_1,\ldots,
F_n),\ldots,\Gamma_2^{(n)}(F_1,\ldots, F_n)\right]\\
&=[X_1,\ldots, X_n]BL + \left[\Psi_2^{(1)}(X_1,\ldots,
X_n),\ldots,\Psi_2^{(n)}(X_1,\ldots, X_n)\right],
\end{split}
\end{equation*}
where $\Psi_2^{(j)}$, $j=1,\ldots, n$, are free holomorphic
functions containing only monomials of degree $\geq 2$ in their
representations. Since $(G\circ F)(X_1,\ldots,
X_n)=(X_1,\ldots,X_n)$ for any $(X_1,\ldots, X_n)\in [B(\cH)^n]_1$,
we deduce that $BL=I_n$ and $\Psi_2^{(j)}(X_1,\ldots, X_n)=0$ for
$j=1,\ldots, n$. Similarly, we can prove that $LB=I_n$. Therefore
$L$ is an invertible operator on $\CC^n$.
Note that, for each $\theta\in\RR$, the map
$$
(X_1,\ldots, X_n)\mapsto e^{-i\theta}F(e^{i\theta} X_1, \ldots,
e^{i\theta} X_n)
$$
is a free holomorphic function on $[B(\cH)^n]_{\gamma_1}$ with
values in $[B(\cH)^n]_{\gamma_2}$. Moreover, due to  \eqref{Fj}, we
have
\begin{equation*}
\begin{split}
e^{-i\theta}F(e^{i\theta} X_1,& \ldots, e^{i\theta}
X_n)\\
&=\left[\sum_{k=1}^n a_{k1}X_k+
e^{-i\theta}\Psi_2^{(1)}(e^{i\theta}X_1,\ldots,e^{i\theta}
X_n),\ldots, \sum_{k=1}^n a_{kn}X_k+
e^{-i\theta}\Psi_2^{(n)}(e^{i\theta}X_1,\ldots,
e^{i\theta}X_n)\right]
\end{split}
\end{equation*}
According to Theorem \ref{compo}, the map
$H:[B(\cH)^n]_{\gamma_1}\to [B(\cH)^n]_{\gamma_1}$ defined by
\begin{equation}
\label{HGF} H(X_1,\ldots, X_n):=G\left(e^{-i\theta}F(e^{i\theta}
X_1, \ldots, e^{i\theta} X_n)\right), \qquad (X_1,\ldots, X_n)\in
[B(\cH)^n]_{\gamma_1},
\end{equation}
is a free holomorphic function with $H(0)=0$. On the other hand,
taking into account the representations of the functions involved in
the definition of $H$,  calculations as above reveal that
$$
H(X_1,\ldots, X_n)=[X_1,\ldots, X_n]BL +
\left[\Phi_2^{(1)}(X_1,\ldots, X_n),\ldots,\Phi_2^{(n)}(X_1,\ldots,
X_n)\right],
$$
where $\Phi_2^{(j)}$, $j=1,\ldots, n$, are free holomorphic
functions containing only monomials of degree $\geq 2$ in their
representations. Since $BL=I_n$, we deduce that $H'(0)=I_n$.
Applying Theorem \ref{cartan1} to the free holomorphic function $H$,
we conclude that $H(X_1,\ldots, X_n)=(X_1,\ldots, X_n)$ for any
$(X_1,\ldots, X_n)\in [B(\cH)^n]_{\gamma_1}$. Hence, and due to
relation  \eqref{HGF}, we obtain
$$
e^{i\theta} F(X_1,\ldots, X_n)=F(e^{i\theta} X_1, \ldots,
e^{i\theta} X_n)
$$
for any $(X_1,\ldots, X_n)\in [B(\cH)^n]_{\gamma_1}$ and $\theta\in
\RR$.  Using the representations given by \eqref{Fj} and the
uniqueness of the coefficients of a  free holomorphic function, the
latter equality implies
$$
a_{\alpha}^{(j)} e^{i\theta|\alpha|}=e^{i\theta} a_\alpha^{(j)}\quad
\text{ for any }\ \theta\in \RR,
$$
where $\alpha\in \FF_n^+$ with $|\alpha|\geq 2$, and $j=1,\ldots,n$.
Hence,  $a_\alpha^{(j)}=0$ and, consequently,
$$
F(X_1,\ldots, X_n)=\left[\sum_{k=1}^n a_{k1} X_k,\ldots,
\sum_{k=1}^n a_{kn} X_k\right]=[X_1,\ldots, X_n]L,
$$
which completes the proof.
\end{proof}

 Now we are ready to prove that any free holomorphic automorphism of
 $[B(\cH)^n]_1$  that fixes the origin is implemented by a unitary
 operator on $\CC^n$.
\begin{theorem}
\label{aut1} Let $\Psi: [B(\cH)^n]_1\to [B(\cH)^n]_1$ be a free
holomorphic function with $\Psi(0)=0$. Then $\Psi$ is a  free
holomorphic automorphism of $[B(\cH)^n]_1$ if and only if there is a
unitary operator $U$ on $\CC^n$  such that
\begin{equation}
\label{phi=U}\Psi(X_1,\ldots X_n)= \Phi_U(X_1,\ldots
X_n):=[X_1,\ldots, X_n]U , \qquad (X_1,\ldots, X_n)\in [B(\cH)^n]_1.
\end{equation}
\end{theorem}
\begin{proof}
One implication is obvious. Assume that $\Psi$ is a  free
holomorphic automorphism of $[B(\cH)^n]_1$. Applying Theorem
\ref{Cartan2} to $\Psi$, we find an invertible operator $U$ on
$\CC^n$ such that \eqref{phi=U} holds.  In particular, taking
$(X_1,\ldots, X_n)=(\lambda_1 I_\cH, \ldots, \lambda_n I_\cH)$ with
$(\lambda_1,\ldots, \lambda_n)\in \BB_n$, we deduce that $U$  is a
contraction on $\CC^n$. On the other hand, since $\Psi$ is a  free
holomorphic automorphism of $[B(\cH)^n]_1$ and
$$\Psi^{-1}(X_1,\ldots X_n)=[X_1,\ldots, X_n]U^{-1}, \qquad (X_1,\ldots X_n)\in  [B(\cH)^n]_1,
$$
one can similarly deduce that $U^{-1}$ is also a contraction.
Consequently, for any $x\in \CC^n$, we have
$$
\|x\|=\|U^{-1} Ux\|\leq \|U^{-1}\|\|U\|\|x\|\leq  \|x\|.
$$
Hence $U$ is an isometry which is invertible, and therefore unitary.
The proof is complete.
\end{proof}

 \bigskip

\section{The group  of  free  holomorphic   automorphisms of  $[B(\cH)^n]_1$ }

The theory of noncommutative characteristic functions for row
contractions is used  to find all the involutive free holomorphic
automorphisms of $[B(\cH)^n]_1$. Combining this  result with those
from  Section 1, we determine all  free holomorphic automorphisms of
the noncommutative ball $[B(\cH)^n]_1$. We show that any $\Psi\in
Aut(B(\cH)^n_1)$  has the form
$$
\Psi=\Phi_U \circ \Psi_\lambda,
$$
where $\Phi_U$ is an automorphism implemented by a unitary operator
$U$ on $\CC^n$ and $\Psi_\lambda$ is an involutive free holomorphic
automorphism associated with $\lambda:=\Psi^{-1} (0)\in \BB_n$.
Moreover, we show that the automorphism group $Aut(B(\cH)^n_1)$ is
isomorphic to $Aut(\BB_n)$, the Moebius group of $\BB_n$, via the
noncommutative Poisson transform.

To begin this section, we  recall from \cite{Po-poisson} a few facts
about noncommutative Poisson transforms associated with row
contractions $T:=[T_1,\ldots, T_n]$, \ $T_i\in B(\cK)$,  where $\cK$
is a Hilbert space. Let $\FF_n^+$ be the unital free semigroup on
$n$ generators
      $g_1,\dots,g_n$, and the identity $g_0$.
       We recall that $e_\alpha:=
e_{i_1}\otimes\cdots \otimes  e_{i_k}$ and $e_{g_0}:=1$. Note that
$\{e_\alpha\}_{\alpha\in \FF_n^+}$ is an orthonormal basis for
$F^2(H_n)$. For  each $0<r\leq 1$, define the defect operator
$\Delta_{T,r}:=(I_\cK-r^2T_1T_1^*-\cdots -r^2 T_nT_n^*)^{1/2}$.
The noncommutative Poisson  kernel associated with $T$ is the family
of operators
$$
K_{T,r} :\cK\to  \overline{\Delta_{T,r}\cH} \otimes  F^2(H_n), \quad
0<r\leq 1,
$$
defined by

\begin{equation*}
K_{T,r}h:= \sum_{k=0}^\infty \sum_{|\alpha|=k} r^{|\alpha|}
\Delta_{T,r} T_\alpha^*h\otimes  e_\alpha,\quad h\in \cK.
\end{equation*}
When $r=1$, we denote $\Delta_T:=\Delta_{T,1}$ and $K_T:=K_{T,1}$.
The operators $K_{T,r}$ are isometries if $0<r<1$, and
$$
K_T^*K_T=I_\cK- \text{\rm SOT-}\lim_{k\to\infty} \sum_{|\alpha|=k}
T_\alpha T_\alpha^*.
$$
Thus $K_T$ is an isometry if and only if $T$ is a {\it pure} row
 contraction,
i.e., $ \text{\rm SOT-}\lim\limits_{k\to\infty} \sum_{|\alpha|=k}
T_\alpha T_\alpha^*=0.$ The noncommutative Poisson transform at
 $T:=[T_1,\ldots, T_n]$ is the unital completely contractive  linear map
 $P_T:C^*(S_1,\ldots, S_n)\to B(\cK)$ defined by
 \begin{equation*}
 P_T[f]:=\lim_{r\to 1} K_{T,r}^* (I_\cK \otimes f)K_{T,r}, \qquad f\in C^*(S_1,\ldots,
 S_n),
\end{equation*}
 where the limit exists in the norm topology of $B(\cK)$. Moreover, we have
 $$
 P_T(S_\alpha S_\beta^*)=T_\alpha T_\beta^*, \qquad \alpha,\beta\in \FF_n^+.
 $$
 When $T:=[T_1,\ldots, T_n]$  is a pure row contraction,
   we have $$P_T(f)=K_T^*(I_{\cD_{T}}\otimes f)K_T,
   $$
   where $\cD_T=\overline{\Delta_T \cK}$.
We refer to \cite{Po-poisson}, \cite{Po-curvature},  and
\cite{Po-unitary} for more on noncommutative Poisson transforms on
$C^*$-algebras generated by isometries. When $T$ is a completely
non-coisometric (c.n.c.) row contraction, i.e.,
 there is no $h\in \cK$, $h\neq 0$, such that
 $$
 \sum_{|\alpha|=k}\|T_\alpha^* h\|^2=\|h\|^2
 \quad \text{\rm for any } \ k=1,2,\ldots,
 $$
an  $F_n^\infty$-functional calculus was developed  in
\cite{Po-funct}.
 We showed that if $f=\sum\limits_{\alpha\in \FF_n^+} a_\alpha S_\alpha$ is
 in $F_n^\infty$, then
 \begin{equation*}
 \Gamma_T(f)=f(T_1,\ldots, T_n):=
 \text{\rm SOT-}\lim_{r\to 1}\sum_{k=0}^\infty
  \sum_{|\alpha|=k} r^{|\alpha|} a_\alpha T_\alpha
\end{equation*}
exists and $\Gamma_T:F_n^\infty\to B(\cK)$ is a completely
contractive homomorphism and WOT-continuous (resp. SOT-continuous)
on bounded sets. Moreover, we showed (see \cite{Po-unitary})  that
$\Gamma_T(f)=P_T[f]$, $f\in F_n^\infty$, where
 \begin{equation}
 \label{pois-sot}
  P_T[f]:= \text{\rm
SOT-}\lim_{r\to 1}K_{T,r} ( I_\cK\otimes f) K_{T,r},\qquad f\in
F_n^\infty,
\end{equation} is the extension of the noncommutative Poisson
transform  to $F_n^\infty$.

The characteristic  function associated with an arbitrary row
contraction $T:=[T_1,\ldots, T_n]$, \ $T_i\in B(\cK)$, was
introduced in \cite{Po-charact} (see \cite{SzF-book} for the
classical case $n=1$) and it was proved to be  a complete unitary
invariant for completely non-coisometric  row contractions. The
characteristic  function  of $T$ is  a   multi-analytic operator
with respect to $S_1,\ldots, S_n$,
$$
\tilde{\Theta}_T:F^2(H_n)\otimes \cD_{T^*}\to F^2(H_n)\otimes \cD_T,
$$
with the formal Fourier representation
\begin{equation*}
\begin{split}
 \Theta_T(R_1,\ldots, R_n):= -I_{F^2(H_n)}\otimes T+
\left(I_{F^2(H_n)}\otimes \Delta_T\right)&\left(I_{F^2(H_n)\otimes \cK}
-\sum_{i=1}^n R_i\otimes T_i^*\right)^{-1}\\
&\left[R_1\otimes I_\cK,\ldots, R_n\otimes I_\cK \right]
\left(I_{F^2(H_n)}\otimes \Delta_{T^*}\right),
\end{split}
\end{equation*}
where $R_1,\ldots, R_n$ are the right creation operators on the full
Fock space $F^2(H_n)$.
 Here,  we need to clarify some notations since some of them are different
 from those considered in \cite{Po-charact}.
The defect operators  associated with a row contraction
$T:=[T_1,\ldots, T_n]$ are
\begin{equation}
\label{DelDel}
 \Delta_T:=\left( I_\cK-\sum_{i=1}^n
T_iT_i^*\right)^{1/2}\in B(\cK) \quad \text{ and }\quad
\Delta_{T^*}:=(I-T^*T)^{1/2}\in B(\cK^{(n)}),
\end{equation}
while the defect spaces are $\cD_T:=\overline{\Delta_T\cK}$ and
$\cD_{T^*}:=\overline{\Delta_{T^*}\cK^{(n)}}$, where $\cK^{(n)}$
denotes the direct sum of $n$ copies of $\cK$.
Due to the $F_n^\infty$-functional calculus  for row contractions,
one can define
$$
\Theta_T(X_1,\ldots, X_n):=\text{\rm SOT-}\lim_{r\to 1}\Theta_T
(rX_1,\ldots, rX_n)
$$
for any c.n.c.  row contraction $(X_1,\ldots, X_n)\in B(\cG)^n$,
where $\cG$ is a Hilbert space. Depending of $T$, the characteristic
function $\Theta_T$ may be well-defined on a larger subset of
$B(\cG)^n$. For example, if $\|T\|<1$, then $X\mapsto \Theta_T (X)$
is  a free holomorphic function on the open ball
$[B(\cG)^n]_{\gamma}$, where $\gamma:=\frac{1}{\|T\|}$.

In particular, the characteristic function  $\tilde \Theta_T$
generates a bounded free holomorphic function $\Theta_T$ (also
called characteristic function) with operator-valued coefficients
in $B(\cD_{T^*}, \cD_T)$. Notice also that
\begin{equation*}
\begin{split}
 \Theta_T(X_1,\ldots, X_n)= -I_{\cG}\otimes T+
\left(I_\cG \otimes \Delta_T\right)&\left(I_{\cG\otimes
\cK}-\sum_{i=1}^n X_i\otimes
T_i^*\right)^{-1}\\
&\left[X_1\otimes I_\cK,\ldots, X_n\otimes I_\cK \right]
\left(I_\cG\otimes \Delta_{T^*}\right)
\end{split}
\end{equation*}
for any $(X_1,\ldots, X_n)\in [B(\cG)^n]_1$.  The characteristic
function $\tilde{\Theta}_T$ is the boundary function of $\Theta_T$
with respect to $R_1,\ldots, R_n$ in the sense that
$$\tilde{\Theta}_T=\text{\rm SOT-}\lim_{r\to 1}\Theta_T(rR_1,\ldots,
rR_n),$$
 where  $\Theta(rR_1,\ldots, rR_n)$ is
in $\cR_n\otimes B(\cK)$.
 In \cite{Po-varieties} (see Theorem 3.2 and   Corollary 3.3), we proved that
\begin{equation}\label{fa}
I-\tilde{\Theta}_T \tilde{\Theta}_T^*=K_TK_T^*,
\end{equation}
where  $K_T$ is the corresponding Poisson kernel, and
$$
I_{\cG\otimes \cD_T}-\Theta_T(X )\Theta_T(X )^*= \Delta_{\tilde
T}(I-\hat X\tilde T^*)^{-1}(I-\hat X \hat X^*)(I-\tilde T \hat
X^*)^{-1} \Delta_{\tilde T}
$$
for any $X:=(X_1,\ldots, X_n)\in [B(\cG)^n]_1$. Here we use the
notations $\hat X:=[X_1\otimes I_\cK,\ldots, X_n\otimes I_\cK]$ and
$\tilde T:=[I_\cG\otimes T_1,\ldots, I_\cG\otimes T_n]$. A closer
look at the proofs of the above-mentioned  results (see
\cite{Po-varieties})  reveals that one can   prove, in a similar
manner,  a little bit more.
\begin{proposition}
\label{fact-formula} Let $T:=[T_1,\ldots, T_n]$, $T_i\in B(\cK)$, be
a row contraction  and let $\Theta_T$ be its characteristic
function.  Then
 \begin{equation}
 \label{VY}
I_{\cG\otimes \cD_T}-\Theta_T(X)\Theta_T(Y)^*= \Delta_{\tilde
T}(I-\hat X\tilde T^*)^{-1}(I-\hat X \hat Y^*)(I-\tilde T \hat
Y^*)^{-1} \Delta_{\tilde T}
\end{equation}
and
\begin{equation}
 \label{VY*}
I_{\cG\otimes \cD_{T^*}}-\Theta_T(X )^*\Theta_T(Y)= \Delta_{{\tilde
T}^*}(I-{\hat X}^*\tilde T)^{-1}(I-{\hat X}^* \hat Y)(I-{\tilde T}^*
\hat Y)^{-1} \Delta_{{\tilde T}^*}
\end{equation}
for any  $X:=[X_1,\ldots, X_n]$ and $Y:=[Y_1,\ldots, Y_n]$ in
$[B(\cG)^n]_1$. Moreover, if $\|T\|<1$, then the relations hold true
for any  $X, Y\in [B(\cG)^n]_{\gamma}$, where
$\gamma:=\frac{1}{\|T\|}$.
\end{proposition}

 Now, we consider an important particular case. Let
$T=\lambda:=(\lambda_1,\ldots,\lambda_n)\in \BB_n$  and think of
$\lambda$  as a row contraction acting from $\CC^n$  to $\CC$. In
this case, due to \eqref{DelDel}, we have
$\Delta_\lambda=(1-\|\lambda\|_2^2)^{1/2} I_\CC$ and
$\Delta_{\lambda^*}=(I_\cK-\lambda^*\lambda)^{1/2}$. Since
$\|\lambda\|_2<1$, it is clear that $\cD_\lambda=\CC$ and
$\cD_{\lambda^*}=\CC^n$.
 For simplicity, we
also use the notation $ {\lambda}:=[ {\lambda}_1 I_\cG,\ldots,
{\lambda}_n I_\cG]$ for the row contraction acting from $\cG^{(n)}$
to $\cG$, where $\cG$ is a Hilbert space.

The characteristic function of the row contraction
$\lambda:=(\lambda_1,\ldots,\lambda_n)\in \BB_n$ is the boundary
function $\tilde{\Theta}_\lambda$, with respect to $R_1,\ldots,
R_n$, of the free holomorphic function
$\Theta_\lambda:[B(\cH)^n]_1\to [B(\cH)^n]_1$ given by
\begin{equation}
\label{Theta-lambda} \Theta_\lambda(X_1,\ldots, X_n):=-{
\lambda}+\Delta_{ \lambda}\left(I_\cH-\sum_{i=1}^n \bar{{
\lambda}}_i X_i\right)^{-1} [X_1,\ldots, X_n] \Delta_{{\lambda}^*}
\end{equation}
for $(X_1,\ldots, X_n)\in [B(\cH)^n]_1$.  Note that, when
$\lambda=0$, we have $\Theta_0(X)=X$.

\begin{proposition}
\label{prop-Tl} Let $\lambda:=(\lambda_1,\ldots,\lambda_n)\in
\BB_n$, $\lambda\neq 0$,  and let  $\tilde \Theta_\lambda$ be its
characteristic function. Then
\begin{enumerate}
\item[(i)] the map $\Theta_\lambda$, defined  by
\eqref{Theta-lambda}, is a free holomorphic function on the open
ball $[B(\cH)^n]_\gamma$, where $\gamma:=\frac{1}{\|\lambda\|_2}$;
\item[(ii)]
$\tilde \Theta_\lambda=\Theta_\lambda(R_1,\ldots, R_n)
 =-{ \lambda}+\Delta_{
\lambda}\left(I_{F^2(H_n)}-\sum_{i=1}^n \bar{{ \lambda}}_i
R_i\right)^{-1} [R_1,\ldots, R_n] \Delta_{{\lambda}^*}$;
\item[(iii)] $\tilde \Theta_\lambda$ is an inner multi-analytic
operator in the noncommutative disc algebra $\cR_n$;
\item[(iv)] $\tilde \Theta_\lambda$ is a pure row contraction;

\item[(v)] $\text{\rm rank}\,(I-\tilde \Theta_\lambda\tilde \Theta_\lambda^*)=1$
and $\tilde \Theta_\lambda$ is unitarily equivalent  to
$[R_1,\ldots, R_n]$.
\end{enumerate}
\end{proposition}
\begin{proof}
Let $\lambda:=(\lambda_1,\ldots,\lambda_n)\in \BB_n$, $\lambda\neq
0$, set $\gamma:=\frac{1}{\|\lambda\|}$, and let $r<\gamma$. Denote
  $\lambda_{\alpha}:=\lambda_{i_1}\lambda_{i_2}\ldots
\lambda_{i_m}$ if $\alpha=g_{i_1}g_{i_2}\dots g_{i_m}\in \FF_n^+$,
and $\lambda_{g_0}:=1$. Since the right creation operators
$R_1,\ldots, R_n$  are isometries with orthogonal ranges, we have
$\left\|\sum_{i=1}^n r\bar \lambda_i R_i\right\|=r\|\lambda\|_2<1$.
Consequently,
$$
\left(I_{F^2(H_n)}-\sum_{i=1}^n r\bar{{ \lambda}}_i R_i\right)^{-1}=
\sum_{k=0}^\infty\sum_{|\alpha|=k}r^{|\alpha|} \bar \lambda_\alpha
R_\alpha
$$
is convergent in the operator norm topology  for any $r\in
[0,\gamma)$ and, therefore, is an element of the noncommutative disc
algebra $\cR_n$. Using the noncommutative von Neumann inequality for
row contractions \cite{Po-von}, we deduce that
$\sum_{k=0}^\infty\sum_{|\alpha|=k} \bar \lambda_\alpha X_\alpha
=\left(I-\sum_{i=1}^n \bar \lambda_i X_i\right)^{-1}$ is a free
holomorphic function on the open ball $[B(\cH)^n]_\gamma$, and this
implies part (i) of the proposition. To prove (ii), note that
$\Theta_\lambda$ is continuous on $[B(\cH)^n]_\gamma$ in the
operator norm topology, which implies
$$\tilde{\Theta}_\lambda=\lim_{r\to 1}\Theta_\lambda(rR_1,\ldots,
rR_n)=\Theta_\lambda(R_1,\ldots, R_n)$$
 and $\tilde{\Theta}_\lambda$
is in $\cR_n$.
 Since
$\lambda=(\lambda_1,\ldots, \lambda_n)$ is a strict row contraction,
it is pure and, consequently, due to \cite{Po-charact},  the
characterisic function $\tilde \Theta_\lambda$ is an inner
multi-analytic operator, i.e. $\tilde\Theta_\lambda^*
\tilde\Theta_\lambda=I$.
  One can also  obtain this fact using
Proposition \ref{fact-formula} in our particular case.

Now, let us  prove that $\tilde \Theta_\lambda$ is a pure row
contraction. Due to  relation \eqref{Theta-lambda}, it is clear that
$\Theta_\lambda=(-\lambda_1+ F_1,\ldots, -\lambda_n+F_n)$, where
$F_i $ is  a bounded free holomorphic function  with $F_i(0)=0$ and
the boundary function $\tilde F_i$  is in the noncommutative disc
algebra $\cR_n$, for each $i=1,\ldots, n$. Moreover, we have $\tilde
F_\alpha^* e_\beta =0$ if $\alpha, \beta\in \FF_n^+$ with
$|\alpha|>|\beta|$.

Let $k,q\in \NN$ be  such that $k>q$,  and let $\beta\in \FF_n^+$
$|\beta|=q$. Setting   $(T_1,\ldots, T_n):=\tilde \Theta_\lambda$,
we have
$$
T_\omega ^* e_\beta =\left(-\bar{\lambda}_{j_1}+\tilde
F_{j_1}^*\right)\cdots \left(-\bar{\lambda}_{j_k}+\tilde
F_{j_k}^*\right)e_\beta
$$
for each $\omega= g_{j_1}\cdots g_{j_k}\in \FF_n^+$.  Since
$\|\tilde \Theta_\lambda\|\leq 1$, it is clear that
 $\|\tilde F_j\|\leq 2$. Now, multiply the
right hand side of the equality above, apply the triangle inequality
followed by Cauchy-Schwarz inequality.
Summing up the resulting inequalities over all $\omega\in \FF_n^+$
with $|\omega|=k$, we obtain
$$
\sum_{|\omega|=k}\|T_\omega^* e_\beta\|^2 \leq \left[
\|\lambda\|_2^{2k}+ \left(\begin{matrix} k\\1\end{matrix}\right)
   \|\lambda\|_2^{2(k-1)}
+\cdots \right . + \left. \left(\begin{matrix}
k\\q\end{matrix}\right)   \|\lambda\|_2^{2(k-q)}
\right][1+2^2+\cdots+ (2^2)^q].
$$
Since $\|\lambda\|_2<1$ , it is easy to see that $\lim_{k\to
\infty}\sum_{|\omega|=k}\|T_\omega^* e_\beta\|^2=0. $
Since $\sum_{|\omega|=k}T_\omega T_\omega^*\leq I$ for any
$k=1,2,\dots$, we infer that $ \text{\rm WOT-}\lim_{k\to
\infty}\sum_{|\omega|=k} T_\omega T_\omega^*=0,$
 which shows that
 $ [T_1,\dots,  T_n]$ is a   pure row contraction.

 To prove (v), note that
  the
noncommutative Poisson kernel $K_\lambda:\CC\to F^2(H_n)$ is given
by
$$
K_\lambda(1)=\sum_{\alpha\in \FF_n^+}
(1-\|\lambda\|_2^2)^{1/2}\bar\lambda_\alpha \otimes e_\alpha.
$$
Due to  relation \eqref{fa}, we have $I-\tilde\Theta_\lambda
\tilde\Theta_\lambda^*=K_\lambda K_\lambda^*$.
  Since $\tilde\Theta_\lambda \tilde\Theta_\lambda^*$ is an
orthogonal projection, so is $K_\lambda K_\lambda^*$. Hence, we
deduce that $I-\tilde \Theta_\lambda\tilde \Theta_\lambda^*$ is a
rank one projection.  On the other hand, since $\tilde \Theta$ is a
pure row isometry,
 the noncommutative Wold type decomposition theorem (see
\cite{Po-isometric})  implies that
 $\tilde \Theta_\lambda$ is unitarily equivalent  to
$[R_1,\ldots, R_n]$. This  completes the proof.
\end{proof}

\begin{theorem}
\label{prop-cara} Let  $\lambda:=(\lambda_1,\ldots, \lambda_n)\in
\BB_n\backslash \{0\}$ and let $\gamma:=\frac{1}{\|\lambda\|_2}$.
Then $\Psi_\lambda:=-\Theta_\lambda$ is a free holomorphic function
on $[B(\cH)^n]_\gamma$ which has the following properties:
\begin{enumerate}
\item[(i)]
$\Psi_\lambda (0)=\lambda$ and $\Psi_\lambda(\lambda)=0$;
\item[(ii)] The identities
\begin{equation*}
I_{\cH}-\Psi_\lambda(X)\Psi_\lambda(Y)^*= \Delta_{\lambda}(I- X
\lambda^*)^{-1}(I-  X   Y^*)(I-  \lambda  Y^*)^{-1} \Delta_{\lambda}
\end{equation*}
and
\begin{equation*}
I_{\cH\otimes \CC^n}-\Psi_\lambda(X)^*\Psi_\lambda(Y)=
\Delta_{\lambda^*}(I-{X}^*\lambda)^{-1}(I-{X}^* Y)(I-{\lambda}^*
Y)^{-1} \Delta_{\lambda^*}
\end{equation*}

hold  for all  $X$ and $Y$ in  $[B(\cH)^n]_\gamma$;

\item[(iii)] $\Psi_\lambda$ is an involution, i.e., $\Psi_\lambda(\Psi_\lambda(X))=X$
for any $X\in [B(\cH)^n]_\gamma$;
\item[(iv)] $\Psi_\lambda$ is a free holomorphic automorphism of the
noncommutative unit ball $[B(\cH)^n]_1$;
\item[(v)] $\Psi_\lambda$ is a homeomorphism of $[B(\cH)^n]_1^-$
onto $[B(\cH)^n]_1^-$.
\end{enumerate}
\end{theorem}

\begin{proof} Using the identities
\begin{equation}
\label{defect} \Delta_\lambda \lambda=\lambda
\Delta_{\lambda^*}\quad \text{ and }\quad \lambda^*\Delta_\lambda=
\Delta_{\lambda^*} \lambda^*,
\end{equation}
one can easily see that $\Psi_\lambda(\lambda)=0$. Part (ii) follows
from Proposition  \ref{fact-formula}, in the particular case when
$T=\lambda=(\lambda_1,\ldots, \lambda_n)$. We remark that for a
fixed $\lambda\in \BB_n$,  $\Psi_\lambda$ is a free holomorphic
function  on the open ball $[B(\cH)^n]_\gamma$, where
$\gamma:=\frac{1}{\|\lambda\|_2}$, and the relations in part (ii)
hold true  on $[B(\cH)^n]_\gamma\supset [B(\cH)^n]_1^-$. Hence,
$\Psi_\lambda$ is continuous in the operator norm topology on
$[B(\cH)^n]_\gamma$.

To prove part (iii), note first that the operator
$I-\Psi_\lambda(X)\lambda^*$ is invertible for any $X\in
[B(\cH)^n]_\gamma$. Indeed, using relation \eqref{defect},  we have
\begin{equation*}
\begin{split}
I-\Psi_\lambda(X)\lambda^*&=I+\left[ -\lambda
+\Delta_\lambda(I-X\lambda^*)^{-1} X \Delta_{\lambda^*}\right]
\lambda^*\\
&=I-\lambda \lambda^* +\Delta_\lambda(I-X\lambda^*)^{-1} X \lambda^*
\Delta_\lambda\\
&=\Delta_\lambda\left[I+(I-X\lambda^*)^{-1} X \lambda^*\right] \Delta_\lambda\\
&=\Delta_\lambda (I-X\lambda^*)^{-1} \Delta_\lambda.
\end{split}
\end{equation*}
Note also that
\begin{equation*}
\begin{split}
\Psi_\lambda(X)\Delta_{\lambda^*} &=\left[ \lambda
-\Delta_\lambda(I-X\lambda^*)^{-1} X
\Delta_{\lambda^*}\right]\Delta_{\lambda^*}\\
&=\Delta_\lambda \left[\lambda-(I-X\lambda^*)^{-1} X
\Delta_{\lambda^*}^2\right].
\end{split}
\end{equation*}
 Due to  the relations above, we have
\begin{equation*}
\begin{split}
\Psi_\lambda (\Psi_\lambda (X))&= \lambda-\Delta_\lambda
(I-\Psi_\lambda(X) \lambda^*)^{-1} \Psi_\lambda(X)
\Delta_{\lambda^*}\\
&= \lambda -\Delta_\lambda\left[\Delta_\lambda (I-X\lambda^*)^{-1}
\Delta_\lambda\right]^{-1} \Delta_\lambda
\left[\lambda-(I-X\lambda^*)^{-1} X \Delta_{\lambda^*}^2\right]\\
&= \lambda -(I-X\lambda^*)\left[\lambda-(I-X\lambda^*)^{-1} X
\Delta_{\lambda^*}^2\right]\\
&= \lambda-(I-X\lambda^*)\lambda +X \Delta_{\lambda^*}^2\\
&=X.
\end{split}
\end{equation*}
Therefore, $\Psi_\lambda (\Psi_\lambda (X))=X$ for  any $X\in
[B(\cH)^n]_\gamma$, which proves (iii).

Fix $X\in [B(\cH)^n]_\gamma$. Since $\Delta_\lambda$ is invertible,
one can use part (ii) to show that $\|X\|\leq 1$ if and only if
$\|\Psi_\lambda (X)\|\leq 1$.  Similarly,  we can deduce  that
$\|X\|< 1$ if and only if $\|\Psi_\lambda (X)\|< 1$. Consequently,
due Theorem \ref{compo}, $\Psi_\lambda\circ \Psi_\lambda$  is a free
holomorphic function  on $[B(\cH)^n]_1$ and, taking into account
that $\Psi_\lambda (\Psi_\lambda (X))=X$ for any $X\in
[B(\cH)^n]_1^-$, we deduce that $\Psi_\lambda$ is a homeomorphism of
$[B(\cH)^n]_1^-$ onto $[B(\cH)^n]_1^-$ with
$\Psi_\lambda^{-1}=\Psi_\lambda$. Now, it is clear that
$\Psi_\lambda$ is a free holomorphic  automorphism of
$[B(\cH)^n]_1$.
 This
completes the proof.
\end{proof}

\begin{corollary}
\begin{enumerate}
\item[(i)]
 If $\lambda, \mu\in \BB_n$, then $\Psi_\mu\circ
\Psi_\lambda$ is a a free holomorphic  automorphism  of
$[B(\cH)^n]_1$ that takes $\lambda$ to $\mu$.
\item[(i)] If $\lambda\in \BB_n$, then the characteristic function
$\Theta_\lambda$ is a free holomorphic automorphism of
$[B(\cH)^n]_1$ and $\Theta_\lambda^{-1}(X)=-\Theta_\lambda(-X)$ for
$X\in [B(\cH)^n]_1$.
\end{enumerate}
\end{corollary}

We remark that a  formula for $\Psi_\mu\circ \Psi_\lambda$ is
presented in Section 4 (see Corollary \ref{formula}).

The next theorem characterizes  the free holomorphic  automorphisms
of $[B(\cH)^n]_1$.

\begin{theorem}
\label{automorph} If $\Psi\in Aut(B(\cH)^n_1)$ and
$\lambda:=\Psi^{-1}(0)$, then there is a unique unitary operator $U$
on $\CC^n$ such that
$$
\Psi=\Phi_U\circ \Psi_\lambda.
$$
The identity
\begin{equation*}
I_{\cH}-\Psi(X)\Psi(Y)^*= \Delta_{\lambda}(I- X \lambda^*)^{-1}(I- X
Y^*)(I-  \lambda  Y^*)^{-1} \Delta_{\lambda}
\end{equation*}
holds  for all $X$ and $Y$ in
 $[B(\cH)^n]_1^-$.
\end{theorem}
\begin{proof} Note that, due to Theorem \ref{compo} and Theorem
\ref{prop-cara},
   $\Psi\circ \Psi_\lambda$ is a a free holomorphic automorphism of
   $[B(\cH)^n]_1$ such that $(\Psi\circ \Psi_\lambda)(0)=0$.
   Applying Theorem \ref{aut1} to $\Psi\circ \Psi_\lambda$, we find a  unitary operator $U$ on
   $\CC^n$ such that $\Psi\circ \Psi_\lambda=\Phi_U$. Hence, and using the fact that
   $\Psi_\lambda\circ \Psi_\lambda=\text{\rm id}$, we deduce that
    $\Psi=\Phi_U\circ \Psi_\lambda.$  Now, the identity above  follows from
    part (ii) of Theorem \ref{prop-cara}. The proof is complete.
\end{proof}

As in  the proof of Theorem \ref{automorph}, but using the fact
 that $( \Psi_\lambda\circ\Psi)(0)=0$, one can also obtain  the following
 result.
\begin{remark}
If $\Psi\in Aut(B(\cH)^n_1)$ and $\lambda:=\Psi^{-1}(0)$, then there
is a  unitary operator $W$ on $\CC^n$ such that
$$
\Psi= \Psi_\lambda\circ \Phi_W.
$$
\end{remark}

Now we can prove the following extension theorem for
$Aut(B(\cH)^n_1)$.

\begin{proposition}
\label{auto-extens} Let $1\leq n <N$. Then every $\Psi\in
Aut(B(\cH)^n_1)$ extends to ${\bf \Psi}\in Aut(B(\cH)^N_1)$.
\end{proposition}

\begin{proof}
Let $\lambda=(\lambda_1,\ldots, \lambda_n)\in \BB_n$  and denote $
\lambda_\circ:=(\lambda, 0)\in \BB_N$. Define ${\bf
\Psi}:[B(\cH)^N]_1\to [B(\cH)^N]_1$ by setting
$$
{\bf  \Psi} (X_1,\ldots, X_n, X_{n+1},\ldots, X_N):=
\left(\Psi_\lambda(X_1,\ldots, X_n),
-\Delta_\lambda\left(I-\sum_{i=1}^n \bar \lambda_i
X_i\right)^{-1}[X_{n+1},\ldots, X_n]\right)
$$
for $(X_1,\ldots, X_n, X_{n+1},\ldots, X_N)\in [B(\cH)^N]_1$. Denote
$\tilde X:=(X, X')=(X_1,\ldots, X_n, X_{n+1},\ldots, X_N)$ and note
that
\begin{equation*}
\begin{split}
\Psi_{ \lambda_\circ}(\tilde X)&= \lambda_\circ-\Delta_{
\lambda_\circ}(I-\tilde X  \lambda_\circ^*)^{-1} \tilde X \Delta_{
\lambda_\circ^*}\\
&=(\lambda, 0)-\Delta_\lambda (I-X\lambda^*)^{-1} [X,
X']\left[\begin{matrix}\Delta_{\lambda^*}&0\\
0&I\end{matrix}\right]\\
&=(\lambda-\Delta_\lambda (I-X\lambda^*)^{-1} X\Delta_{\lambda^*},
-\Delta_\lambda (I-X\lambda^*)^{-1}X')\\
&=(\Psi_\lambda(X), -\Delta_\lambda (I-X\lambda^*)^{-1}X')\\
&={\bf \Psi}(\tilde X).
\end{split}
\end{equation*}
This proves that ${\bf \Psi}=\Psi_{ \lambda_\circ}\in
Aut(B(\cH)^N_1)$ and it is an extension of $\Psi_\lambda\in
Aut(B(\cH)^n_1)$. Since any unitary operator on $\CC^n$ extends to a
unitary operator on $\CC^N$, one can use Theorem \ref{automorph}, to
complete the proof.
\end{proof}

We need to recall (see  \cite{Ru}) a few facts concerning the
automorphisms of the unit ball  $\BB_n$. Let $a\in \BB_n$ and
consider $\varphi_a\in
 Aut(\BB_n)$, the automorphism of the unit ball, defined by
 \begin{equation}
 \label{auto}
 \varphi_a(z):= \frac{a-Q_az-s_a (I-Q_a)z}{1-\left<z,a\right>}, \qquad z\in
 \BB_n,
 \end{equation}
where   $Q_0=0$, \, $Q_a
z:=\frac{\left<z,a\right>}{\left<a,a\right>} a$ if $ a\neq 0$, and
$s_a:=(1- \left<a,a\right>)^{1/2}$. The general form of an
automorphism of $\BB_n$ is $\varphi=\omega\circ \varphi_a$ for $a\in
\BB_n$ and a unitary map $\omega$ on $\CC^n$.

If $\Psi\in Aut(B(\cH)^n_1)$ we denote by $\hat\Psi:= \text{\rm
SOT-}\lim_{r\to 1}\Psi(rS_1,\ldots, rS_n)$, the boundary function of
$\Psi$ with respect to $S_1,\ldots, S_n$.  Note that, due to
Proposition \ref{prop-Tl}, we have $\hat \Psi=\Psi(S_1,\ldots,
S_n)$.

In what follows, we show that the automorphisms of the unit ball
$\BB_n$ coincide with the noncommutative Poisson transforms of the
free holomorphic automorphisms of $[B(\cH)^n]_1$ at the the elements
of $\BB_n$, and $Aut(B(\cH)^n_1)\simeq Aut(\BB_n)$. More precisely,
we can prove the following result.

\begin{theorem}  The map
 $\Gamma :Aut(B(\cH)^n_1)\to  Aut(\BB_n)$, defined by
 $$[\Gamma(\Psi)](z):=(P_z\otimes \text{\rm id})[\hat \Psi], \qquad
 z\in \BB_n,
 $$
 is a group isomorphism, where $\hat \Psi$ is the boundary
 function of $\Psi$ with respect to $S_1,\ldots, S_n$, and
  $P_z$ is the noncommutative Poisson transform at $z$.
\end{theorem}

\begin{proof} Let $\Psi\in Aut(B(\cH)^n_1)$ and  $\lambda=\Psi^{-1}(0)\in \BB_n$.
 Then, due to Theorem \ref{automorph},  there exists
a    unitary operator $U\in B(\CC^n)$ such that $\Psi=\Phi_U\circ
\Psi_\lambda$. According  to Proposition \ref{prop-Tl}, the boundary
function
$$\hat \Psi_\lambda=
 \lambda-\Delta_{
\lambda}\left(I_{F^2(H_n)}-\sum_{i=1}^n \bar{{ \lambda}}_i
S_i\right)^{-1} [S_1,\ldots, S_n] \Delta_{{\lambda}^*}
$$
is in  $\cA_n\otimes M_{1\otimes n}$. On the other hand, it is clear
that $\widehat {\Phi_U\circ \Psi_\lambda}=\hat \Psi_\lambda U$. Note
that if $z=(z_1,\ldots, z_n)\in \BB_n$, then the Poisson kernel
$K_z:\CC\to F^2(H_n)$ is an isometry and $z_i=K_z^* S_i K_z$ for
$i=1,\ldots, n$. Hence, using the continuity of the noncommutative
Poisson transform in the operator norm topology and relation
\eqref{auto}, we deduce that
\begin{equation*}
\begin{split}
[\Gamma(\Psi)](z):&= (P_z\otimes \text{\rm id})[\widehat
{\Phi_U\circ \Psi_\lambda}]=
 K_z^*(\hat \Psi_\lambda U)( K_z\otimes I_{\CC^n})\\
&= \left[\lambda-\Delta_{ \lambda}\left(I_{F^2(H_n)}-\sum_{i=1}^n
\bar{{ \lambda}}_i z_i\right)^{-1} [z_1,\ldots, z_n]
\Delta_{{\lambda}^*}\right]U\\
&=\varphi_\lambda(z)U
\end{split}
\end{equation*}
for any $z\in \BB_n$, where $\varphi_\lambda$ is defined by
\eqref{auto}. Therefore, $\Gamma(\Psi)\in Aut(\BB_n)$. Moreover, we
have $[\Gamma(\Psi)](z)=\Psi(z)$, $z\in \BB_n$, which clearly
implies that $\Gamma$ is a homomorphism. Since the surjectivity of
$\Gamma$ was already proved, we assume that $\Gamma(\Psi)=\text{\rm
id}$, where $\Psi=\Phi_U\circ \Psi_\lambda$. Using the calculations
above, we have $\varphi_\lambda(z)U=z$ for any $z\in \BB_n$. Hence,
we deduce that $\lambda=0$ and $U=-I$, which implies $\Psi=\text{\rm
id}$. This completes the proof.
\end{proof}

\bigskip

\section{The noncommutative Poisson transform under the action of
$Aut(B(\cH)^n_1)$}

In this section, we show that the noncommutative Poisson transform
commutes with the action of the automorphism group
$Aut(B(\cH)^n_1)$. This leads to a characterization
 of the unitarily implemented automorphisms  of the
 Cuntz-Toeplitz algebra $C^*(S_1,\ldots, S_n)$ which leave invariant
 the noncommutative disc algebra $\cA_n$, and provides new insight
  into   Voiculescu's group of automorphisms (see \cite{Vo})  of the Cuntz-Toeplitz
 algebra.

We also show that   the unitarily implemented automorphisms of the
noncommutative  disc algebra $\cA_n$   and the noncommutative
analytic Toeplitz algebra $F_n^\infty$, respectively, are determined
by the free holomorphic automorphisms of $[B(\cH)^n]_1$, via the
noncommutative Poisson transforms. We also  provide new proofs  of
some results obtained by Davidson and Pitts \cite{DP2}.

\begin{theorem}
\label{Poisson-auto} Let  $T \in   B(\cK)^n $ and let $\Psi\in
Aut(B(\cH)^n_1)$.   Then the noncommutative Poisson transform has
the following  properties:
\begin{enumerate}
\item[(i)] If $T$ is a row contraction, then
$$ P_{\Psi(T)}[g]=P_T[P_{\hat \Psi}[g] ]\quad \text{ for any }\
 g\in C^*(S_1,\ldots, S_n),
$$ where  $\hat \Psi\in \cA_n\otimes M_{1\times n}$ is the
boundary function of $\Psi$ with respect to $S_1,\ldots, S_n$.
\item[(ii)]
  If $T$ is a completely non-coisometric  row contraction,
 then so is $\Psi(T)$ and
 $$
P_{\Psi(T)}[f]=P_T[P_{\hat \Psi}[f] ]\quad \text{ for any }\
 f\in  F_n^\infty.
$$
\end{enumerate}
\end{theorem}

\begin{proof}
Let $T:=[T_1,\ldots, T_n]\in [B(\cK)^n]_1^-$ and
$\Psi=(\Psi_1,\ldots, \Psi_n)\in Aut(B(\cH)^n_1)$. Due to Theorem
\ref{automorph} and Proposition \ref{prop-Tl}, $\Psi(T)$ is in
$[B(\cK)^n]_1^-$ and the boundary function $\hat \Psi:=\text{\rm
SOT-}\lim_{r\to 1} \Psi(rS_1,\ldots, rS_n)$ is  a row contraction
with entries in the noncommutative disc algebra  $\cA_n$. Denote
$A=(A_1,\ldots, A_n):=\Psi(T)$ and  let
$\hat\Psi=(\hat\Psi_1,\ldots, \hat \Psi_n)$. Note   that
$A_\alpha=\Psi_\alpha(T)=P_T[\hat  \Psi_\alpha]$ for  $\alpha\in
\FF_n^+$. We recall that the  noncommutative Poisson transform
$P_T:C^*(S_1,\ldots, S_n)\to B(\cK)$  is a unital completely
contractive map  such that
$$
P_T[fg^*]=(P_T[f]) (P_T[g])^*, \qquad f, g\in \cA_n,
$$
and $P_T|_{\cA_n}$ is a unital  homomorphism from $\cA_n$ to
$B(\cK)$. Now, it is easy to see that
\begin{equation*}
\begin{split}
P_T[P_{\hat \Psi}[S_\alpha S_\beta^*] ]&= P_T[\hat\Psi_\alpha \hat
\Psi_\beta^*]= P_T[\hat\Psi_\alpha] P_T[\hat
\Psi_\beta]^*\\
&=\Psi_\alpha (T) \Psi_\beta(T)^*=A_\alpha
A_\beta^*=P_{\Psi(T)}[S_\alpha S_\beta^*]
\end{split}
\end{equation*}
for any  $\alpha, \beta\in \FF_n^+$.  Since the polynomials in
$S_\alpha S_\beta^*$, $\alpha, \beta\in \FF_n^+$ are dense in
$C^*(S_1,\ldots, S_n)$ and the noncommutative Poisson transform is
continuous in the operator norm topology, we deduce that $P_{\hat
\Psi}[g]$ is in $ C^*(S_1,\ldots, S_n)$ for any $g\in
C^*(S_1,\ldots, S_n)$ and that (i) holds.

Now, we prove part (ii). First note that if $T:=(T_1,\ldots, T_n)\in
[B(\cK)^n]_1^-$ and $\cM\subset\cK$ is an invariant subspace under
$T_1^*,\ldots, T_n^*$, then
$$
\Psi(P_\cM T_1|_{\cM},\ldots, P_\cM T_n|_{\cM})=
P_\cM\Psi(T_1,\ldots, T_n)|_{\cM^{(n)}}
$$
for any $\Psi\in Aut(B(\cH)^n_1)$. This is due to the structure of
free holomorphic automorphisms (see Theorem \ref{automorph} and
Proposition \ref{prop-Tl}). Denote
$$
\cN_T:=\left\{ h\in \cK:\ \sum_{|\alpha|=k} \|T_\alpha^*
h\|^2=\|h\|^2,\ k=1,2,\ldots\right\}.
$$
According to \cite{Po-isometric}, $\cN_T$ is an invariant subspace
under $T_1^*,\ldots, T_n^*$, and $T':=[P_{\cN_T}
T_1|_{\cN_T},\ldots,P_{\cN_T} T_1|_{\cN_T}]$ is a co-isometry acting
from $\cN_T^{(n)}$ to $\cN_T$.

 Consider the case
$\Psi=\Psi_\lambda$, $\lambda\in \BB_n$. Applying part (ii) of
Theorem \ref{prop-cara}  when $X=Y=T'$, we deduce that
$\Psi_\lambda(T')=P_{\cN_T}\Psi_\lambda(T)|_{\cN_T^{(n)}}$ is a
co-isometry, which shows that $\cN_T \subseteq
\cN_{\Psi_\lambda(T)}$. The same argument applied to
$\Psi_\lambda(T)$ yields $\cN_{\Psi_\lambda(T)}\subseteq
\cN_{\Psi_\lambda(\Psi_\lambda(T))}$. Since
$\Psi_\lambda(\Psi_\lambda(T))=T$, we have
$\cN_{\Psi_\lambda(T)}\subseteq \cN_{T} $. Therefore,
$\cN_T=\cN_{\Psi_\lambda(T)}$, which implies that $T$ is a c.n.c.
row contraction, i.e., $\cN_T=\{0\}$, if and only if
$\Psi_\lambda(T)$  is c.n.c. The case $\Psi=\Phi_U$, where $U$ is a
unitary operator on $\CC^n$ can be treated in a similar manner.
Since, due to Theorem \ref{automorph}, any  free holomorphic
automorphism of $[B(\cH)^n]_1$ is  of the form $\Psi=\Phi_U\circ
\Psi_\lambda$, the first part of (ii) follows.

Let $f\in F_n^\infty$ have the Fourier representation
$\sum_{\alpha\in \FF_n^+} a_\alpha S_\alpha$ and set
\begin{equation}
\label{fr} f_r(S_1,\ldots, S_n):= \sum_{k=0}^\infty
\sum_{|\alpha|=k} a_\alpha r^{|\alpha|} S_\alpha, \quad r\in[0,1),
\end{equation}
 where the convergence is in the operator
norm topology. Since $\Psi(T)$  is c.n.c.,  we can use the
$F_n^\infty$-functional calculus for row contractions to deduce that
\begin{equation*}
\begin{split}
P_{\Psi(T)}[f]&= \text{\rm SOT-}\lim_{r\to 1}K_{A,r} (I_\cH\otimes f
) K_{A,r}\\
&=\text{\rm SOT-}\lim_{r\to 1}\sum_{k=0}^\infty \sum_{|\alpha|=k}
a_\alpha r^{|\alpha|} A_\alpha,
\end{split}
\end{equation*}
where  $A=[A_1,\ldots, A_n]:=\Psi(T)$. On the other hand, since the
boundary function $\hat \Psi=(\hat \Psi_1,\ldots, \hat \Psi_n)$ is a
pure row contraction, we have
\begin{equation*}
\begin{split}
P_{\hat \Psi}[f]&=K_{\hat \Psi}^* (I\otimes f) K_{\hat \Psi}\\
&=\text{\rm SOT-}\lim_{r\to 1}K_{\hat \Psi}^* (I\otimes f_r) K_{\hat
\Psi}\\
&=\text{\rm SOT-}\lim_{r\to 1}\sum_{k=0}^\infty \sum_{|\alpha|=k}
a_\alpha r^{|\alpha|} \hat\Psi_\alpha.
\end{split}
\end{equation*}
Now, since $T$ is c.n.c., the Poisson transform $P_T:F_n^\infty\to
B(\cH)$ is SOT-continuous on bounded sets, and it coincides with the
$F_n^\infty$-functional calculus. Hence,   using the calculations
above and the fact that $P_T[\hat\Psi_\alpha]=A_\alpha$ for
$\alpha\in \FF_n^+$,  we deduce that
\begin{equation*}
\begin{split}
P_T[P_{\hat \Psi}[f]]&=\text{\rm SOT-}\lim_{r\to
1}P_T\left[\sum_{k=0}^\infty \sum_{|\alpha|=k} a_\alpha
r^{|\alpha|} \hat\Psi_\alpha\right]\\
&=\text{\rm SOT-}\lim_{r\to 1}\sum_{k=0}^\infty \sum_{|\alpha|=k}
a_\alpha r^{|\alpha|} A_\alpha\\
&=P_{\Psi(T)}[f]
\end{split}
\end{equation*}
for any $f\in F_n^\infty$.
 This completes the proof.
\end{proof}

A closer look at the proof of Theorem \ref{Poisson-auto}, reveals
the following.

\begin{remark}\label{ext}
Part (i) of Theorem \ref{Poisson-auto} remains valid  if $\Psi $ is
replaced by any contractive free holomorphic function
$\Phi:[B(\cH)^n]_1\to [B(\cH)^n]_1^-$ with the boundary function
$\hat \Phi$ in $\cA_n\otimes M_{1\times n}$. Part (ii) of Theorem
\ref{Poisson-auto} holds true if $T\in [B(\cK)^n]_1$ and  $\Psi $ is
replaced by  any free holomorphic function $\Phi:[B(\cH)^n]_1\to
[B(\cH)^n]_1$ with the boundary function $\hat \Phi \in \cA_n\otimes
M_{1\times n}$  a pure row contraction.
\end{remark}

A remarcable consequence of Thorem \ref{Poisson-auto} is the
following.

\begin{corollary}\label{compo-Pois}
 If $\Psi, \Phi\in Aut(B(\cH)^n_1)$, then
$$
P_{\widehat{\Psi\circ \Phi}}[g]=(P_{\hat \Phi} P_{\hat \Psi})[g]
$$
for any $g$ in the Cuntz-Toeplitz algebra $ C^*(S_1,\ldots, S_n)$,
or $g$ in the noncommutative analytic Toeplitz algebra $F_n^\infty$.
\end{corollary}
\begin{proof} Note that $\widehat{\Psi\circ \Phi}=(\Psi\circ \Phi )(S_1,\ldots,
S_n)=\Psi(\hat \Phi)$. Taking $T=\hat \Phi$ in  Theorem
\ref{Poisson-auto}, the result follows.
\end{proof}

In what follows we characterize the unitarily implemented
automorphisms  of the
 Cuntz-Toeplitz algebra $C^*(S_1,\ldots, S_n)$ which leave invariant
 the noncommutative disc algebra $\cA_n$. We mention that the first
 part of  the theorem is due to Voiculescu \cite{Vo}. Our
 approach is quite different, using noncommutative Poisson
 transforms.
 \begin{theorem}
 \label{Voicu}
 If $\Psi\in Aut([B(\cH)^n]_1)$ and $\hat\Psi$ is its boundary function in $\cA_n\otimes M_{1\times n}$,
 then the noncommutative Poisson
 transform $P_{\hat\Psi}$   is a  unitarily implemented automorphism
 of the Cuntz-Toeplitz algebra $C^*(S_1,\ldots, S_n)$ which leaves invariant
 the noncommutative disc algebra $\cA_n$.

 Conversely, if $\Phi\in Aut_u(C^*(S_1,\ldots, S_n))$ and
 $\Phi(\cA_n)\subset \cA_n$, then there is  $\Psi\in
 Aut([B(\cH)^n]_1$ such that $\Phi=P_{\hat\Psi}$.
Moreover, in this case
 $$\Phi(g)=K_{\hat \Psi}^* ( I_{\cD_{ {\hat\Psi}}}\otimes g) K_{\hat \Psi}, \qquad g\in C^*(S_1,\ldots,
 S_n),
 $$
  and the noncommutative Poisson kernel $K_{\hat \Psi}$ is a
  unitary operator.
\end{theorem}

\begin{proof} Due  to Proposition \ref{prop-Tl}, Theorem   \ref{prop-cara}, and
Theorem \ref{automorph}, the boundary function $\hat \Psi=(\hat
\Psi_1,\ldots, \hat \Psi_n)$ is a pure inner  multi-analytic
operator with respect to $S_1,\ldots, S_n$,  with the property that
$\text{\rm rank}\,(I-\hat \Psi\hat \Psi^*)=1$. Consequently, $\hat
\Psi^* \hat \Psi=I$, and $\Delta_{\hat \Psi}=(I-\hat \Psi\hat
\Psi^*)^{1/2}$ is a rank one orthogonal projection. Since $\hat
\Psi$ is a pure row contraction, the noncommutative Poisson
transform $P_{\hat \Psi}:C^*(S_1,\ldots, S_n)\to B(F^2(H_n))$ is
defined by
$$
P_{\hat \Psi}[g]:=K_{\hat \Psi}^*( I_{\cD_{\hat \Psi}}\otimes g)
K_{\hat \Psi},
$$
where the Poisson kernel $K_{\hat \Psi}:F^2(H_n)\to \cD_{\hat
\Psi}\otimes  F^2(H_n)$ is an isometry. On the other hand, since
$\hat \Psi^* \hat \Psi=I$, the characteristic function $\tilde
\Theta_{\hat \Psi}=0$. According to \eqref{fa}, we have
$I-\tilde{\Theta}_{\hat \Psi} \tilde{\Theta}_{\hat \Psi}^*=K_{\hat
\Psi}K_{\hat \Psi}^*$. Consequently, $K_{\hat \Psi}K_{\hat
\Psi}^*=I$, which implies that $K_{\hat \Psi}$ is a unitary
operator. Since $P_{\hat \Psi}(S_\alpha S_\beta^*)=\hat \Psi_\alpha
\hat \Psi_\beta^*$ for $\alpha, \beta\in \FF_n^+$ and $\hat
\Psi_i\in \cA_n$ for $i=1,\ldots, n$, we deduce that the range of
$P_{\hat \Psi}$ is in $C^*(S_1,\ldots, S_n)$. Therefore the
noncommutative Poisson transform  $P_{\hat \Psi}:C^*(S_1,\ldots,
S_n)\to C^*(S_1,\ldots, S_n)$ is a $*$-representation  such that
$P_{\hat \Psi}(\cA_n)\subset \cA_n$. Due to Corollary
\ref{compo-Pois}, we have $ (P_{\widehat {\Psi^{-1}}} P_{\hat
\Psi})(g)=g$ for any $g\in C^*(S_1,\ldots, S_n)$. This shows that
$P_{\hat \Psi}$ is an automorphism of the Cuntz-Toeplitz algebra
$C^*(S_1,\ldots, S_n)$. Since  $\Delta_{\hat \Psi}$ is a rank one
orthogonal projection and  we can identify the defect space
$\cD_{{\hat \Psi}}$ with $\CC$. Moreover, under this identification,
$P_{\hat \Psi}$ is a unitarily implemented automorphism. Indeed, due
to Theorem \ref{automorph}, $ \Psi=\Phi_U\circ \Psi_\lambda$ for
some unitary operator $U\in B(\CC^n)$ and
$\lambda:=(\lambda_1,\ldots, \lambda)\in \BB_n$. When
$X=Y=[S_1,\ldots, S_n]$ in Theorem \ref{prop-cara}, we deduce that
$$\Delta_{\hat\Psi_\lambda}^2=\Delta_\lambda \left(I-\sum_{i=1}^n \bar
\lambda S_i\right)^{-1} P_\CC \left(I-\sum_{i=1}^n \lambda_i
S_i^*\right)^{-1} \Delta_\lambda.
$$
Therefore, there is a unitary operator $U_\lambda:\cD_{{\hat
\Psi_\lambda}}\to \CC$ defined by
\begin{equation*}
\begin{split}
U_\lambda\Delta_{\hat \Psi_\lambda} f&:=(1-\|\lambda\|_2^2)^{1/2}
P_\CC
\left(I-\sum_{i=1}^n \lambda_i S_i^*\right)^{-1} f\\
&=(1-\|\lambda\|_2^2)^{1/2} f(\lambda)
\end{split}
\end{equation*}
for any $ f\in F^2(H_n)$. Since $\Delta_{\hat \Psi}=\Delta_{\hat
\Psi_\lambda}$, our assertion follows.

Conversely, assume that $\Phi\in Aut_u(C^*(S_1,\ldots, S_n))$ such
that $\Phi(\cA_n)\subset \cA_n$ and let $\varphi_i:=\Phi(S_i)\in
\cA_n$ for $i=1,\ldots, n$. Consequently, there is a unitary
operator $U\in B(F^2(H_n))$ such that  $\Phi(Y)=U^*Y U$ for any
$Y\in C^*(S_1,\ldots, S_n)$. Since $\varphi_i=U^* S_i U$,
$i=1,\ldots,n$, it is easy to see that $\varphi_1,\ldots, \varphi_n$
are isometries with orthogonal ranges and
$$
I-\sum_{i=1}^n\varphi_i\varphi_i^* =U^*\left(I-\sum_{i=1}^n
S_iS_i^*\right)U=U^*P_\CC U
$$
is a rank one projection. Since $\varphi_i\in \cA_n$ it is clear
that the subspace $\cM:=\oplus_{i=1}^n \varphi_i (F^2(H_n))\subset
F^2(H_n)$ is invariant under each  right creation operator
$R_1,\ldots, R_n$, and has codimension one. According  to
\cite{ArPo}, $\cM^\perp =\CC z_\lambda$ for a unique $\lambda\in
\BB_n$, where
$$
z_{\lambda}:=\sum_{\alpha\in \FF_n^+} \overline {\lambda}_{\alpha}
e_\alpha.
$$
Note that $\|z_\lambda\|=\frac{1}{\sqrt{1-\|\lambda\|_2^2}}$ and
denote $u_\lambda:=\frac{z_\lambda}{\|z_\lambda\|}$.  Since
$\lambda=(\lambda_1,\ldots, \lambda_n)\in \BB_n$ is a  pure row
contraction, the noncommutative Poisson kernel $K_\lambda:\CC\to
F^2(H_n)$ satisfies the equation
$$
K_\lambda(1)=\sum_{\alpha\in \FF_n^+}
(1-\|\lambda\|_2^2)^{1/2}\bar\lambda_\alpha\otimes e_\alpha
=u_\lambda.
$$
Consequently $K_\lambda K_\lambda^*$ is the projection of $F^2(H_n)$
onto $\CC u_\lambda$. On the  other hand, due to  relation
\eqref{fa}, we have $I-\tilde\Theta_\lambda
\tilde\Theta_\lambda^*=K_\lambda K_\lambda^*$. Consequently
$\cM=\tilde\Theta_\lambda\left(F^2(H_n)^{(n)}\right)$. Using the
uniqueness, up to a unitary equivalence,  in  the Beurling type
theorem characterizing the invariant subspaces of the right creation
operators (see \cite{Po-charact}), we find a unitary operator $U\in
B(\CC^n)$ such that $\varphi:=[\varphi_1,\ldots, \varphi_n]=
\tilde\Theta_\lambda U$. Setting $\Psi:=\Phi_U\circ \Theta_\lambda$,
we have $\Psi\in Aut(B(\cH)^n_1)$ and $\hat \Psi=\varphi$. Using the
first part of the proof, we deduce that $P_{\hat\Psi}$ is a
unitarily implemented automorphism of $C^*(S_1,\ldots, S_n)$ with
$P_{\hat\Psi}(\cA_n)\subset \cA_n$. Moreover, since
$P_{\hat\Psi}(S_i)=\varphi_i= \Phi(S_i)$ for any $i=1,\ldots, n$,
and $C^*(S_1,\ldots, S_n)$ is generated by the left creation
operators,  the continuity of $P_{\hat\Psi}$ and $\Phi$ in the
operator norm topology implies   $P_{\hat\Psi}=\Phi$. This completes
the proof.
\end{proof}

We recall (see  \cite{ArPo}) that $
S_i^*z_{\lambda}=\overline{\lambda}_{i} z_{\lambda}$ for
$i=1,\ldots, n.$ In this case, we have $\lambda_i=K_\lambda^* S_i
K_\lambda$, $i=1,\ldots, n$, and the one dimensional co-invariant
subspaces $\cM\subset F^2(H_n)$ under $S_1,\ldots, S_n$ are of the
form $\cM=\CC u_\lambda$, $\lambda\in \BB_n$.
This shows that the unit ball $\BB_n$ coincides with the
compressions $(P_\cM S_1|_\cM,\ldots, P_\cM S_n|_\cM)$ of the left
creation operators to the one dimensional co-invariant subspaces
under $S_1,\ldots, S_n$.  This  gives another indication  why there
is such a close connection between the the noncommutative ball
$[B(\cH)^n]_1$ (resp.~free holomorphic functions) and the unit  ball
$\BB_n$ (resp.~analytic functions).

In what follows we show that the unitarily implemented automorphisms
of the noncommutative  disc algebra $\cA_n$   and the noncommutative
analytic Toeplitz algebra $F_n^\infty$, respectively, are determined
by the free holomorphic automorphisms of $[B(\cH)^n]_1$, via the
noncommutative Poisson transforms.

 We  mention  that Davidson and Pitts  showed  in \cite{DP2} that the
 subgroup $Aut_u(F_n^\infty)$ of  unitarily implemented
automorphisms of $F_n^\infty$  is isomorphic with $Aut(\BB_n)$. In
what follows (see the next theorem and its corollaries) we  obtain a
new proof of their result, using
 noncommutative Poisson transforms,  which extends
to  the noncommutative disc algebra $\cA_n$.

 \begin{theorem}
 \label{disc-auto}
 If $\Psi\in Aut(B(\cH)^n_1)$ and $\hat\Psi$ is its boundary function,
 then the noncommutative Poisson
 transform $P_{\hat\Psi}|_{\cA_n}$   is a  unitarily implemented automorphism
 of  the noncommutative disc algebra $\cA_n$.

 Conversely, if $\Phi\in Aut_u( \cA_n)$,  then there is   $\Psi\in
 Aut([B(\cH)^n]_1$ such that $\Phi=P_{\hat\Psi}|_{\cA_n}$.
Moreover, in this case
 $$\Phi(f)=K_{\hat \Psi}^* ( I_{\cD_{ {\hat\Psi}}}\otimes f) K_{\hat \Psi}, \qquad f\in  \cA_n,
 $$
  and the noncommutative Poisson kernel $K_{\hat \Psi}$ is a
  unitary operator.
\end{theorem}

\begin{proof} Assume that $\Psi\in Aut(B(\cH)^n_1)$. According to
Theorem \ref{Voicu}, the noncommutative Poisson transform
$P_{\hat\Psi}$ is a unitarily implemented automorphism of the
Cuntz-Toeplitz algebra $C^*(S_1,\ldots, S_n)$ such that
$P_{\hat\Psi}(\cA_n)\subset \cA_n$. Due to Corollary
\ref{compo-Pois}, we have $ (P_{\widehat {\Psi^{-1}}} P_{\hat
\Psi})(f)=f$ for any $f\in \cA_n$. Hence, we deduce that
$P_{\hat\Psi}|_{\cA_n}$ is a unitarily implemented automorphism of
the noncommutative disc algebra $\cA_n$.

 Conversely, let $\Phi\in Aut_u(\cA_n)$. As in the proof of Theorem
 \ref{Voicu}, we find $\Psi\in Aut(B(\cH)^n_1)$ such that
 $\Phi=P_{\hat\Psi}|_{\cA_n}$.
 This completes the proof.
\end{proof}

\begin{corollary}
 \label{DP}
 If $\Psi\in Aut(B(\cH)^n_1)$ and $\hat\Psi$ is its boundary function,
 then the noncommutative Poisson
 transform $P_{\hat\Psi}$   is a  unitarily implemented automorphism
 of  $F_n^\infty$.

 Conversely, if $\Phi\in Aut_u( F_n^\infty)$,  then there is  $\Psi\in
 Aut([B(\cH)^n]_1$ such that $\Phi=P_{\hat\Psi}$.
Moreover, in this case
 $$\Phi(f)=K_{\hat \Psi}^* (I_{\cD_{ {\hat\Psi}}}\otimes f) K_{\hat \Psi}, \qquad f\in  F_n^\infty,
 $$
  and the noncommutative Poisson kernel $K_{\hat \Psi}$ is a
  unitary operator.
\end{corollary}

\begin{proof} If  $\Psi\in Aut(B(\cH)^n_1)$, the $\hat \Psi\in
\cA_n\otimes M_{1\times n}$.  Since $\hat \Psi$ is a pure  row
contraction, it makes sense to consider the noncommutative Poisson
transform $P_{\hat \Psi}: F_n^\infty\to B(F^2(H_n))$  defined by
$$
P_{\hat \Psi}[f]:=K_{\hat \Psi}^* (I_{\cD_{ {\hat\Psi}}}\otimes f)
K_{\hat \Psi}, \qquad f\in  F_n^\infty.
 $$
We saw in the proof of Theorem \ref{Voicu} that $K_{\hat \Psi}$ is a
unitary operator. We recall that if $f\in F_n^\infty$, then $f_r\in
\cA_n$,  $\|f_r\|\leq \|f\|$, and $\text{\rm SOT-}\lim_{r\to 1}
f_r=f$. Since $ P_{\hat \Psi}(S_\alpha)=\hat \Psi_\alpha\in \cA_n$,
$\alpha\in \FF_n^+$,  and $F_n^\infty$ is the WOT-closed
non-selfadjoint algebra generated by $S_1,\ldots, S_n$ and the
identity, we deduce that $P_{\hat \Psi}(F_n^\infty)\subset
F_n^\infty$. On the other hand, due to Corollary \ref{compo-Pois},
we have $ (P_{\widehat {\Psi^{-1}}} P_{\hat \Psi})(f)=f$ for any
$g\in  F_n^\infty$. This shows that $P_{\hat
\Psi}(F_n^\infty)=F_n^\infty$  and therefore $P_{\hat \Psi}$ is a
unitarily implemented  automorphism of $F_n^\infty$.

Conversely, let $\Phi\in Aut_u( F_n^\infty)$. As in the proof of
Theorem
 \ref{Voicu}, we find $\Psi\in Aut(B(\cH)^n_1)$ such that
 $\Phi|_{\cA_n}=P_{\hat\Psi}|_{\cA_n}$.
 Since $\cA_n$ is $w^*$-dense in $F_n^\infty$ and both $\Psi$ and
 $P_{\hat\Psi}$  are unitarily implemented (therefore
 $w^*$-continuous), we deduce that $\Phi=P_{\hat\Psi}$.
This completes the proof.
\end{proof}

\begin{corollary}
$
 Aut(B(\cH)^n_1)\simeq  Aut_u(\cA_n)\simeq Aut_u(F_n^\infty)$.
\end{corollary}
\begin{proof}
Let $\Gamma: Aut(B(\cH)^n_1)\to Aut_u(\cA_n)$ be defined by
$$
\Gamma(\Psi):=P_{\widehat {\Psi^{-1}}}|_{\cA_n},\qquad \Psi \in
Aut(B(\cH)^n_1).
$$
Due to Theorem  \ref{disc-auto}, the map $\Gamma$ is well-defined
and surjective.  On the other hand, if we assume that
$\Gamma(\Psi)=\text{\rm id}$, then $P_{\widehat
{\Psi^{-1}}}(S_i)=S_i $ for  $i=1,\ldots, n$ and therefore $\widehat
{\Psi^{-1}}=[S_1,\ldots, S_n]$. Consequently, $\Psi^{-1}(X)=X$ for
any $X\in [B(\cH)^n]_1$, which proves that $\Gamma$ is one-to-one.
Using Corollary \ref{compo-Pois},  one can deduce that $\Gamma $ is
a group homomorphism and   $Aut(B(\cH)^n_1)\simeq Aut_u(\cA_n)$.
 In a similar manner, using Corollary \ref{DP} and Corollary
 \ref{compo-Pois}, one can show that $Aut(B(\cH)^n_1)\simeq
 Aut_u(F_n^\infty)$.
This completes the proof.
\end{proof}

We say that a free holomorphic function $F$ on $[B(\cH)^n]_1$  is
bounded if
$$
\|F\|_\infty:=\sup  \|F(X_1,\ldots, X_n)\|<\infty,
$$
where the supremum is taken over all $n$-tuples  of operators
$[X_1,\ldots, X_n]\in [B(\cH)^n]_1$. Let $H^\infty(B(\cH)^n_1)$ be
the set of all bounded free holomorphic functions and   denote by
$A(B(\cH)^n_1)$   the set of all  elements $F$  in $Hol(B(\cH)^n_1)$
such that the mapping
$$[B(\cH)^n]_1\ni (X_1,\ldots, X_n)\mapsto F(X_1,\ldots, X_n)\in B(\cH)$$
 has a continuous extension to the closed unit ball $[B(\cH)^n]^-_1$.
We  showed in \cite{Po-holomorphic} that $H^\infty(B(\cH)^n_1)$  and
$A(B(\cH)^n_1)$ are Banach algebras under pointwise multiplication
and the norm $\|\cdot \|_\infty$, which can be identified with the
noncommutative analytic Toeplitz algebra $F_n^\infty$ and the
noncommutative disc algebra $\cA_n$, respectively.

\begin{lemma}
\label{comp-iso} Let $f :[B(\cH)^n]_1 \to B(\cH)$  and
$g:[B(\cH)^n]_1\to [B(\cH)^n]_1$ be free holomorphic functions. Then
the following properties hold.

 \begin{enumerate}
 \item[(i)] If $f $ and $g$ have continuous extension to the closed ball $[B(\cH)^n]_1$ , then
 $f\circ g\in A(B(\cH)^n_1)$.
\item[(ii)] If $f\in H^\infty(B(\cH)^n_1)$, then
$f\circ g\in H^\infty(B(\cH)^n_1)$   and $\|f\circ g\|_\infty\leq
\|f\|_\infty$.
\item[(iii)] If $f\in H^\infty(B(\cH)^n_1)$ and  $\hat g$ is  a pure row contraction
 with entries in $\cA_n$, then $$(f\circ
g)(X)=P_X\left[ P_{\hat g}\left[\hat f\right]\right]$$
 for any  $X\in [B(\cH)^n]_1$, where   $\hat f$ and $\hat g$ are the
  corresponding boundary
functions.
 \end{enumerate}
\end{lemma}
\begin{proof} Due to Theorem  \ref{compo},
$f\circ g$ is a free holomorphic function on $[B(\cH)^n]_1$.  Part
(i) is obvious. To prove (ii), assume that  $f\in
H^\infty(B(\cH)^n_1)$.  In this case, we have
 $$
\|f\circ g\|_\infty= \sup_{X\in [B(\cH)^n]_1} \|f(g(X))\|\leq
\sup_{Y\in [B(\cH)^n]_1}\|f(Y)\|=\|f\|_\infty,
$$
which shows that $f\circ g\in H^\infty(B(\cH)^n_1)$.
 Now, assume that $f\in H^\infty(B(\cH)^n_1)$ and $g\in
A(B(\cH)^n_1)$. Due to the fact that  $\text{\rm range}\, g\subseteq
[B(\cH)^n]_1$, we have $\|g(rS_1,\ldots, rS_n)\|<1$ (see Lemma
\ref{range}). Since $f\circ g\in H^\infty(B(\cH)^n_1)$, its boundary
function $\widehat {f\circ g}$  exists and
\begin{equation} \label{fcirc} \widehat {f\circ g} = \text{\rm
SOT-}\lim_{r\to 1}f( g(rS_1,\ldots, rS_n)).
\end{equation}
Using the noncommutative Poisson transform  and Remark \ref{ext}, we
deduce that
\begin{equation}
\label{fP}
 f( g(rS_1,\ldots, rS_n))=P_{g(rS_1,\ldots, rS_n)}
[\hat f] =P_{[rS_1,\ldots, rS_n]}\left[ P_{\hat g}[\hat f]\right].
\end{equation}
We recall  that if $g\in F_n^\infty$, then $g=\text{\rm
SOT-}\lim_{r\to 1} g_r$ where  $g_r=P_{[rS_1,\ldots, rS_n]} [g]$.
Applying this result to  $ P_{\hat g}[\hat f]$, which is in
$F_n^\infty$, and using relations \eqref{fcirc} and \eqref{fP}, we
deduce that $\widehat {f\circ g}=P_{\hat g}[\hat f]$. Since $(f\circ
g)(X)=P_X[\widehat {f\circ g}]$, \ $X\in [B(\cH)^n]_1$, we complete
the proof.

\end{proof}

\begin{corollary}
\label{comp-iso2} Let $f\in Hol(B(\cH)^n_1)$ and  $\Psi\in
 Aut(B(\cH)^n_1)$. Then
 \begin{enumerate}
 \item[(i)] $f\circ \Psi\in A(B(\cH)^n_1)$ for $f\in A(B(\cH)^n_1)$;
\item[(ii)] $f\circ \Psi\in H^\infty(B(\cH)^n_1)$ for $f\in H^\infty(B(\cH)^n_1)$;
\item[(iii)] $\|f\circ \Psi\|_\infty=\|f\|_\infty$ for $f\in H^\infty(B(\cH)^n_1)$;
\item[(iv)] if $f\in H^\infty(B(\cH)^n_1)$, then $(f\circ \Psi)(X)
=P_X\left[ P_{\hat \Psi}\left[\hat f\right]\right]$
 for any  $X\in [B(\cH)^n]_1$, where $\hat \Psi$ and $\hat f$ are the
  corresponding boundary
functions.
 \end{enumerate}
\end{corollary}
\begin{proof} Apply Lemma \ref{comp-iso} in the particular case when
$g=\Psi\in
 Aut(B(\cH)^n_1)$.
Due to   part (ii) of the same  lemma, we get $\|f\circ
\Psi\|_\infty\leq\|f\|_\infty$. Applying the same inequality to
$f\circ \Psi$ and the automorphism $\Psi^{-1}$, we obtain
$\|f\|_\infty\leq \|f\circ \Psi\|_\infty$. Therefore, $\|f\circ
\Psi\|_\infty=\|f\|_\infty$, which completes the proof.
\end{proof}

We  remark that one can obtain  versions of Lemma \ref{comp-iso} and
Corollary \ref{comp-iso2} when $f$ has operator-valued coefficients,
i.e., $f :[B(\cH)^n]_1 \to B(\cH)\bar\otimes B(\cE, \cG)$. The proof
is basically the same.

We proved in \cite{Po-holomorphic} that there is a completely
isometric isomorphism
$$
A(B(\cH)^n_1)\ni f\mapsto \hat f:=\lim_{r\to 1} f(rS_1,\ldots,
rS_n)\in \cA_n,
$$
where the limit is in the operator norm topology,  whose inverse is
the noncommutative Poisson  transform $P$, i.e.,   $f(X)=P_X[\hat
f]$, $X\in [B(\cH)^n]_1$. If $\Phi: A(B(\cH)^n_1)\to A(B(\cH)^n_1)$
is a homomorphism, it induces a unique homomorphism $\hat
\Phi:\cA_n\to \cA_n$ such that the  diagram
\begin{equation*}
\begin{matrix}
\cA_n& \mapright{\hat\Phi}&
\cA_n\\
 \mapdown{P}& & \mapdown{P}\\
A(B(\cH)^n_1)&  \mapright{\Phi}& A(B(\cH)^n_1)
\end{matrix}
\end{equation*}
is commutative, i.e., $\Phi P=P\hat \Phi$. The homomorphism $\Phi$
and $\hat \Phi$ uniquely determine each other by the formulas:
\begin{equation*}
\begin{split}
(\Phi f)(X)&=P_X[\hat \Phi(\hat f)], \qquad  f\in A(B(\cH)^n_1),\
X\in [B(\cH)^n]_1, \quad \text{ and }\\
\hat \Phi(\hat f)&=\widehat{\Phi(f)}, \qquad \hat f\in \cA_n.
\end{split}
\end{equation*}
Similar results hold for the Hardy algebra $H^\infty(B(\cH)^n_1)$
and the noncommutative analytic algebra $F_n^\infty$, when $\hat
f:=\text{\rm SOT-}\lim_{r\to 1} f(rS_1,\ldots, rS_n)$ for $f\in
H^\infty(B(\cH)^n_1)$.

Using a  {\it gliding bump}  argument as in \cite{DP2} (see Lemma
4.2), one can  show that any automorphism $\hat \Phi$ of the
noncommutative disc algebra $\cA_n$ is continuous. We sketch the
proof for completeness.

 Indeed, suppose that $\hat \Phi$ is not
continuous.  Let  $M:=\max\{\|\psi_1\|, \|\psi_2\|, 1\}$,  where
$\psi_j:=\hat\Phi^{-1}(S_j)\in \cA_n$, $j=1,2$. Then one can find  a
sequence $\{f_k\}$ of elements in $\cA_n$ such that $\|f_k\|\leq
\frac{1}{(2M)^k}$ and $\|\hat \Phi (f_k)\|\geq k$ for any $k\in
\NN$.
Since $f:=\sum_{k=1}^\infty \psi_1^k\psi_2 f_k$  is convergent in
the operator norm, it is an element in $\cA_n$. For each $m\in \NN$,
we have $\hat\Phi(f)=\sum_{k=1}^m S_1^k S_2\hat \Phi(f_k)+ S_1^{m+1}
\hat\Phi(g_m)$ for some $g_m\in \cA_n$. Using the fact that $S_1$,
$S_2$ are isometries with orthogonal ranges, we have
$$
\|\hat \Phi(f)\|\geq \|S_2^* {S_1^*}^m \hat \Phi(f)\|=\|\hat
\Phi(f_m)\|\geq m \quad  \text{ for any } \ m\in \NN,
$$
  which is a contradiction. Therefore $\hat\Phi$ is
continuous.

  We recall that Davidson and Pitts \cite{DP2}
  proved that the  group of completely isometric  automorphisms of
  $F_n^\infty$ can be identified with $Aut(\BB_n)$. It is easy to see that their result
  extends to $A(B(\cH)^n_1)$.
In what follows, we  obtain a new proof of their result and new
characterizations of the   completely isometric automorphisms of the
noncommutative disc algebra
  $A(B(\cH)^n_1)$.

 \begin{theorem}
 \label{auto-disk}
 Let  $\Phi:A(B(\cH)^n_1)\to A(B(\cH)^n_1)$ be an automorphism
 of
  $A(B(\cH)^n_1)$. Then the following statements are equivalent:
  \begin{enumerate}
  \item[(i)]
$\Phi$ is a
 completely isometric  automorphism of
  $A(B(\cH)^n_1)$;
    \item[(ii)]
 there is a   $\Psi\in
 Aut(B(\cH)^n_1)$ such that
 $$
 \Phi(f)=f\circ \Psi,\quad f\in A(B(\cH)^n_1);
 $$
 \item[(ii)] $[\hat \Phi(S_1),\ldots, \hat \Phi(S_n)]$ and
  $[\hat\Phi^{-1}(S_1),\ldots, \hat \Phi^{-1}(S_n)]$ are
  row contractions.
\end{enumerate}
\end{theorem}

\begin{proof} First, we prove that $(i)\leftrightarrow (ii)$. Assume
that  (i) holds. Let $\hat\psi_i:=\hat \Phi(S_i)\in \cA_n\subset
F_n^\infty$, $i=1,\ldots, n$, and note that $\hat \psi:=[\hat
\psi_1,\ldots, \hat \psi_n]$ is a  a row  isometry.   The subspaces
$\hat \psi_1(F^2(H_n))$ and $\oplus_{j=2}^n \hat \psi_j(F^2(H_n))$
are invariant under the right creation operators $R_1,\ldots, R_n$.
If we  suppose that
\begin{equation}
\label{Cuntz} \sum_{j=1}^n \hat \psi_j \hat \psi_j^*=I,
\end{equation}
then $\hat \psi_1(F^2(H_n))$ is a reducing subspace  for
$R_1,\ldots, R_n$. Since $R_1,\ldots, R_n$ have no nontrivial
reducing subspaces, we deduce that \eqref{Cuntz} doesn't hold if
$n\geq 2$. If $n=1$, then $\psi_1$ is a unitary operator and,
consequently, a constant which is a contradiction with $\hat
\Phi(I)=I$ and the fact that $\hat \Phi$ is one-to-one. Therefore,
since $\sum_{j=1}^n \hat \psi_j \hat \psi_j^*\neq I$, we deduce that
$1\notin \oplus_{j=1}^n \psi_j (F^2(H_n))$, which implies
$$
\sum_{j=1}^m |\hat\psi_j(0)|^2= \left< \left(\sum_{j=1}^m \hat\psi_j
\hat\psi_j^*\right)1,1\right><1.
$$
Setting $\lambda:=(\hat \psi_1(0),\ldots, \hat \psi_n(0))\in \BB_n$,
one can prove  that $\hat \psi:=[\hat \psi_1,\ldots, \hat \psi_n]$
is a pure row contraction. The proof is similar to the proof of the
fact that $\tilde \Theta_\lambda$ is pure  in Theorem \ref{prop-Tl}.

Now, let $\Psi:=(\psi_1,\ldots, \psi_n)$ be the unique contractive
free holomorphic function  on $[B(\cH)^n]_1$ having the boundary
function $\hat \psi$. Due to Theorem \ref{prop-Tl} and Theorem
\ref{prop-cara},  we deduce that the map $\Gamma:=\Psi_\lambda
\circ\Psi:[B(\cH)^n]_1\to [B(\cH)^n]_1^-$ is a free holomorphic
function with $\Gamma(0)=0$, where $\Psi_\lambda\in
Aut(B(\cH)^n_1)$. Applying the noncommutative Schwartz type lemma
from \cite{Po-holomorphic} (see Section 5 for a stronger version),
we deduce that $\|\Gamma(X)\|\leq \|X\|$ for any $X\in
[B(\cH)^n]_1$. Now, by  Theorem \ref{prop-cara}, we obtain
$$
\|\Psi(X)\|=\|\Psi_\lambda(\Gamma (X))\|\leq \|\Gamma(X)\|\leq \|X\|
$$
for $X\in [B(\cH)^n]_1$. Hence, $\Psi:[B(\cH)^n]_1\to [B(\cK)^n]_1$
is a free holomorphic function.

Similarly, setting $\hat \varphi_i:=\hat\Phi^{-1}(S_i)$,
$i=1,\ldots,n$, and letting  $\varphi=(\varphi_1,\ldots, \varphi_n)$
be the free holomorphic function on $[B(\cH)^n]_1$ with boundary
function $(\hat\varphi_1,\ldots, \hat\varphi_n)$, one can prove that
$\varphi:[B(\cH)^n]_1\to [B(\cH)^n]_1$ is a free holomorphic
function.

Since $\hat \varphi_i\in \cA_n$, let $p_k(S_1,\ldots, S_n)$ be a
sequence of polynomials in $S_1,\ldots, S_n$ such that $$\hat
\varphi_i=\lim_{k\to \infty} p_k(S_1,\ldots, S_n)$$ in the operator
norm topology. Due to the continuity in norm of $\Phi$ and the
Poisson transform $P_{[\hat \psi_1,\ldots, \hat \psi_n]}$, we have
\begin{equation*}
\begin{split}
S_i&=\hat\Phi(\hat\varphi_i)\\
&=\lim_{k\to \infty} p_k(\hat\Phi(S_1),\ldots,
\hat\Phi(S_n))\\
&=\lim_{k\to \infty} P_{[\hat \psi_1,\ldots, \hat \psi_n]}
[p_k(S_1,\ldots, S_n)]\\
&=P_{[\hat \psi_1,\ldots, \hat \psi_n]}[\hat \varphi_i]
\end{split}
\end{equation*}
for any $i=1,\ldots, n$.  On the other hand, since $\hat \Psi=[\hat
\psi_1,\ldots, \hat \psi_n]$ is a pure row contraction, we can use
Remark \ref{ext} to deduce that
\begin{equation*}
\begin{split}
(\varphi_i\circ \Psi)(X)= P_X\left[P_{\hat \psi}[\hat
\varphi_i]\right] =P_X[S_i]=X_i, \quad i=1,\ldots,n,
\end{split}
\end{equation*}
for any $X=[X_1,\ldots, X_n]\in [B(\cH)^n]_1$. Consequently, $
\varphi\circ \Psi=\text{\rm id}$. Similar arguments imply   $
 \Psi\circ \varphi=\text{\rm id}$. Therefore, $\Psi\in
Aut(B(\cH)^n_1)$.

If $p(X)=\sum_{|\alpha|\leq k} a_\alpha X_\alpha$, then $p\circ
\psi=\sum_{|\alpha|\leq k} a_\alpha \psi_\alpha=\Phi(p)$. Due to
Corollary \ref{comp-iso2}, the map $f\mapsto f\circ \Psi$ is
continuous on $H^\infty(B(\cH)^n_1)$.  Since the noncommutative
polynomials are dense in $H^\infty(B(\cH)^n_1)$ and $\Phi$ is also
continuous, we deduce that $\Phi(f)=f\circ \Psi$ for any $f\in
H^\infty(B(\cH)^n_1)$. This completes the proof of the implication
$(i)\implies (ii)$.

Now, assume that  $\Psi\in
 Aut(B(\cH)^n_1)$  and
 $
 \Phi(f)=f\circ \Psi,\quad f\in A(B(\cH)^n_1)$.
Note that Proposition \ref{prop-Tl} and Theorem  \ref{automorph}
imply that the boundary function $\hat \Psi:=(\hat \psi_1,\ldots,
\hat \psi_n)$ is a row isometry with $\hat\psi_i\in \cA_n$. Using
Corollary \ref{comp-iso2}, we have $(f\circ \Psi)(X)=P_X\left[
P_{\hat \Psi}[\hat f]\right]$ for $X\in [B(\cH)^n]_1$. Since the
noncommutative Poisson transform $P_{\hat \Psi}$ is continuous in
norm, $\hat \Psi, \hat f\in \cA_n$, and the polynomials in
$S_1,\ldots, S_n$  are  norm dense in $\cA_n$, we deduce that
$P_{\hat \Psi}[\hat f]$ is in $\cA_n$. Consequently, $\Phi(f)$ is a
well-defined homomorphism. As in the proof of Lemma \ref{comp-iso},
one can show that $\Phi$ is completely isometric. A similar result
can be obtained for  the map $\Lambda(f):=f\circ\Psi^{-1}$, $f\in
A(B(\cH)^n_1)$. Since $\Phi\circ \Lambda=\Lambda\circ \Phi=id$, we
deduce that $\Phi$ is a completely isometric automorphism of
$A(B(\cH)^n_1)$. Therefore $(i)\leftrightarrow (ii)$.

Since the implication $(i)\implies (iii)$ is obvious, it remains to
prove that $(iii)\implies (i)$. Assume that $[\hat \Phi(S_1),\ldots,
\hat \Phi(S_n)]$  is a row contraction. Due to the noncommutative
von Neumann inequality \cite{Po-von}, we have
$$
\|[\hat\Phi(p_{ij}(S_1,\ldots, S_n))]_{k\times k}\|=
\|[p_{ij}(\hat\Phi(S_1),\ldots,\hat\Phi( S_n)]_{k\times k}\| \leq
\|[ p_{ij}(S_1,\ldots, S_n)]_{k\times k}\|
$$
for any operator matrix  $[p_{ij}(S_1,\ldots, S_n)]_{k\times
k}\in\cA_n\otimes M_{k\times k}$. Since $\hat \Phi$ is continuous on
$\cA_n$ (see the remarks preceding Theorem \ref{auto-disk}), which
is the norm closed algebra generated by $S_1,\ldots, S_n$ and the
identity, we deduce that $\hat\Phi:\cA_n\to \cA_n$ is a completely
contractive homomorphism. Similarly, if $[\hat\Phi^{-1}(S_1),\ldots,
\hat \Phi^{-1}(S_n)]$ is
 a row contraction , we can prove that $\hat \Phi^{-1}$ is
 completely contractive. Therefore, $\hat\Phi$ is a complete
 isometry and (i) holds.
The proof is complete.
\end{proof}

Using the ideas from the proof of Theorem \ref{auto-disk}, we can
prove the following result.

 \begin{theorem}
 \label{auto-hardy}
 Let  $\Phi: H^\infty(B(\cH)^n_1)\to H^\infty(B(\cH)^n_1)$  be a
 WOT-continuous   homeomorphism  of the noncommutative  Hardy algebra
 $H^\infty(B(\cH)^n_1)$.
  Then the following statements are equivalent:
  \begin{enumerate}
  \item[(i)]
$\Phi$ is a
 completely isometric  automorphism of
   $H^\infty(B(\cH)^n_1)$;
    \item[(ii)]
 there is a   $\Psi\in
 Aut(B(\cH)^n_1)$ such that
 $$
 \Phi(f)=f\circ \Psi,\quad f\in H^\infty(B(\cH)^n_1);
 $$
 \item[(iii)] $[\hat \Phi(S_1),\ldots, \hat \Phi(S_n)]$ and
  $[\hat\Phi^{-1}(S_1),\ldots, \hat \Phi^{-1}(S_n)]$ are
   row contractions.
\end{enumerate}
\end{theorem}
\begin{proof} The proof follows the lines of the proof of Theorem \ref{auto-disk}.
We only mention the differences.  The proof of the implication
$(i)\implies (ii)$ is the same but uses, in addition, the fact that
$F_n^\infty$ is the WOT closed algebra generated by $S_1,\ldots,
S_n$ and the identity, and that the noncommutative Poisson transform
$P_{[\hat \psi_1,\ldots, \hat \psi_n]}$ coincide with the
$F_n^\infty$-functional calculus for pure row contractions and,
therefore,  is  WOT continuous. For the proof of the implication
$(ii)\implies (i)$, note, in addition,  that  $P_{\hat\Psi}[\hat f]$
is in $F_n^\infty$ for any $\hat f\in F_n^\infty$, and  and the
Poisson transform $P_{\hat \Psi}$ is WOT-continuous. Since the
implication $(i)\implies (iii)$ is obvious, it remains to prove that
$(iii)\implies (i)$. Assume that $[\hat \Phi(S_1),\ldots, \hat
\Phi(S_n)]$  is a row contraction. Note that the map
$$\chi(f):=\left<\hat\Phi(f)1,1\right>,\qquad f\in F_n^\infty,
$$
is a  nonzero WOT-continuous multiplicative functional. According to
\cite{DP2}, there exists $\lambda\in \BB_n$ such that
$\chi(f)=\left< f u_\lambda, u_\lambda\right>=P_\lambda[f]$, \ $f\in
F_n^\infty$. Therefore, setting $\hat\psi_i:=\hat \Phi(S_i)\in
F_n^\infty$, $i=1,\ldots, n$,  we have
$$
(\hat \psi_1(0),\ldots, \hat \psi_n(0))=(\lambda_1,\ldots,
\lambda_n)\in \BB_n.
$$
As in the proof of the implication $(i)\implies (ii)$ of Theorem
\ref{auto-disk},  one can prove  that $\hat \Psi:=[\hat
\psi_1,\ldots, \hat \psi_n]$ is a pure row contraction. Now, using
the fact that the noncommutative Poisson transform at a pure row
contraction is WOT-continuous, we can show, as in the proof of
Theorem \ref{auto-disk}, that $\hat\Phi:F_n^\infty\to F_n^\infty$ is
a completely contractive  homomorphism. Similarly, one can prove
that $\hat\Phi^{-1}:F_n^\infty\to F_n^\infty$ is a completely
contractive homomorphism, which shows that $\hat \Phi$ is  a
completely isometric automorphism of $F_n^\infty$. This completes
the proof.
\end{proof}

We mention that  Theorem \ref{auto-disk} and Theorem
\ref{auto-hardy} imply, in the particular case when $n=1$,    the
classical results \cite{H} concerning the conformal automorphisms of
the disc algebra and Hardy algebra $H^\infty$, respectively.

 We  also remark that Davidson
and Pitts proved in \cite{DP2} that any automorphism of $F_n^\infty$
is WOT continuous. Due to their result, we can remove the WOT
continuity from Theorem \ref{auto-hardy}. Moreover, according to
\cite{DP2}, any contractive automorphism of $F_n^\infty$ is
completely isometric. If we combine this with Theorem
\ref{auto-hardy}, we can also deduce the following result.
\begin{corollary} If $\hat \Phi\in Aut(F_n^\infty)$ such that
$[\hat \Phi(S_1),\ldots, \hat \Phi(S_n)]$ is a row contraction then
$\hat \Phi$ is a contractive automorphism of $F_n^\infty$. Moreover,
$$
Aut(B(\cH)^n_1)\simeq Aut_{c}(H^\infty(B(\cH)^n_1)). $$
\end{corollary}

\bigskip
\section{ Model theory  for row contractions  and  the automorphism group
$Aut(B(\cH)^n_1)$  }

This section  deals with the dilation and model theory of row
contractions under the action of the  free holomorphic automorphisms
of $[B(\cH)^n]_1^-$.  We show that if $T\in [B(\cK)^n]_1^-$ and
$\Psi \in Aut(B(\cH)^n_1)$,
  then  the    characteristic function  of
  $\Theta_{\Psi(T)}$
  {\it coincides} with  $\Theta_T\circ\Psi^{-1}$.
  This enables us to  obtain  some results concerning the behavior of  the curvature  and the
Euler characteristic of  a  row contraction under the automorphism
group $Aut(B(\cH)^n_1)$.

If $\lambda=(\lambda_1,\ldots, \lambda_n)\in \BB_n$   and
$X=Y=T:=[T_1,\ldots, T_n]\in [B(\cK)^n]_1^-$ in Theorem
\ref{prop-cara}, we deduce the identities
\begin{equation}
\label{identities}
\begin{split}
 I_{\cK}-\Psi_\lambda(T)\Psi_\lambda(T)^*&=
\Delta_{\lambda}(I- T \lambda^*)^{-1}\Delta_T^2(I-  \lambda
T^*)^{-1} \Delta_{\lambda}, \text{ and }\\
I_{\cK^{(n)}}-\Psi_\lambda(T)^*\Psi_\lambda(T)&=
\Delta_{\lambda^*}(I-{T}^*\lambda)^{-1}\Delta_{T^*}^2(I-{\lambda}^*
T)^{-1} \Delta_{\lambda^*},
\end{split}
\end{equation}
where $\Psi_\lambda$ is the    involutive automorphism of the
noncommutative ball $[B(\cH)^n]_1$,  defined by
\begin{equation}
\label{Phi}
 \Psi_\lambda(X):=-\Theta_\lambda(X)=\lambda-
\Delta_\lambda (I-X \lambda^*)^{-1} X \Delta_{\lambda^*}, \qquad
X\in [B(\cH)^n]_1.
\end{equation}
Consider  the operators $\Omega:\cD_{\Psi_\lambda(T)}\to \cD_T$ and
$\Omega_*:\cD_{\Psi_\lambda(T)^*}\to \cD_{T^*}$  given by
\begin{equation}
\label{OM} \begin{split} \Omega \Delta_{\Psi_\lambda(T)}h&=\Delta_T
(I-\lambda T^*)^{-1} \Delta_\lambda h,\qquad h\in \cK, \text{ and}\\
\Omega_* \Delta_{\Psi_\lambda(T)^*}y&=\Delta_{T^*} (I-\lambda^*
T)^{-1} \Delta_{\lambda^*} y,\qquad y\in \cK^{(n)},
\end{split}
\end{equation}
respectively. Due to the identities above,  $\Omega$ and $\Omega_*$
are unitary maps. We remark that, in particular,  if $T\in
[B(\cK)^n]_1$, then $\Psi_\lambda(T)\in [B(\cK)^n]_1$ and,
consequently, $\Delta_{\Psi_\lambda(T)}$ and
$\Delta_{\Psi_\lambda(T)^*}$ are invertible operators, which implies
$\cD_{\Psi_\lambda(T)}=\cK$ and $\cD_{\Psi_\lambda(T)^*}=\cK^{(n)}$.
In this case  $\Omega:\cK\to \cK$ and $\Omega_*:\cK^{(n)}\to
\cK^{(n)}$.

Let $U=[u_{ij}]_{n\times n}$ be  a unitary operator in $\CC^n$ and
consider the free holomorphic  automorphism $\Phi_U$  of the
noncommutative ball $[B(\cH)^n]_1$, defined by
$$
\Phi_U(X)=XU:=\left[\sum_{i=1}^n a_{i1} X_i,\cdots,\sum_{i=1}^n
a_{in} X_i\right], \qquad X:=[X_1,\ldots, X_n]\in [B(\cH)^n]_1,
$$
 We also use the
notation  $\Phi_U(X)=X{\bf U}$, where ${\bf
U}:=[u_{ij}I_\cH]_{n\times n}$. It is easy to see that  if $T\in
[B(\cK)^n]_1^-$, then
\begin{equation}
\label{Delta2}
\begin{split}
\Delta_{\Phi_U(T)}&=\Delta_T;\qquad\qquad  \cD_{\Phi_U(T)}=\cD_T,  \ \text{ and }\\
\Delta_{\Phi_U(T)^*}&={\bf U^*} \Delta_{T^*} {\bf U};\qquad
\cD_{\Phi_U(T)^*}={\bf U}^* \cD_{T^*},
\end{split}
\end{equation}
where ${\bf U}$ is seen as  the operator $[u_{ij}I_\cK]_{n\times
n}$.

Let $\Psi :=\Phi_U \circ \Psi_\lambda$ be an  free holomorphic
automorphism and let $T\in [B(\cK)^n]_1$.  Since
\begin{equation}
\label{Delta3} \Delta_{\Psi(T)}=\Delta_{\Psi_\lambda(T)}\quad \text{
and } \quad \Delta_{\Psi(T)^*}={\bf
U}^*\Delta_{\Psi_\lambda(T)^*}{\bf U}, \end{equation}
 one can easily
see that the operators $\Omega:\cD_{\Psi(T)}\to \cD_{T}$ and
$\Omega_* {\bf U}:\cD_{\Psi(T)^*}\to \cD_{T^*}$ are unitary
operators identifying the corresponding defect spaces.

According to the considerations above,
  the following lemma  is a simple consequence of Theorem
\ref{prop-cara} and Theorem \ref{automorph}.

\begin{lemma}\label{propre}
Let $X:=[X_1,\ldots, X_n]\in [B(\cK)^n]_1^-$ and let  $\Psi\in
Aut(B(\cH)^n_1)$. Then
\begin{enumerate}
\item[(i)] $\left\|\sum_{i=1}^n X_i X_i^*\right\|\leq 1$ if and only if
$\|\Psi(X)\|\leq 1$;
\item[(ii)] $\left\|\sum_{i=1}^n X_i X_i^*\right\|<1$ if and only if
$\|\Psi(X)\|<1$;
\item[(iii)]
 $X$ is an isometry if and only if $\Psi(X)$ is an isometry;
\item[(iv)] $X$ is a coisometry if and only if $\Psi(X)$ is a
coisometry.
\end{enumerate}
\end{lemma}

Now we can prove some results concerning the minimal isometric
dilation of a row contraction  and the Wold decomposition of a row
isometry  (\cite{Po-isometric})  under the action of
$Aut(B(\cH)^n_1)$.

\begin{proposition}\label{propre1}
Let $X:=[X_1,\ldots, X_n]\in [B(\cK)^n]_1^-$ and let  $\Psi\in
Aut(B(\cH)^n_1)$. Then
\begin{enumerate}
\item[(i)]  $V:=[V_1,\ldots, V_n]$, $V_i\in B(\cG)$, is the
minimal isometric dilation of  $X:=[X_1,\ldots, X_n]$ on a Hilbert
space $\cG\supset \cK$, if and only if $\Psi(V)$ is the minimal
isometric dilation of
 $\Psi(X)$;
 \item[(ii)] If \, $V:=[V_1,\ldots, V_n]$ is a row isometry and
 $(S\otimes I_\cM)\oplus W$ is the  noncommutative Wold decomposition of $V$,
 where $S:=[S_1,\ldots, S_n]$ is $n$-tuple of left creation operators
  and $W:=[W_1,\ldots, W_n]$ is the corresponding Cuntz
 isometry, i.e., $W_1W_1^*+\cdots + W_nW_n^*=I$,
 then $(\Psi(S)\otimes I_\cM)\oplus \Phi(W)$ is
 the noncommutative Wold decomposition  of the row isometry $\Phi(V)$.
\end{enumerate}
\end{proposition}

\begin{proof} Due to Theorem \ref{automorph}, it is enough to
consider the cases when $\Psi=\Psi_\lambda$ and $\Psi=\Phi_U$,
respectively. First, assume that  $\Psi=\Psi_\lambda$ for some
$\lambda\in \BB_n$. Let $V:=[V_1,\ldots, V_n]$, $V_i\in B(\cG)$, be
the minimal isometric dilation of $X:=[X_1,\ldots, X_n]$, on a
Hilbert space $\cG\supset \cK$, i.e., $V_i^*|_{\cK}=X_i^*$ for any
$i=1,\ldots, n$, and $\cG=\bigvee_{\alpha\in \FF_n^+} V_\alpha \cK$.
Due to Proposition \ref{prop-Tl}, $\Psi_\lambda(X)$ is in the norm
closed non-selfadjoint algebra generated by $X_1,\ldots, X_n$ and
the identity. A similar result holds for $\Psi_\lambda(V)$.
Consequently, we deduce that
$\Psi_\lambda(V)^*|_{\cK}=\Psi_\lambda(X)^*$. Moreover, setting
$[W_1,\ldots, W_n]:=\Psi_\lambda(V)$,  we deduce that
$\bigvee_{\alpha\in \FF_n^+} W_\alpha \cK\subseteq
\bigvee_{\alpha\in \FF_n^+} V_\alpha \cK$. Since $\Psi_\lambda
(W)=V$, we also deduce that $\bigvee_{\alpha\in \FF_n^+} V_\alpha
\cK \subseteq \bigvee_{\alpha\in \FF_n^+} W_\alpha \cK$, which
proves that $\Psi_\lambda (V)$ is the minimal isometric dilation  of
$\Psi_\lambda (X)$. The converse  can be proved id a similar manner
and using  the fact that $\Psi_\lambda(\Psi_\lambda(X))=X$.

To prove (ii), note first that, due to Proposition \ref{prop-Tl},
$\Psi_\lambda(S)$ is a pure row isometry which is unitarily
equivalent to $S:=[S_1,\ldots, S_n]$. On the other hand, due to
Lemma \ref{propre}, $W:=[W_1,\ldots, W_n]$ is  a Cuntz
 isometry if and only if  $\Psi_\lambda(W)$ has the same property.
 Since the Wold decomposition of a row isometry is unique up to a
 unitary equivalence (see \cite{Po-isometric}) and
 $$
 \Psi_\lambda((S\otimes I_\cM)\oplus W)=(\Psi_\lambda(S)\otimes
 I_\cM)\oplus \Psi_\lambda (W),
 $$
the result follows. The case when $\Psi=\Phi_U$  follows in a
similar manner. The proof is complete.
\end{proof}

\begin{corollary}\label{propre2}
Let $X:=[X_1,\ldots, X_n]\in [B(\cK)^n]_1^-$ and let   $\Psi\in
Aut(B(\cH)^n_1)$. Then
\begin{enumerate}
\item[(i)]  $X$ is  pure   if and only if $\Psi(X)$ is   pure;
\item[(ii)]
if $V:=[V_1,\ldots, V_n]$  is a pure row isometry, then $V$ is
unitary equivalent to $\Psi(V)$;
\item[(iii)]  $X$ is c.n.c. if and only if $\Psi(X)$ is
c.n.c.
\end{enumerate}

\end{corollary}

\begin{proof} The case when $\Psi=\Phi_U$ for some unitary operator
on $\CC^n$ is straightforward.  Assume that $\Psi=\Psi_\lambda$,
where $\lambda\in \BB_n$. Since $\Psi_\lambda(\Psi_\lambda(X))=X$,
it is enough to prove the direct implication. We recall  (see
\cite{Po-isometric} that a row contraction $X$ is pure if and only
if its minimal isometric dilation $V:=[V_1,\ldots, V_n]$ is a pure
row isometry. Consequently $[V_1,\ldots, V_n]$ is unitarily
equivalent to $[S_1\otimes I_\cM,\ldots, S_n\otimes \cM]$ for some
Hilbert space $\cM$. Applying Proposition \ref{propre}, we deduce
that $\Psi_\lambda (V)=\Psi_\lambda(S)\otimes I_\cM$ is the minimal
isometric  dilation of $\Psi_\lambda(X)$. Since $\Psi_\lambda(S)$ is
a pure row isometry, we conclude that $\Psi_\lambda(X)$ is a pure
row contraction.

We proved in Proposition \ref{propre} that  $\Psi_\lambda(S)$ is a
pure row isometry which is unitarily equivalent to $S:=[S_1,\ldots,
S_n]$.   Due to  the Wold decomposition for row isometries,  $V$ is
unitarily equivalent to $S\otimes I_\cM$ for some Hilbert space
$\cM$.  Hence, we deduce hat $\Psi_\lambda(V)$ is unitarily
equivalent to $V$. Part (iii) was considered  in  Theorem
\ref{Poisson-auto}. This completes the proof.
\end{proof}

We recall (see Section 2) that the characteristic function  of a row
contraction $T:=[T_1,\ldots, T_n]$, \ $T_i\in B(\cK)$,  generates a
bounded free holomorphic function $\Theta_T$ with operator-valued
coefficients in $B(\cD_{T^*}, \cD_T)$ which satisfies the equation
$$\Theta_T(X)={\bf P}_X[\hat \Theta_T]=\tilde T-\Delta_{\tilde T}
(I-\hat X \tilde T^*)^{-1}\hat X \Delta_{\tilde T^*}
$$
 for any $X:=[X_1,\ldots, X_n]\in [B(\cH)^n]_1$,   where we use the
notations $\hat X:=[X_1\otimes I_\cK,\ldots, X_n\otimes I_\cK]$,
$\tilde T:=[I_\cH\otimes T_1,\ldots, I_\cH\otimes T_n]$, and $\hat
\Theta_T:=\text{\rm SOT-}\lim_{r\to 1} \Theta_T(rS_1,\ldots, rS_n)$.
We should add that $${\bf P}_X: F_n^\infty\bar\otimes B(\cD_T,
\cD_{T^*})\to B(\cH)\bar\otimes B(\cD_T, \cD_{T^*})$$
 is the
noncommutative Poisson transform with operator-valued coefficients
defined by
$$
{\bf P}_X[G]:=\left(K_X^*\otimes I_{\cD_{T^*}}\right)\left(
I_{\cD_X}\otimes G\right) \left( K_X\otimes I_{\cD_T}\right)
$$
for any $G\in F_m^\infty\bar \otimes B(\cD_T, \cD_{T^*})$ and
$K_X:\cD_X\otimes F^2(H_n)$ is the Poisson kernel (see Section 2).

The next result concerns the behavior of the characteristic function
under the free holomorphic automorphisms of the noncommutative ball
$[B(\cH)^n]_1$.

\begin{theorem}
\label{cara} Let $T\in [B(\cK)^n]_1^-$ and  let $\Psi :=\Phi_U \circ
\Psi_\lambda$ be an  free holomorphic automorphism of
$[B(\cH)^n]_1$, where $U$ is a unitary operator on $\CC^n$ and
$\lambda\in \BB_n$.
  Then the  standard characteristic function has
the property that
$$
\Theta_{\Psi(T)}(X)=-(I_\cH\otimes \Omega^*)(\Theta_T\circ
\Psi^{-1})(X) (I_\cH \otimes \Omega_* {\bf U}), \qquad X\in
[B(\cH)^n]_1,
$$
 where $\Omega$ and $\Omega_*$ are the
unitary operators defined by \eqref{OM}. Moreover,
$$
\hat \Theta_{\Psi(T)}=-(K_{\widehat {\Psi^{-1}}}^*\otimes \Omega^*)(
I_{\cD_{\widehat{\Psi^{-1}}}}\otimes \hat
\Theta_T)(K_{\widehat{\Psi^{-1}}}\otimes \Omega_* {\bf U} ),
$$
where $K_{\widehat {\Psi^{-1}}}$ is the noncommutative Poisson
kernel of \ $\widehat {\Psi^{-1}}:=\Psi^{-1}(S_1,\ldots, S_n)$ and
it is a unitary operator.
\end{theorem}

\begin{proof}

First, we consider  the case when $\Psi=\Psi_\lambda$. Note that
\begin{equation}
\label{DEDE} \Psi_\lambda(\tilde
T)\Delta_{\lambda^*}=\Delta_\lambda(I-\tilde T \lambda^*)^{-1}
(\lambda-\tilde T).
\end{equation}
Indeed, due to \eqref{Phi} and the fact that $\lambda
\Delta_{\lambda^*}=\Delta_\lambda \lambda$, we have
\begin{equation*}
\begin{split}
\Psi_\lambda(\tilde T)\Delta_{\lambda^*}&= \left[\lambda-
\Delta_\lambda
(I-\tilde T \lambda^*)^{-1} \tilde T \Delta_{\lambda^*}\right]\Delta_{\lambda^*}\\
&= \Delta_\lambda\left[\lambda-(I-\tilde T \lambda^*)^{-1} \tilde T
\Delta_{\lambda^*}^2\right]\\
&= \Delta_\lambda (I-\tilde T \lambda^*)^{-1}\left[ (1-\tilde T
\lambda^*)\lambda-\tilde T(1-\lambda^* \lambda)\right]\\
&=\Delta_\lambda (I-\tilde T \lambda^*)^{-1}(\lambda-\tilde T).
\end{split}
\end{equation*}
Now, we prove that
\begin{equation}
\label{4Delta} \Delta_{\tilde T}[I-\Psi_\lambda(\hat X)\tilde
T^*]^{-1} \Delta_\lambda (I-\hat X \lambda^*)^{-1}= \Delta_{\tilde
T} (I-\lambda \tilde T^*)^{-1} \Delta_\lambda [I-\hat X \Psi_\lambda
(\tilde T)^*]^{-1}.
\end{equation}
First note that  due to the identities of Theorem \ref{prop-cara}
part (ii) and the fact that $\Psi_\lambda (\Psi_\lambda(\tilde
T))=\tilde T$, we have
\begin{equation}
\label{ident1}
\begin{split}
I-\Psi_\lambda(\hat X) \tilde T^*&= (I-\Psi_\lambda(\hat X)
\Psi_\lambda (\Psi_\lambda(\tilde T))^*\\
&= \Delta_\lambda (I-\hat X \lambda^*)^{-1} [I-\hat X
\Psi_\lambda(\tilde T)^*](I-\lambda \Psi_\lambda (\tilde T)^*)^{-1}
\Delta_\lambda.
\end{split}
\end{equation}
Using  again the  identities of Theorem \ref{prop-cara} and  the
fact that $\Psi_\lambda(0)=\lambda$, we have
\begin{equation}
\label{ident2} I-\lambda\Psi_\lambda(\tilde T)^*=I-\Psi_\lambda(0)
\Psi_\lambda(\tilde T)^*=\Delta_\lambda(I-\lambda \tilde T^*)
\Delta_\lambda.
\end{equation}
Due  to \eqref{ident1} and  \eqref{ident2}, we deduce that
\begin{equation*}
\begin{split}
\Delta_{\tilde T}[I-\Psi_\lambda(\hat X)\tilde T^*]^{-1}&
\Delta_\lambda (I-\hat X \lambda^*)^{-1} \\
&=
 \Delta_{\tilde T}\Delta_\lambda^{-1}(I-\lambda \Psi_\lambda (\tilde T)^*)
 [I-\hat X
\Psi_\lambda(\tilde T)^*]^{-1}(I-\hat X
\lambda^*)\Delta_\lambda^{-1}
\Delta_\lambda (I-\hat X \lambda^*)^{-1}\\
&=\Delta_{\tilde T}  \Delta_\lambda^{-1} \Delta_\lambda (I-\lambda
\tilde T^*)^{-1} \Delta_\lambda [I-\hat X \Psi_\lambda (\tilde
T)^*]^{-1}\\
&=\Delta_{\tilde T} (I-\lambda \tilde T^*)^{-1} \Delta_\lambda
[I-\hat X \Psi_\lambda (\tilde T)^*]^{-1},
\end{split}
\end{equation*}
which proves  \eqref{4Delta}.

Since $\psi_\lambda(T)$ is a row contraction, its characteristic
function generates a unique  free holomorphic function
$\Theta_{\psi_\lambda(T)}:[B(\cH)^n]_1\to B(\cH)\bar\otimes
B(\cD_{\psi_\lambda(T)^*}, \cD_{\psi_\lambda(T)})$ with
operator-valued coefficients, which,  due to the fact that
$\Psi_\lambda(\tilde T)=I\otimes \Psi_\lambda(T)$, satisfies the
equation
$$
\Theta_{\psi_\lambda(T)}(X)=-\Psi_\lambda(\tilde T)+[I-\hat X
\Psi_\lambda(\tilde T)^*]^{-1} \hat X \Delta_{\Psi_\lambda(\tilde
T)^*}, \qquad X\in [B(\cH)^n]_1.
$$

Using the the  relation $\Psi_\lambda(\tilde
T)\Delta_{\Psi_\lambda(\tilde T)^*}=\Delta_{\Psi_\lambda(\tilde
T)}\Psi_\lambda(\tilde T)$, the definition of $\Omega$,  and
relation  \eqref{4Delta}, we obtain that
\begin{equation*}
\begin{split}
&(I\otimes \Omega)\Theta_{\psi_\lambda(T)}(X)
\Delta_{\Psi_\lambda(\tilde T)^*}\\
&=
(I\otimes \Omega)\Delta_{\Psi_\lambda(\tilde T)}
\left[-\Psi_\lambda(\tilde T)+[I-\hat X \Psi_\lambda(\tilde
T)^*]^{-1} \hat X \Delta_{\Psi_\lambda(\tilde T)^*}^2\right]\\
&= \Delta_{\tilde T} (I-\lambda \tilde T^*)^{-1} \Delta_\lambda
\left[-\Psi_\lambda(\tilde T)+[I-\hat X \Psi_\lambda(\tilde
T)^*]^{-1} \hat X \Delta_{\Psi_\lambda(\tilde T)^*}^2\right]\\
&= \Delta_{\tilde T}[I-\Psi_\lambda(\hat X)\tilde T^*]^{-1}
\Delta_\lambda (I-\hat X \lambda^*)^{-1}[I-\hat X \Psi_\lambda
(\tilde T)^*]\left[-\Psi_\lambda(\tilde T)+[I-\hat X
\Psi_\lambda(\tilde T)^*]^{-1} \hat X \Delta_{\Psi_\lambda(\tilde
T)^*}^2\right]\\
&=-\Delta_{\tilde T}[I-\Psi_\lambda(\hat X)\tilde T^*]^{-1} \left\{
\Delta_\lambda (I-\hat X \lambda^*)^{-1}[I-\hat X \Psi_\lambda
(\tilde T)^*]\Psi_\lambda(\tilde T)-\Delta_\lambda (I-\hat X
\lambda^*)^{-1} \hat X \Delta_{\Psi_\lambda(\tilde T)^*}^2\right\}\\
&=-\Delta_{\tilde T}[I-\Psi_\lambda(\hat X)\tilde T^*]^{-1}
\Delta_\lambda (I-\hat X \lambda^*)^{-1} [\Psi_\lambda(\tilde
T)-X]\\
&= -\Delta_{\tilde T}[I-\Psi_\lambda(\hat X)\tilde T^*]^{-1}
\left\{\Delta_\lambda (I-\hat X \lambda^*)^{-1}(\lambda-\hat
X)+\Delta_\lambda (I-\hat X \lambda^*)^{-1}[\Psi_\lambda(\tilde
T)-\lambda]\right\}\\
&= -\Delta_{\tilde T}[I-\Psi_\lambda(\hat X)\tilde T^*]^{-1} \left[
\Psi_\lambda(\hat X) \Delta_{\lambda^*}-\Delta_\lambda(I-\hat X
\lambda^*)^{-1} \Delta_\lambda (I-\tilde T \lambda^*)^{-1} \tilde T
\Delta_{\lambda^*}\right],
\end{split}
\end{equation*}
where the latter equality is due to the identity \eqref{DEDE}, where
we  replace $\tilde T$ by $\hat X$.

On the other hand, we have
\begin{equation}
\label{eq1}
\begin{split}
-& (\Theta_T\circ \Psi_\lambda)(X) (I\otimes \Omega_*)\Delta_{\Psi_\lambda(\tilde T)^*}\\
&=
\left[ \tilde T-\Delta_{\tilde T} [I-\Psi_\lambda(\hat X) \tilde
T^*]^{-1}\Psi_\lambda(\hat X) \Delta_{\tilde T^*}\right]
\Delta_{\tilde T^*} (I-\lambda^* \tilde T)^{-1} \Delta_{\lambda^*}\\
&= \tilde T\Delta_{\tilde T^*} (I-\lambda^* \tilde T)^{-1}
\Delta_{\lambda^*}-\Delta_{\tilde T} [I-\Psi_\lambda(\hat X) \tilde
T^*]^{-1}  \Psi_\lambda(\hat X)(I-\lambda^*
\tilde T)^{-1} \Delta_{\lambda^*}\\
&\qquad  \qquad \qquad+\Delta_{\tilde T} [I-\Psi_\lambda(\hat X)
\tilde T^*]^{-1}  \Psi_\lambda(\hat X)\tilde T^* \tilde
T(I-\lambda^* \tilde T)^{-1} \Delta_{\lambda^*}.
\end{split}
\end{equation}
Note that the latter term in  the sum above  is equal to
\begin{equation*}
\begin{split}
\Delta_{\tilde T} &\left\{[I-\Psi_\lambda(\hat X) \tilde
T^*]^{-1}-I\right\}\tilde T(I-\lambda^* \tilde T)^{-1}
\Delta_{\lambda^*}\\
&= \Delta_{\tilde T}  [I-\Psi_\lambda(\hat X) \tilde T^*]^{-1}\tilde
T(I-\lambda^* \tilde T)^{-1} \Delta_{\lambda^*}- \Delta_{\tilde T}
\tilde T(I-\lambda^* \tilde T)^{-1} \Delta_{\lambda^*}\\
&= \Delta_{\tilde T}  [I-\Psi_\lambda(\hat X) \tilde T^*]^{-1}\tilde
T(I-\lambda^* \tilde T)^{-1} \Delta_{\lambda^*}- \tilde
T\Delta_{\tilde T^*} (I-\lambda^* \tilde T)^{-1} \Delta_{\lambda^*}.
\end{split}
\end{equation*}
 Hence,  going back to  \eqref{eq1}, we obtain
\begin{equation*}
\begin{split}
-&(I\otimes \Omega)\Theta_{\psi_\lambda(T)}(X)
\Delta_{\Psi_\lambda(\tilde T)^*}\\
&=\Delta_{\tilde T}  [I-\Psi_\lambda(\hat X) \tilde
T^*]^{-1}\left[\tilde T(I-\lambda^* \tilde T)^{-1}
\Delta_{\lambda^*} -\Psi_\lambda(\hat X) (I-\lambda^*\tilde T)^{-1}
\Delta_{\lambda^*}\right]\\
&= \Delta_{\tilde T}  [I-\Psi_\lambda(\hat X) \tilde
T^*]^{-1}\left[\tilde T(I-\lambda^* \tilde T)^{-1}
\Delta_{\lambda^*} -\Psi_\lambda(\hat X)\Delta_{\lambda^*} -
\Psi_\lambda(\hat X) \lambda^*\tilde T (I-\lambda^*\tilde T)^{-1}
\Delta_{\lambda^*}\right]\\
&=
 \Delta_{\tilde T}  [I-\Psi_\lambda(\hat X) \tilde
T^*]^{-1}\left[ -\Psi_\lambda(\hat X)\Delta_{\lambda^*}
+[I-\Psi_\lambda(\hat X) \lambda^*] \tilde T (I-\lambda^*\tilde
T)^{-1}
\Delta_{\lambda^*}\right]\\
&= \Delta_{\tilde T}  [I-\Psi_\lambda(\hat X) \tilde T^*]^{-1}\left[
-\Psi_\lambda(\hat X)\Delta_{\lambda^*} + \Delta_\lambda (I-\hat X
\lambda^*)^{-1} \Delta_\lambda \tilde T (I-\lambda^*\tilde T)^{-1}
\Delta_{\lambda^*}\right],
\end{split}
\end{equation*}
where the latter equality is due to the identity
$$
I-\Psi_\lambda(\hat X) \lambda^*=I-\Psi_\lambda(\hat X)\Psi_\lambda
(0) =\Delta_\lambda (I-\hat X \lambda^*)^{-1} \Delta_\lambda,
$$
which follows from  the identities of Theorem \ref{prop-cara} part
(ii) and the fact that $\Psi_\lambda (0)=\lambda$. Now, using the
identities above and  relation
$$
(I-\tilde T \lambda^*)^{-1} \tilde T=\tilde T (I- \lambda^*\tilde
T)^{-1},
$$
we  deduce that
\begin{equation}
\label{case1} \Theta_{\Psi_\lambda(T)}(X)=-(I_\cH\otimes
\Omega^*)(\Theta_T\circ \Psi_\lambda)(X) (I_\cH \otimes \Omega_*),
\qquad X\in [B(\cH)^n]_1.
\end{equation}

Now let us consider the case when $\Psi=\Phi_U$ for some unitary
operator $U$ on $\CC^n$. Using  relation \eqref{Delta2} and the
definition of the characteristic function for a row contraction, one
can easily deduce that
\begin{equation}
\label{case2}
 \Theta_{\Phi_U(T)}(X)=(\Theta_T\circ
\Phi_U^{-1})(X)(I_\cH\otimes {\bf U}), \qquad X\in [B(\cH)^n]_1.
\end{equation}

 Now  we consider the general case when $\Psi :=\Phi_U \circ
\Psi_\lambda$. Combining relations \eqref{case1} and
 \eqref{case2}, we obtain
$$
\Theta_{\Psi(T)}(X)=-(I_\cH\otimes \Omega^*)(\Theta_T\circ
\Psi^{-1})(X) (I_\cH \otimes \Omega_* {\bf U}), \qquad X\in
[B(\cH)^n]_1,
$$
which proves the first part of the theorem.

To prove the second part,  note first that, due to Proposition
\ref{prop-Tl}, $\hat \Psi_\lambda=\Psi(S_1,\ldots, S_n)$ is  a pure
row contraction. By Corollary \ref{propre2},  $\hat \Psi$ has the
same property. Applying the operator-valued versions of Lemma
\ref{comp-iso} and Remark \ref{ext} to the bounded free holomorphic
function $\Theta_T:[B(\cH)^n]_1\to [B(\cH)^n]_1\bar \otimes B(\cD_T,
\cD_{T^*})$ with operator-valued coefficients, we deduce that
$$
\Theta_T(\Psi^{-1}(rS_1,\ldots, rS_n))= {\bf P}_{[rS_1,\ldots,
rS_n]}\left[ {\bf P}_{\widehat{\Psi^{-1}}}\left[\hat
\Theta_T\right]\right]
$$
where   $\hat \Theta_T$ and $\widehat{\Psi^{-1}}$ are the
  corresponding boundary
functions, and ${\bf P}$ is the corresponding noncommutative Poisson
transform  with operator-valued coefficients. On the other hand,
since ${\bf P}_{\widehat {\Psi^{-1}}}\left[\hat \Theta_T\right]$ is
in $F_n^\infty \bar \otimes  B(\cD_T, \cD_{T^*})$, the generalized
$F_n^\infty$ functional calculus implies
$$
\text{\rm SOT-}\lim_{r\to 1} {\bf P}_{[rS_1,\ldots, rS_n]}\left[
{\bf P}_{\widehat {\Psi^{-1}}}\left[\hat \Theta_T\right]\right]={\bf
P}_{\widehat {\Psi^{-1}}}\left[\hat \Theta_T\right].
$$
Therefore, we obtain
$$
\text{\rm SOT-}\lim_{r\to 1}\Theta_T(\Psi^{-1}(rS_1,\ldots,
rS_n))={\bf P}_{\widehat {\Psi^{-1}}}\left[\hat \Theta_T\right],
$$
which implies
$$
\hat \Theta_{\Psi(T)}=-(K_{\widehat {\Psi^{-1}}}^*\otimes \Omega^*)(
I_{\cD_{\widehat{\Psi^{-1}}}}\otimes \hat \Theta_T)(K_{\widehat
{\Psi^{-1}}}\otimes \Omega_* {\bf U}).
$$
The fact that  the noncommutative Poisson kernel $K_{\widehat
{\Psi^{-1}}}$ is a unitary operator was proved in Theorem
\ref{Voicu}. The proof is complete.
\end{proof}

We  recall  from the proof of Theorem \ref{Voicu} that there is a
unitary operator $U_\lambda:\cD_{{\hat \Psi_\lambda}}\to \CC$
defined by
\begin{equation*}
\begin{split}
U_\lambda\Delta_{\hat \Psi_\lambda} f =(1-\|\lambda\|_2^2)^{1/2}
f(\lambda),\qquad f\in F^2(H_n).
\end{split}
\end{equation*}
 Consequently, Theorem \ref{cara} implies
$$
\hat \Theta_{\Psi_\lambda(T)}=-( W_\lambda^*\otimes \Omega^*)\hat
\Theta_T(W_\lambda\otimes \Omega_*),
$$
where $W_\lambda:=(U_\lambda\otimes
I_{F^2(H_n)})K_{\hat\Psi_\lambda}$ is a unitary operator on the full
Fock space $F^2(H_n)$. A similar observation can be made in the case
of an arbitrary  free  holomorphic automorphism of $[B(\cH)^n]_1$.

\begin{corollary}\label{formula} If $\lambda,\mu\in \BB_n$, then
$$
(\Psi_\mu \circ \Psi_\lambda) (X)=-(I_\cH\otimes \Omega) \Psi
_{\Psi_\lambda(\mu)}(X) (I_\cH\otimes \Omega_*^*), \qquad X\in
[B(\cH)^n]_1,
$$
 where $\Omega\in B(\CC)$ and $\Omega_*\in B(\CC^n)$ are the
unitary operators defined by \eqref{OM},  when $\cK=\CC$ and
$T=\mu\in \BB_n$.
\end{corollary}

 We   showed in \cite{Po-holomorphic} that Arveson's  curvature  $K(T)$ (see \cite{Arv2})  associated with
  a commutative row contraction  $T:=[T_1,\ldots, T_n]$, i.e., $T_iT_j=T_jT_i$,
   \ $ i,j=1,\ldots,
  n$,
with  $\rank \Delta_T<\infty$ can be expressed in terms     of the
constrained  characteristic function  (also denoted by $\Theta_T$)
given by
$$
\Theta_{T}(z):= -T+\Delta_T(I-z_1T_1^*-\cdots -z_nT_n^*)^{-1}
[z_1I_\cH,\ldots, z_nI_\cH]\Delta_{T^*},\quad z\in \BB_n.
$$
More precisely, we proved that
$$
K(T)=\int_{\partial \BB_n}\lim_{r\to 1}\text{\rm trace}\,
 [I_{\cD_T}-\Theta_{T}(r\xi)\Theta_{T}(r\xi)^*] d\sigma(\xi),
 $$
where $\sigma$ is the rotation-invariant probability measure on
$\partial \BB_n$.
 The proof of the following lemma is
straightforward.

\begin{lemma} Let $T:=[T_1,\ldots, T_n]\in B(\cK)^n$ be an   $n$-tuple of operators and let
  $\Psi \in  Aut(B(\cH)^n_1)$. Then the following statements hold:
\begin{enumerate}
\item[(i)] if $T$ is a row contraction, then
$\rank \Delta_T=\rank \Delta_{\Psi(T)}$;
 \item[(ii)]
 $T$ is a
commuting row contraction if and only if $\Psi(T)$ is a commuting
row contraction;
\item[(iii)] $A\in \{T\}'$ if and only if $A\in \{\Psi
(T)\}'$, where $'$ denote the commutant;
\item[(iv)]  $X,Y\in [B(\cK)^n]_1$ are  unitarily
equivalent  if and only if $\Psi(X)$ and $\Psi(Y)$ have the same
property.
\end{enumerate}
\end{lemma}

Now, using Theorem \ref{cara}, one can  deduce the following result.

\begin{theorem}
Let $T:=[T_1,\ldots, T_n]$, \ $T_i\in B(\cK)$, be a commutative  row
contraction  with $\rank \Delta_T<\infty$. Then   Arveson's
curvature  satisfies the equation
$$K(\Psi(T))= \int_{\partial \BB_n}\lim_{r\to 1}\text{\rm trace}\,
 [I_{\cD_T}-\Theta_{T}(\Psi^{-1}(r\xi))\Theta_{T}(\Psi^{-1}(r\xi))^*] d\sigma(\xi)
$$
for any  $\Psi\in Aut(B(\cH)^n_1)$.
\end{theorem}

In the noncommutative case, the notions of curvature and  Euler
characteristic of a  row contraction  were introduced in
\cite{Po-curvature} and \cite{Kr}.  We showed in \cite{Po-varieties}
that the curvature and the Euler characteristic of an arbitrary  row
contraction   can be expressed only in terms of the standard
characteristic function. More precisely, we  proved that if
$T:=[T_1,\ldots, T_n]$, \ $T_i\in B(\cK)$, is a row contraction with
$\rank \Delta_T<\infty$, and $\text{\rm curv}\,(T)$ and $\chi(T)$
denote its curvature and Euler characteristic, respectively, then
$$
\text{\rm curv}\,(T)=\rank \Delta_T-\lim_{m\to\infty} \frac
{\text{\rm trace}\, [\hat\Theta_T\hat \Theta_T^*(P_m\otimes I)]}
{n^m}
$$
and
$$
\chi(T)= \lim_{m\to\infty} \frac {\rank[(I-\hat\Theta_T
\hat\Theta_T^*)(P_{\leq m}\otimes I)]} {1+n+\cdots+n^{m-1}},
$$
where $P_m$ $($resp. $P_{\leq m}$$)$ is the orthogonal projection of
the full Fock space $F^2(H_n)$ onto the subspace of all homogeneous
polynomials of degree $m$ $($resp. polynomials of degree $\leq
m$$)$. Using the result of  Theorem \ref{cara}, we can express the
curvature (resp. Euler characteristic) of $\Psi(T)$ in terms of  the
boundary function of the free holomorphic function $\Theta\circ
\Psi^{-1}$.

\begin{proposition}
\label{curv-auto} If $V:=[V_1,\ldots, V_n]$ is a pure isometry with
$\rank \Delta_V<\infty$, then
$$
\text{\rm curv}\,(V)=\text{\rm curv}\,(\Psi(V))=\rank \Delta_V,\quad
\Psi\in Aut(B(\cH)^n_1).
$$
If $T\in B(\cK)$ is a pure contraction, then $$ \text{\rm
curv}\,(T)=\text{\rm curv}\,(\Psi(T)),\quad \Psi\in Aut(B(\cH)_1).
$$
\end{proposition}
\begin{proof}

According to Proposition \ref{propre2}, if $V:=[V_1,\ldots, V_n]$ is
a pure row isometry, then $V$ is unitary equivalent to $\Psi(V)$.
Consequently, $\text{\rm curv}\,(V)=\text{\rm curv}\,\Psi(V)$. The
second equality is due to the fact that $V$ is unitarily equivalent
to $[S_1\otimes \cM,\ldots, S_n\otimes I_\cM]$, where $\cM$ is a
Hilbert space of dimension equal to $\rank \Delta_V$.

To prove the second part of this proposition, we recall (see
\cite{Par}) that $ \text{\rm curv}\,(T)=\rank \Delta_T-\rank
\Delta_{T^*}. $ The operators $\Omega $ and $\Omega_*$ defined by
relation \eqref{OM} are unitaries. Using this result when $n=1$, we
have $\rank \Delta_T=\rank \Delta_{\Psi(T)}$ and $\rank
\Delta_{T^*}=\rank \Delta_{\Psi(T)^*}$. The proof is complete.
\end{proof}
It will be interesting to know if $ \text{\rm curv}\,(T)=\text{\rm
curv}\,(\Psi(T))$ for any pure row contraction $T$ and any  $\Psi\in
Aut(B(\cH)^n_1)$.

 \bigskip

\section{ Maximum principle and Schwartz type results   for free  holomorphic   functions
   }

In this section we  prove a maximum principle for free holomorphic
functions and obtain    noncommutative versions of Schwarz lemma
from  complex analysis.

We proved in \cite{Po-holomorphic} the following     analogue of
Schwartz lemma   for free holomorphic functions. Let
$F(X)=\sum_{\alpha\in \FF_n^+} X_\alpha\otimes A_{(\alpha)}$, \
$A_{(\alpha)}\in B(\cE, \cG)$,  be a free holomorphic function on
$[B(\cH)^n]_1$ with $\|F\|_\infty\leq 1$ and $F(0)=0$. Then $$
\|F(X)\| \leq \|X\|,\qquad \text{ for any } \   X\in [B(\cH)^n]_1.
$$

In what follows we use this result and the free holomorphic
automorphisms of the noncommutative ball $[B(\cH)^n]_1$ to  prove a
maximum principle for free holomorphic functions.
\begin{theorem}
\label{max-princ} If $F:[B(\cH)^n]_1\to B(\cH)$ is a free
holomorphic function  and there exists $X_0\in [B(\cH)^n]_1$ such
that
$$\|F(X_0)\|\geq \|F(X)\| \quad \text{ for all }\ X\in [B(\cH)^n]_1,
$$
 then $F$ must be a constant.
\end{theorem}

\begin{proof}
Without loss of generality, we can assume that
$\|F(X_0)\|=\|F\|_\infty=1$. Let $F$ have  the representation
$F(X)=\sum_{k=0}^\infty \sum_{|\alpha|=k} a_\alpha X_\alpha $.
According to \cite{PPoS}, we have
$$
\left( \sum_{|\alpha|=k} |a_\alpha|^2\right)^{1/2}\leq
1-|F(0)|^2\qquad  \text{ for }  k=1,2,\ldots.
$$
Hence, if $|F(0)|=1$, then $a_\alpha=0$ for all $\alpha\in \FF_n^+$
with $|\alpha|\geq 1$. Therefore $F=F(0)$ is a constant. Now set
$\lambda:= F(0)$  and assume that $|\lambda|<1$. Note that
$G:=\Psi_\lambda \circ F$ is a free holomorphic function such that
$G(0)=0$ and $\|G\|_\infty\leq 1$. Due to the noncommutative Schwarz
type lemma for  bounded free holomorphic functions, we have
$\|G(X)\|\leq \|X\|$ for any $X\in [B(\cH)^n]_1$. Therefore, we have
\begin{equation}
\label{FF}
 \|\Psi_\lambda(F(X_0))\|\leq \|X_0\|<1.
\end{equation}
On the other hand, since $\|F(X_0)\|=1$,  Lemma \ref{propre} implies
$\|\Psi_\lambda(F(X_0))\|=1$, which contradicts \eqref{FF}.
Therefore, $F$ must be a constant. The proof is complete.
\end{proof}

We remark that if $F:[B(\cH)^n]_1\to B(\cH)\bar\otimes B(\cK)$ is a
free holomorphic function with coefficients in $B(\cK)$ and $\dim
\cK\geq 2$, the maximum principle of Theorem \ref{max-princ} fails.
Indeed, take $\cK=\CC^2$      and $$F(X_1,\ldots, X_n)=
I\otimes \left[\begin{matrix} 1& 0\\
0&0\end{matrix} \right] + X_1\otimes \left[\begin{matrix} 0& 0\\
0&1\end{matrix} \right]
$$
for $(X_1,\ldots, X_n)\in  [B(\cH)^n]_1$, and note that
$\|F\|_\infty=1=\|F(0)\|$ and $F(0)$ is a projection.  We also
mention that,  when $\cK$ is an arbitrary Hilbert space,
$\|F\|_\infty\leq 1$, and $F(0)$ is an isometry, then $F$ must be a
constant. Indeed, if $F$ has the representation
$f(X)=\sum_{k=0}^\infty \sum_{|\alpha|=k} X_\alpha \otimes
A_{(\alpha)}$, then, due to \cite{Po-Bohr}, we have $$
\sum_{|\alpha|=k}A_{(\alpha)}^*A_{(\alpha)}\leq I-F(0)^* F(0) \quad
\text{ for }  k=1,2,\ldots. $$
 Hence, we deduce our assertion.

\begin{proposition}\label{f=0} Let $F:[B(\cH)^n]_1\to B(\cH)^m$ be  a  bounded free
holomorphic function such that $\|F(0)\|\neq \|F\|_\infty$. Then
there in no $X_0\in [B(\cH)^n]_1$ such that
$\|F(X_0)\|=\|F\|_\infty$.
\end{proposition}

\begin{proof} Suppose that there is $X_0\in [B(\cH)^n]_1$ such that
$\|F(X_0)\|=\|F\|_\infty$. Without loss of generality, we can assume
that $\|F\|_\infty=1$. Set $\lambda:= F(0)\in \BB_n$  and   note
that  $G:=\Psi_\lambda \circ F$ is a free holomorphic function such
that $G(0)=0$ and $\|G\|_\infty\leq 1$. The rest of the proof is
similar to that of Theorem \ref{max-princ}.
\end{proof}

What happens if $\|F(0)\|= \|F\|_\infty$, in Proposition \ref{f=0} ?
Does this  condition imply that $F$ is a constant ?    We already
proved  that the answer is positive if $m=1$.

The classical Schwarz's lemma  states that if $f:\DD\to \CC$ is  a
bounded analytic function with $f(0)=0$ and $|f(z)|\leq 1$  for
$z\in \DD$, then $|f'(0)|\leq 1$ and $|f(z)|\leq |z|$ for $z\in
\DD$. Moreover, if $|f'(0)|= 1$ or if $|f(z)|=|z|$ for some $z\neq
0$, then there is a constant $c$ with $|c|=1$ such that $f(w)=cw$
for any $w\in \DD$.

 Let $F_1,\ldots, F_m$ be free  holomorphic functions  on $[B(\cH)^n]_{\gamma_1}$  with
 scalar coefficients. Then the map $F:[B(\cH)^n]_{\gamma_1}\to B(\cH)^m$ defined by $F:=(F_1,\ldots,
 F_m)$ is a free holomorphic function.  We define $F'(0)$ as the
 linear operator from $\CC^n$ to $\CC^m$  having the matrix $\left[\left( \frac{\partial F_i}
{\partial Z_j}\right)(0)\right]_{m\times n}$. In what follows, we
obtain   a new proof of the  Schwarz type lemma   for free
holomorphic functions \cite{Po-holomorphic}  and new results
concerning the uniqueness.

\begin{theorem}
\label{Schw1} Let $F:[B(\cH)^n]_{1}\to [B(\cH)^m]_{1}$ be a free
holomorphic function.  Then
\begin{enumerate}
\item[(i)]  the free holomorphic function
$$\varphi(X_1,\ldots, X_n)= [X_1,\ldots, X_n] F'(0)^t, \qquad (X_1,\ldots,
X_n)\in [B(\cH)^n]_{1},
$$
 maps $[B(\cH)^n]_{1}$ into
$[B(\cH)^m]_{1}$, where ${}^t$ denotes the transpose. In particular,
$\|F'(0)\|\leq 1$;

\item[(ii)] if $F(0)=0$ then  $\|F(X)\|\leq \|X\|$ for any $X\in [B(\cH)^n]_1$.
\end{enumerate}
\end{theorem}

\begin{proof} Let $X:=(X_1,\ldots, X_n)\in [B(\cH)^n]_1$ and let
$t\in (0,1)$ such that $\|X\|< t$. For each $x\in \cH^{(n)}$ and
$y\in \cH$ with $\|x\|\leq 1$ and $\|y\|\leq 1$, define
$g_{xy}:\DD\to \CC$ by stting
$$
g_{xy} (\lambda):=\left<
F\left(\frac{\lambda}{t}X_1,\ldots,\frac{\lambda}{t}X_n\right)x,
y\right>,\qquad \lambda\in \DD.
$$
Since $F$ is  a free holomorphic function on $[B(\cH)^n]_1$, the map
$g_{xy}$ is analytic on the open unit disc $\DD$ and
\begin{equation}
\label{gxy} g_{xy}'(0)=\left<\frac{1}{t}[X_1,\ldots, X_n] F'(0)^t
x,y\right>.
\end{equation}
Assume that there exists $c_{xy}\in \CC$ with $|c_{xy}|=1$ such that
$g_{xy}(\lambda)=c_{xy} \lambda$ for  $\lambda\in \DD$.
 Due to the continuity of $F$ on
$[B(\cH)^n]_1$ and taking $\lambda\to 1$ in the equality
$|g_{xy}(\lambda)|=|\lambda|$, we deduce that
$$
\left|\left< F\left(\frac{1}{t}X_1,\ldots,\frac{1}{t}X_n\right)x,
y\right>\right|=1,
$$
which contradicts the fact that
$\left\||F\left(\frac{1}{t}X_1,\ldots,\frac{1}{t}X_n\right)\right\|<1.
$ The classical Schwarz lemma  applied  to $g_{xy}$ implies
$|g_{xy}'(0)|<1$ for any
 $x\in \cH^{(n)}$ and
$y\in \cH$ with $\|x\|\leq 1$ and $\|y\|\leq 1$. By \eqref{gxy}, we
get $
 \left\|[X_1,\ldots, X_n] F'(0)^t\right\|\leq t<1
$ for any $(X_1,\ldots, X_n)\in [B(\cH)^n]_1$, which  also implies
that $\|F'(0)\|\leq 1$ and proves part (i).

To prove (ii), assume that $F(0)=0$. Once again  Schwartz lemma
implies $|g_{xy}(\lambda)|<|\lambda|$ for any $\lambda\in
\DD\backslash\{0\}$, except when there exists $c_{xy}\in \CC$ with
$|c_{xy}|=1$ such that $g_{xy}(\lambda)=c_{xy} \lambda$, \
$\lambda\in \DD$. As we showed above, the latter condition does not
hold. Therefore, we  must have

$$
\left|\left<
F\left(\frac{\lambda}{t}X_1,\ldots,\frac{\lambda}{t}X_n\right)x,
y\right>\right|<|\lambda|,\qquad \lambda\in \DD\backslash\{0\}.
$$
Hence, considering $\lambda\in \DD$ with $|\lambda|=\|X\|<1$ and
taking  $t\to \|X\|$, we deduce that $|\left< F(X) x, y\right> |\leq
\|X\|$ for any
 $x\in \cH^{(n)}$ and
$y\in \cH$ with $\|x\|\leq 1$ and $\|y\|\leq 1$. Consequently,
$\|F(X)\|\leq \|X\|$, which proves part (ii). The proof is complete.
\end{proof}

We remark that one can extend  Theorem \ref{Schw1}  to free
holomorphic functions  $F:[B(\cH)^n]_{\gamma_1}\to
[B(\cH)^m]_{\gamma_2}$.  In this case, one can similarly show that
the free holomorphic function
$$\varphi(X_1,\ldots, X_n)= [X_1,\ldots, X_n] F'(0)^t, \qquad (X_1,\ldots,
X_n)\in [B(\cH)^n]_{\gamma_1},
$$
 maps $[B(\cH)^n]_{\gamma_1}$ into
$[B(\cH)^m]_{\gamma_2}$  and, if $F(0)=0$, then  $F$ maps \
$r[B(\cH)^n]_{\gamma_1}$ into \  $r[B(\cH)^m]_{\gamma_2}$ for any
$r\in (0, 1]$.

\begin{theorem}
\label{Schw}
 Let $F:[B(\cH)^n]_{1}\to [B(\cH)^n]_{1}$ be a free
holomorphic function such that $F'(0)$ is a unitary operator on
$\CC^n$. Then $F$ is a  free holomorphic automorphism  of
$[B(\cH)^n]_{1}$ and
$$
F(X_1,\ldots, X_n)= [X_1,\ldots, X_n] [F'(0)]^t, \qquad (X_1,\ldots,
X_n)\in [B(\cH)^n]_{1}.
$$
\end{theorem}

\begin{proof} Assume that $F$ has the representation
$$
F(X_1,\ldots, X_n)=A_0+\sum_{i=1}^n X_i\otimes A_i+\sum_{k=2}^\infty
\sum_{|\alpha|=k} X_\alpha\otimes A_{(\alpha)},
$$
where $A_0, A_i, A_{(\alpha)}\in \BB_n$ are written   as row
operators with entries in $\CC$. Note that $F'(0)=[A_1^t \ \cdots \
A_n^t]$, where ${}^t$ denotes the transpose. Since $F'(0)$ is a
co-isometry we  have $\sum_{i=1}^n A_i^* A_i=I$. On the other hand,
 since $F$ is a bounded free holomorphic  function with $\|F\|_\infty\leq 1$, we can use
 Theorem 2.9
  from   \cite{Po-Bohr} to deduce that
  $
  \sum_{i=1}^n A_i^* A_i\leq I-F(0)^* F(0).
  $
 Now, it is clear that $F(0)=0$. Therefore,
we have
\begin{equation}
\label{FFF}
 F(X_1,\ldots, X_n)=[X_1,\ldots, X_n]\left[\begin{matrix}
A_1\\\vdots\\ A_n\end{matrix} \right]+ \sum_{k=2}^\infty
\sum_{|\alpha|=k} X_\alpha\otimes A_{(\alpha)}.
\end{equation}
Since $F'(0)$ is an isometry, we have $ \left[\begin{matrix}
A_1\\\vdots\\ A_n\end{matrix} \right] [A_1^*\ \cdots\ A_n^*]=I. $
Multiplying  relation \eqref{FFF} to the right by the operator
matrix $[A_1^*\ \cdots\ A_n^*]$, we deduce that
\begin{equation}
\label{H=F}
 H(X):=F(X)[A_1^*\ \cdots\ A_n^*]=[X_1,\ldots,
X_n]+\cdots
\end{equation}
Therefore, $H$ is a free holomorphic function  with $H(0)=0$ and
$H'(0)=I_n$. Applying Theorem \ref{cartan1} to $H$, we deduce that
$H(X)=X$ for any  $X\in [B(\cH)^n]_1$. Multiplying  relation
\eqref{H=F}  to the right by $\left[\begin{matrix} A_1\\\vdots\\
A_n\end{matrix} \right]$ and using  the fact that  $F'(0)$ is a
co-isometry, we deduce that
$$
F(X_1,\ldots, X_n)= [X_1,\ldots, X_n] [F'(0)]^t, \qquad (X_1,\ldots,
X_n)\in [B(\cH)^n]_{1},
$$
where  $F'(0)=[A_1^t \ \cdots \ A_n^t]$.  This  completes the proof.
\end{proof}
 A slight extension of Theorem \ref{Schw} is the following.

\begin{proposition}
\label{Schw2} Let $F:[B(\cH)^n]_{1}\to [B(\cH)^m]_{1}$ be a free
holomorphic function such that $F'(0)$ is an isometry of $\CC^n$
into $\CC^m$. Then  $F$ is left invertible and $\Phi_L\circ
F=\text{\rm id}$, where $L:= [F'(0)^*]^t$.
\end{proposition}

\begin{proof}
The scalar representation of $F$ is  an analytic function
$f:\BB_n\to \BB_n$ defined by $f(\lambda):= F(\lambda)$, $\lambda\in
\BB_n$. It is easy to see that $F'(0)=f'(0)$, where
$f'(0)=\left[\frac{\partial f_i}{\partial
\lambda_j}(0)\right]_{m\times n}$. According to \cite{Ru} (see
Theorem 8.1.3), since $f'(0)$ is an isometry, we must have $f(0)=0$,
and therefore $F(0)=f(0)=0$. As in the proof of Theorem \ref{Schw},
we can use the fact that $F'(0)$ is an isometry to show that $
F(X)[F'(0)^*]^t=X$, $X\in [B(\cH)^n]_1$. The proof is complete.
\end{proof}

Here is another  extension of Schwarz  lemma, for bounded free
holomorphic functions.

\begin{theorem}
\label{Schw3} Let $F:[B(\cH)^n]_1\to [B(\cH)^m]_1$ be a free
holomorphic function, $a\in \BB_n$, and $b:=F(a)\in \BB_m$. Then
$$
\left\|\Psi_b(F(X))\right\|\leq \|\Psi_a(X)\|,\quad X\in
[B(\cH)^n]_1,
$$
where $\Psi_a$ and $\Psi_b$ are the corresponding free holomorphic
automorphisms.
\end{theorem}
\begin{proof}
Note that $\Psi_a\in Aut(B(\cH)^n]_1$ and $\Psi_b\in
Aut(B(\cH)^m]_1$. Due to Theorem \ref{compo} and Theorem
\ref{prop-cara}, the map $G:=\Psi_b\circ F\circ
\Psi_b:[B(\cH)^n]_1\to [B(\cH)^m]_1$ is  a free holomorphic function
with $G(0)=0$. Applying Theorem \ref{Schw1} part (ii), we deduce
that $\|(\Psi_b\circ F\circ \Psi_b)(Y)\|\leq \|Y\|$ for $Y\in
[B(\cH)^n]_1$. Setting $Y=\Psi_a(X)$ and using the fact that
$\Psi_a\circ \Psi_a=\text{\rm id}$, we complete the proof.
\end{proof}

\begin{corollary}\label{Gleason}
Let $a\in \BB_n$ and  and let $F:[B(\cH)^n]_1\to  B(\cH)$ be a free
holomorphic function. Then  $F-F(a)$  has a factorization of the
form
$$
F(X)-F(a)=\Psi_a(X) (H\circ\Psi_a)(X), \qquad X\in [B(\cH)^n]_1,
$$
 where   $H:[B(\cH)^n]_1\to
B(\cH)\otimes M_{n\times 1}$  is a free holomorphic function.
\end{corollary}
\begin{proof}
Note that $G:[B(\cH)^n]_1\to  B(\cH)$  defined by $G(X):=F(X)-F(a)$
is  a free holomorphic function with $G(a)=0$. Then $G\circ \Psi_a$
is a free holomorphic function with $(G\circ \Psi_a)(0)=0$.
 Therefore, it has a
representation $(G\circ \Psi_a)(X)=\sum_{k=1}^\infty
\sum_{|\alpha|=k} a_\alpha X_\alpha$, where the convergence is in
the operator norm topology. Hence, we deduce that
$$
(G\circ \Psi_a)(rS_1,\ldots, rS_n)=\sum_{k=1}^\infty
\sum_{|\alpha|=k} a_\alpha  r^{|\alpha|} S_\alpha=\sum_{i=1}^n rS_i
H_i(rS_1,\ldots, rS_n) \quad \text{ for any } \ r\in (0,1),
$$
where $H_i(rS_1,\ldots, rS_n):=\frac{1}{r} S_i^*(G\circ
\Psi_a)(rS_1,\ldots, rS_n)$ and the corresponding series
representation for $H_i(rS_1,\ldots, rS_n)$ is convergent in the
norm operator topology for any $r\in (0,1)$. Consequently, we deduce
that  each $H_i$ is a free holomorphic function on $[B(\cH)^n]_1$
and
$$
(G\circ \Psi_a)(X_1,\ldots, X_n)=\sum_{k=1}^\infty \sum_{|\alpha|=k}
a_\alpha   X_\alpha=\sum_{i=1}^n  X_i H_i(X_1,\ldots, X_n)
$$
  for any  $  X:=[X_1,\ldots, X_n]\in [B(\cH)^n]_1$.
  Setting $H:=\left[\begin{matrix} H_1\\\vdots\\
  H_n\end{matrix}\right]$, we obtain $(G\circ \Psi_a)(X)=X H(X)$.
  Since $\Psi_a\in Aut(B(\cH)^n_1)$ and $\Psi_a\circ \Psi_a
  =\text{\rm id}$, we deduce that
  $
  G (X)=\Psi_a(X) (H\circ\Psi_a)(X)$ for $ X\in [B(\cH)^n]_1,
$ which completes the proof.
\end{proof}

We remark that, setting $\Psi_a=(\psi_1,\ldots, \psi_n)$, one can
deduce that $F-F(a)$  is in the right ideal  generated by
$\psi_1,\ldots, \psi_n$ in  $Hol([B(\cH)^n]_1)$, the algebra of all
free holomorphic functions on $[B(\cH)^n]_1$.

\bigskip

       %

      \end{document}